\UseRawInputEncoding
\documentclass[12pt,reqno]{amsart}
\usepackage{fullpage}

\newtheorem{theorem}{Theorem}[section]
\newtheorem{lemma}[theorem]{Lemma}

\newtheorem{cor}[theorem]{Corollary}
\newtheorem{conj}[theorem]{Conjecture}

\newtheorem{prob}[theorem]{Problem}
\usepackage{graphicx}
\usepackage{color}
\usepackage[dvipsnames]{xcolor}
\usepackage{subfigure}
\usepackage{amssymb}
\usepackage{amsmath,mathrsfs}
\usepackage{colonequals}
\usepackage[backref=page]{hyperref}

\theoremstyle{definition}
\newtheorem{definition}[theorem]{Definition}

\newtheorem{remark}[theorem]{Remark}

\renewcommand{\subset}{\subseteq}
\renewcommand{\supset}{\supseteq}
\renewcommand{\epsilon}{\varepsilon}

\newcommand{\abs}[1]{\left|#1\right|}                   
\newcommand{\absf}[1]{|#1|}                             
\newcommand{\vnorm}[1]{\left\|#1\right\|}    
\newcommand{\vnormf}[1]{\|#1\|}                         
\newcommand{\vnormt}[1]{\left\|#1\right\|}    

\newcommand{\N}{\mathbb{N}}
\newcommand{\E}{\mathbb{E}}

\renewcommand{\d}{\mathrm{d}}

\newcommand{\R}{\mathbb{R}}

\newcommand{\embolden}[1]{\textbf {#1}}
\newcommand{\redA}{\Sigma}
\newcommand{\redb}{\partial^{*}}

\newcommand{\sdimn}{n}
\newcommand{\adimn}{n+1}

\newcommand{\Phiz}{\Phi\Big(\frac{\beta \langle \nu(z),x\rangle-\alpha}{\sqrt{1-\beta^{2}}}\Big)}
\newcommand{\Phizs}{\Phi\Big(\frac{\beta \langle \nu(z^{(s)}),x\rangle-\alpha}{\sqrt{1-\beta^{2}}}\Big)}
\newcommand{\Phizderiv}{\frac{\beta e^{-\frac{ [\beta\langle \nu(z),x\rangle-\alpha]^{2}}{2(1-\beta^{2})}}}{\sqrt{1-\beta^{2}}}\nu(z)}
\newcommand{\Phizdd}{\frac{\beta e^{-\frac{ [\beta\langle \nu(z),x\rangle-\alpha]^{2}}{2(1-\beta^{2})}}}{\sqrt{1-\beta^{2}}}}
\newcommand{\Phizdds}{\frac{\beta e^{-\frac{ [\beta\langle \nu(z^{(s)}),x\rangle-\alpha]^{2}}{2(1-\beta^{2})}}}{\sqrt{1-\beta^{2}}}}
\newcommand{\Phizddy}{\frac{\beta e^{-\frac{ [\beta\langle \nu(z),y\rangle-\alpha]^{2}}{2(1-\beta^{2})}}}{\sqrt{1-\beta^{2}}}}
\newcommand{\pen}{\Big[\int_{\Omega}\Phiz\gamma_{\adimn}(x)dx\Big]^{2}}   
\newcommand{\pens}{\Big[\int_{\Omega^{(s)}}\Phizs\gamma_{\adimn}(x)\,\d x\Big]^{2}}   
\newcommand{\pensnosq}{\int_{\Omega^{(s)}}\Phizs\gamma_{\adimn}(x)\,\d x}   
\newcommand{\pennosq}{\int_{\Omega}\Phiz\gamma_{\adimn}(x)\,\d x}   
\newcommand{\penf}{\Big[\int_{\R^{\adimn}}f(x)\Phiz\gamma_{\adimn}(x)dx\Big]^{2}}
\newcommand{\penfw}{\Big[\int_{\R^{\adimn}}f(x)\Phiz\gamma_{\adimn}(x)dx\Big]^{2}}
\newcommand{\SimplePhiz}{\Phi\Big(\frac{\rho \langle \frac{z}{\vnorm{z}},x\rangle}{\sqrt{1-\rho^{2}}}\Big)}
\newcommand{\SimplePhizs}{\Phi\Big(\frac{\rho \langle \frac{z^{(s)}}{\vnormf{z^{(s)}}},x\rangle}{\sqrt{1-\rho^{2}}}\Big)}
\newcommand{\SimplePhizy}{\Phi\Big(\frac{\rho \langle \frac{z}{\vnorm{z}},y\rangle}{\sqrt{1-\rho^{2}}}\Big)}
\newcommand{\SimplePhizderiv}{\frac{\rho e^{-\frac{ \left[\rho\left\langle \frac{z}{\vnorm{z}},x\right\rangle\right]^{2}}{2(1-\rho^{2})}}}{\sqrt{1-\rho^{2}}}\frac{z}{\vnorm{z}}}

\newcommand{\Simplepen}{\Big[\int_{\Omega}\SimplePhiz\gamma_{\adimn}(x)dx\Big]^{2}}   
\newcommand{\Simplepens}{\Big[\int_{\Omega+sv}\SimplePhizs\gamma_{\adimn}(x)\,\d x\Big]^{2}}   
\newcommand{\Simplepennosq}{\int_{\Omega}\SimplePhiz\gamma_{\adimn}(x)\,\d x}   
\newcommand{\Phizstilde}{\Phi\Big(\frac{\beta \langle \nu(\widetilde{z}^{(s)}),x\rangle-\alpha}{\sqrt{1-\beta^{2}}}\Big)}
\newcommand{\penstilde}{\Big[\int_{\widetilde{\Omega}^{(s)}}\Phizstilde\gamma_{2}(x)\,\d x\Big]^{2}}   

\begin{document}

\title{A Variational Proof of Robust Gaussian Noise Stability}

\author{Steven Heilman}
\address{Department of Mathematics, University of Southern California, Los Angeles, CA 90089-2532}
\email{stevenmheilman@gmail.com}

\date{\today}
\thanks{S. H. is Supported by NSF Grant CCF 1911216}

\begin{abstract}
Using the calculus of variations, we prove that a Euclidean set of fixed Gaussian measure that nearly maximizes Gaussian noise stability is close to a half space.  The main result proves a modification of a conjecture of Eldan from 2013: a robust Borell inequality that removes a logarithmic dependence on the distance of the set to a half space.  For sets of Gaussian measure $1/2$, we prove Eldan's 2013 conjecture.

The noise stability of a Euclidean set $A$ with correlation $\rho$ is the probability that $(X,Y)\in A\times A$, where $X,Y$ are standard Gaussian random vectors with correlation $\rho\in(-1,1)$.

Barchiesi, Brancolini and Julin proved that a Euclidean set of fixed Gaussian measure that nearly minimizes Gaussian surface area is close to a half space, using a variational ``penalty function'' method.  Our proof adapts their method to the more general setting of noise stability.

We also show that half spaces are the only sets that are stable (in the sense of second variation) for noise stability, generalizing a result of McGonagle and Ross for Gaussian surface area.
\end{abstract}

\maketitle
\setcounter{tocdepth}{1}
\tableofcontents
%
%
%
%
%

\section{Introduction}\label{secintro}

The Gaussian isoperimetric inequality says that a half space has the smallest Gaussian surface area among all Euclidean sets of fixed Gaussian volume \cite{sudakov74,borell75,ledoux94,bobkov97,burchard01,mcgonagle15}.  A robust version of this inequality says: if a Euclidean set nearly minimizes its Gaussian surface area (subject to a Gaussian volume constraint), then this set is close to a half space.  Such an inequality was proven in \cite{barchiesi16}, following \cite{mossel15,mossel12,eldan13}.

The noise stability of a measurable Euclidean set $A$ with correlation $\rho$ is the probability that $(X,Y)\in A\times A$, where $X,Y$ are standard Gaussian random vectors with correlation $\rho\in(-1,1)$.  Borell's inequality \cite{borell85} generalizes the Gaussian isoperimetric inequality in the following way: a half space has the largest noise stability among all Euclidean sets of fixed Gaussian volume, when $0<\rho<1$.  Letting $\rho\to1^{-}$ in Borell's inequality recovers the Gaussian isoperimetric inequality \cite{ledoux94}.

A robust version of Borell's inequality says: if a Euclidean set nearly maximizes noise stability (subject to a Gaussian volume constraint), then this set is close to a half space.  Robust versions of Borell's inequality were proven in \cite{mossel12,eldan13}.

The proof of the robust Gaussian isoperimetric inequality in \cite{barchiesi16} uses the calculus of variations to minimize the Gaussian surface area plus a ``penalty'' function.  The minimum of this quantity occurs at a half space, so that the ``penalty'' function quantifies how far an arbitrary set is from being a half space.  The main step of the proof computes the second derivative of infinitesimal translations of an optimal set that are Gaussian volume-preserving.

The proof methods of the more general robust Borell inequality \cite{mossel12,eldan13} are arguably ad hoc, so one might hope for a more elementary proof, along the lines of \cite{barchiesi16}.  Moreover, the proof methods of \cite{mossel12,eldan13} do not seem to generalize to inequalities for the noise stability of partitions of Euclidean space, as opposed to the calculus of variations arguments of e.g. \cite{heilman18b,heilman20d}.

In this paper, we demonstrate that the penalty function method of \cite{barchiesi16} can prove a robust Borell inequality.  As in \cite{barchiesi16}, the main step of the proof computes the second derivative of infinitesimal translations of the optimal set that are Gaussian volume-preserving.  This step works for all correlation parameters $0<\rho<1$, and it also seems to hold for an arbitrary number of sets that partition Euclidean space, though we avoid pursuing such a statement at this time.  This ``dimension reduction'' step of Section \ref{secdimred} shows that a maximizing set $\Omega\subset\R^{\adimn}$ is ``one-dimensional,'' i.e. after rotating $\Omega$, there exists $\Omega'\subset\R$ such that $\Omega=\Omega'\times\R^{\sdimn}$.  The final step of the proof, contained in Section \ref{secpro}, shows that in fact $\Omega'$ is an unbounded interval.

Besides their intrinsic interest, inequalities for noise stability have applications to social choice theory \cite{mossel10}, the Unique Games Conjecture \cite{khot07,mossel10,khot15}, to semidefinite programming algorithms such as MAX-CUT \cite{khot07,isaksson11}, to learning theory \cite{feldman12}, etc.  For some surveys on this and related topics, see  \cite{odonnell14b,khot10b,heilman20b}.

The combination of our robust Borell inequality, Theorem \ref{mainthm1} below, with previous works such as \cite{barchiesi16,heilman18b,heilman20d} essentially shows that one single argument can prove nearly every known inequality for sets or partitions that maximize noise stability, with respect to Gaussian volume constraints.  So, instead of having disparate arguments to prove these inequalities, one single calculus of variations argument has emerged, providing an aesthetically pleasing way to prove these optimal inequalities.

\subsection{More Formal Introduction}

For any $k\geq1$, we define the Gaussian density as
\begin{equation}\label{zero0.0}
\begin{aligned}
\gamma_{k}(x)&\colonequals (2\pi)^{-k/2}e^{-\vnormt{x}^{2}/2},\qquad
\langle x,y\rangle\colonequals\sum_{i=1}^{\adimn}x_{i}y_{i},\qquad
\vnormt{x}^{2}\colonequals\langle x,x\rangle,\\
&\qquad\forall\,x=(x_{1},\ldots,x_{\adimn}),y=(y_{1},\ldots,y_{\adimn})\in\R^{\adimn}.
\end{aligned}
\end{equation}

Let $f\colon\R^{\adimn}\to[0,1]$ be measurable and let $\rho\in(-1,1)$.  Define the \textbf{Ornstein-Uhlenbeck operator with correlation $\rho$} applied to $f$ by
\begin{equation}\label{oudef}
\begin{aligned}
T_{\rho}f(x)
&\colonequals\int_{\R^{\adimn}}f(x\rho+y\sqrt{1-\rho^{2}})\gamma_{\adimn}(y)\,\d y\\
&=(1-\rho^{2})^{-(\adimn)/2}(2\pi)^{-(\adimn)/2}\int_{\R^{\adimn}}f(y)e^{-\frac{\vnorm{y-\rho x}^{2}}{2(1-\rho^{2})}}\,\d y,
\qquad\forall x\in\R^{\adimn}.
\end{aligned}
\end{equation}
$T_{\rho}$ is a parametrization of the Ornstein-Uhlenbeck operator, which gives a fundamental solution of the (Gaussian) heat equation
\begin{equation}\label{oup}
\frac{d}{d\rho}T_{\rho}f(x)=\frac{1}{\rho}\Big(-\overline{\Delta} T_{\rho}f(x)+\langle x,\overline{\nabla}T_{\rho}f(x)\rangle\Big),\qquad\forall\,x\in\R^{\adimn}.
\end{equation}
Here $\overline{\Delta}\colonequals\sum_{i=1}^{\adimn}\partial^{2}/\partial x_{i}^{2}$ and $\overline{\nabla}$ is the usual gradient on $\R^{\adimn}$.  Our main object of study is the noise stability (or Gaussian heat content) of Euclidean sets.

\begin{definition}[\embolden{Noise Stability}]\label{noisedef}
Let $\Omega\subset\R^{\adimn}$ be measurable.  Let $\rho\in(-1,1)$.  We define the \textit{noise stability} of the set $\Omega$ with correlation $\rho$ to be
\begin{equation}\label{nseq}
\int_{\R^{\adimn}}1_{\Omega}(x)T_{\rho}1_{\Omega}(x)\gamma_{\adimn}(x)\,\d x
\stackrel{\eqref{oudef}}{=}(2\pi)^{-(\adimn)}(1-\rho^{2})^{-(\adimn)/2}\int_{\Omega}\int_{\Omega}e^{\frac{-\|x\|^{2}-\|y\|^{2}+2\rho\langle x,y\rangle}{2(1-\rho^{2})}}\,\d x\d y.
\end{equation}
Equivalently, if $X=(X_{1},\ldots,X_{\adimn}),Y=(Y_{1},\ldots,Y_{\adimn})\in\R^{\adimn}$ are $(\adimn)$-dimensional jointly Gaussian distributed random vectors with $\E X_{i}Y_{j}=\rho\cdot1_{(i=j)}$ $\forall$ $i,j\in\{1,\ldots,\adimn\}$,
$$\int_{\R^{\adimn}}1_{\Omega}(x)T_{\rho}1_{\Omega}(x)\gamma_{\adimn}(x)\,\d x=\mathbb{P}((X,Y)\in \Omega\times \Omega).$$
\end{definition}

Borell's inequality says that a half space maximizes noise stability among all measurable Euclidean sets of fixed Gaussian volume.

\begin{theorem}[\embolden{Borell's Inequality}, {\cite{borell85}}]\label{borellthm}
Let $\Omega\subset\R^{\adimn}$ be measurable.  Let $H\subset\R^{\adimn}$ be a half space with $\gamma_{\adimn}(\Omega)=\gamma_{\adimn}(H)$.  Let $0<\rho<1$.  Then
$$\int_{\R^{\adimn}}1_{\Omega}(x)T_{\rho}1_{\Omega}(x)\gamma_{\adimn}(x)\,\d x\leq \int_{\R^{\adimn}}1_{H}(x)T_{\rho}1_{H}(x)\gamma_{\adimn}(x)\,\d x.$$
If $-1<\rho<0$, then the inequality is reversed.
\end{theorem}

An improved, robust version of this inequality says that if the inequality in Theorem \ref{borellthm} is nearly an equality, then $\Omega$ is close to a half space.  Such an improved statement was proven in \cite{mossel12} and also \cite{eldan13}.

\begin{theorem}[\embolden{Robust Borell's Inequality}, {\cite[Theorem 1.4]{mossel12}}]\label{robustbor}
Let $0<\rho<1$ and fix $0<a<1$.  Then there exists $c(\rho,a)>0$ such that the following holds.  Let $\Omega\subset\R^{\adimn}$ be measurable with $\gamma_{\adimn}(\Omega)=a$.  Let $H\subset\R^{\adimn}$ be a half space with $\gamma_{\adimn}(H)=a$.  Define
$$\delta\colonequals \int_{\R^{\adimn}}1_{H}(x)T_{\rho}1_{H}(x)\gamma_{\adimn}(x)\,\d x-\int_{\R^{\adimn}}1_{\Omega}(x)T_{\rho}1_{\Omega}(x)\gamma_{\adimn}(x)\,\d x.$$
Then, after rotating $H$ if necessary,
$$\gamma_{\adimn}(\Omega\Delta H)\leq c(\rho,a)\cdot\delta^{\frac{1}{4}\frac{(1-\rho)(1-\rho^{2})}{1+3\rho}}.$$
Here $\Delta$ denotes the symmetric difference of sets.
\end{theorem}
A similar inequality is given in \cite{mossel12} when $-1<\rho<0$.  Since $\max_{0\leq\rho\leq1}\frac{(1-\rho)(1-\rho^{2})}{1+3\rho}=1$, the largest exponent of $\delta$ in Theorem \ref{robustbor} is $1/4$.

In order to remove the sub-optimal power dependence on $\delta$ in Theorem \ref{robustbor}, Eldan proved the following nearly matching upper and lower bounds on the ``noise stability deficit,'' $\delta$.

\begin{theorem}[\embolden{Robust Borell's Inequality}, {\cite[Theorem 2]{eldan13}}]\label{robustborv2}
Let $0<\rho<1$ and fix $0<a<1$.  Then there exists $c(a),c'(a)>0$ such that the following holds.  Let $\Omega\subset\R^{\adimn}$ be measurable with $\gamma_{\adimn}(\Omega)=a$.  Let $H\subset\R^{\adimn}$ be a half space with $\gamma_{\adimn}(H)=a$.  Define
$$\delta\colonequals \int_{\R^{\adimn}}1_{H}(x)T_{\rho}1_{H}(x)\gamma_{\adimn}(x)\,\d x-\int_{\R^{\adimn}}1_{\Omega}(x)T_{\rho}1_{\Omega}(x)\gamma_{\adimn}(x)\,\d x.$$
$$\eta\colonequals \Big\|\int_{H} x\gamma_{\adimn}(x)\d x\Big\|^{2}-\Big\|\int_{\Omega} x\gamma_{\adimn}(x)\,\d x\Big\|^{2}.$$
Then $\eta\geq0$ (with $\eta=0$ only when $\Omega$ is a half space) and, if $\eta<e^{-1/\rho}$, then
$$c(a)\frac{\eta}{\abs{\log\eta}}\sqrt{1-\rho}\leq \delta\leq c'(a)\frac{\eta}{\sqrt{1-\rho}}.$$
\end{theorem}

Theorem \ref{robustborv2} removes the power dependence on $\delta$ in Theorem \ref{robustbor}, yet Theorem \ref{robustborv2} still has an extra logarithmic factor in the lower bound on $\delta$, so the upper and lower bounds on $\delta$ do not match.  Unfortunately, as shown in \cite[Section 1.1]{eldan13}, this logarithmic factor cannot be fully removed, as one can see by considering sets of the form $\Omega=(-\infty,0]\cup[d,\infty)$ where $d\to\infty$.   Desiring a matching upper and lower bound on $\delta$, Eldan considered a different definition of $\eta$ that quantifies how far a set $\Omega$ is from a half space, resulting in the following conjecture.

Before stating the conjecture, denote $\Phi(t)\colonequals\gamma_{1}(-\infty,t]$, $\forall$ $t\in\R$, and denote $\nu(z)\colonequals z/\vnorm{z}$ if $z\neq0$ and $\nu(0)\colonequals(1,0,\ldots,0)\in\R^{\adimn}$.  Also define $\alpha\colonequals-\Phi^{-1}(a)$.

\begin{conj}[\embolden{Eldan's Robust Borell Conjecture}, {\cite[Conjecture 7]{eldan13}}]\label{eldconj}
Let $0<\rho<1$ and fix $0<a<1$.  Then there exists $c(a),c'(a)>0$ such that the following holds.  Let $\Omega\subset\R^{\adimn}$ be measurable with $\gamma_{\adimn}(\Omega)=a$.  Denote $z\colonequals\int_{\Omega}x\gamma_{\adimn}(x)\,\d x\in\R^{\adimn}$.  Let $H$ be a half space with $\gamma_{\adimn}(H)=a$ such that $\int_{H}x\gamma_{\adimn}(x)\,\d x$ is a positive multiple of $z$.  Define
$$\delta\colonequals \int_{\R^{\adimn}}1_{H}(x)T_{\rho}1_{H}(x)\gamma_{\adimn}(x)\,\d x-\int_{\R^{\adimn}}1_{\Omega}(x)T_{\rho}1_{\Omega}(x)\gamma_{\adimn}(x)\,\d x.$$
$$\eta\colonequals \int_{\R^{\adimn}}\Phi\Big(\frac{\rho \langle \nu(z),x\rangle-\alpha}{\sqrt{1-\rho^{2}}}\Big)(1_{H}(x)-1_{\Omega}(x))\gamma_{\adimn}(x)\d x.$$
Then $\eta\geq0$ (with $\eta=0$ only when $\Omega$ is a half space) and
$$c(a)\cdot \eta(1-\rho)\leq \delta\leq c'(a)\cdot\eta.$$
\end{conj}

Moreover, in \cite[Equation (129) and (130)]{eldan13}, the bound $\delta\leq 2\eta$ is shown to hold by the Cauchy-Schwarz inequality, so the only remaining part of Conjecture \ref{eldconj} is the lower bound on $\delta$.

In this paper we prove a modified version of the lower bound of Conjecture \ref{eldconj}.

\subsection{Our Contribution}

\begin{theorem}[\embolden{Robust Borell's Inequality, Small Correlation}]\label{mainthm1}
Let $0<a<1$.  Let $\Omega\subset\R^{\adimn}$ be measurable with $\gamma_{\adimn}(\Omega)=a$.  Let $H$ be a half space with $\gamma_{\adimn}(H)=a$.  Denote $z\colonequals\int_{\Omega}x\gamma_{\adimn}(x)\,\d x\in\R^{\adimn}$.  Let $H$ be a half space with $\gamma_{\adimn}(H)=a$ such that $\int_{H}x\gamma_{\adimn}(x)\,\d x$ is a positive multiple of $z$.  Assume that $z\neq0$ and $\vnorm{z}\geq z_{0}>0$.  Define $\nu(z)\colonequals z/\vnorm{z}$.  Let $0\leq\beta\leq\rho$.  Define
$$\delta\colonequals \int_{\R^{\adimn}}1_{H}(x)T_{\rho}1_{H}(x)\gamma_{\adimn}(x)\,\d x-\int_{\R^{\adimn}}1_{\Omega}(x)T_{\rho}1_{\Omega}(x)\gamma_{\adimn}(x)\,\d x.$$
$$\eta_{\beta}\colonequals \int_{\R^{\adimn}}\Phiz(1_{H}(x)-1_{\Omega}(x))\gamma_{\adimn}(x)\d x.$$
Then $\eta_{\beta}\geq0$ (with $\eta_{\beta}=0$ only when $\Omega$ is a half space) and
$$\Big[\frac{10^{-7}az_{0}^{2}e^{-\alpha^{2}\cdot\max\left(1,\frac{\beta}{\rho-\beta}\right)}}{(6+\abs{\alpha})^{2}}\rho(1-\rho)^{2}\beta(1-\beta^{2})a(1-a)
 \Big]
\eta_{\beta}\leq\delta\leq2\eta_{\rho}.$$

In the case $a=1/2$ (so that $\alpha=0$ since $\alpha=-\Phi^{-1}(a)$), we can let $\beta\colonequals\rho$ and deduce
$$10^{-9}\rho^{2}z_{0}^{2}(1-\rho^{2})^{2}
\eta_{\rho}\leq\delta\leq2\eta_{\rho}.$$
\end{theorem}
\begin{remark}
It seems possible to lower bound $\delta$ by the quantity
$$
\int_{\R^{\adimn}}\Phi\Big(\frac{\beta \langle \nu(z),x\rangle}{\sqrt{1-\beta^{2}}}\Big)(1_{H}(x)-1_{\Omega}(x))\gamma_{\adimn}(x)\d x
$$
by changing some minor details in the proof of Theorem \ref{mainthm1}, though we will not pursue such an inequality, since it seems further from Conjecture \ref{eldconj} than Theorem \ref{mainthm1}.
\end{remark}
\begin{remark}
It is unclear whether or not the proof of Theorem \ref{mainthm1} can achieve an exponent $1$ in place of $4$ in the term $(1-\rho^{2})^{4}$.  It seems difficult to avoid an exponent smaller than $2$ here, due to a use of the Cauchy-Schwarz inequality in Lemma \ref{techlem}.
\end{remark}
\begin{remark}
The restriction that $\vnorm{z}\geq z_{0}>0$ might appear a bit unnatural.  Note however that a similar restriction appears in Eldan's Theorem \ref{robustborv2}, i.e. $\eta<e^{-1/\rho}$.  Note also that, in the case $z=0$, it is easily seen that $\Omega$ with $z=\int_{\Omega}\gamma_{\adimn}(x)\,\d x=0$ cannot maximize noise stability, since if it did, e.g. the main result of \cite{heilman20d} implies that every direction $v\in\R^{\adimn}$ of translation of $\Omega$ has nonpositive second variation for the noise stability, implying that $\langle v,N(x)\rangle=0$ $\forall$ $v\in\R^{\adimn}$, and $\forall$ $x\in\partial\Omega$, where $N(x)$ is the unit exterior pointing normal vector to $\partial\Omega$.  No such unit vector $N(x)$ can satisfy $\langle v,N(x)\rangle=0$ $\forall$ $v\in\R^{\adimn}$, so we have found a contradiction.  That is, $\Omega$ with $z=0$ cannot maximize noise stability.
\end{remark}

\subsection{Discussion of Proof Methods}

The proof of Theorem \ref{robustbor} in \cite{mossel12} proceeds by a careful analysis of the rate of change of the auxiliary function
$$\E J(T_{\rho}1_{\Omega}(X),T_{\rho}1_{\Omega}(Y))$$
with respect to $\rho$, where $X=(X_{1},\ldots,X_{\adimn})$ and $Y=(Y_{1},\ldots,Y_{\adimn})$ are each standard Gaussian random variables that are correlated so that $\E X_{i}Y_{j}=\rho 1_{\{i=j\}}$ for all $1\leq i,j\leq \adimn$.  Here $J(a,b)\colonequals\int_{-\infty}^{\Phi^{-1}(a)}T_{\rho}1_{[-\infty,\Phi^{-1}(a))}(x)\gamma_{1}(x)\,\d x$ for all $a,b\in[0,1]$ where $\Phi(t)\colonequals\gamma_{1}(-\infty,t]$ for all $t\in\R$.

The proof of Theorem \ref{robustborv2} in \cite{eldan13} uses stochastic calculus to estimate the rate of change of the quantity
$$ \E \Big\|\int_{\frac{\Omega-\rho X}{\sqrt{1-\rho}}} x\gamma_{\adimn}(x)\,\d x\Big\|^{2}$$
with respect to $\rho$.

In the present paper, instead of analyzing the rate of change of an auxiliary quantity, we simply consider how the noise stability itself changes when the set $\Omega$ is translated.  This strategy first appeared in \cite{colding11} in the context of mean curvature flows, then appearing for inequalities for the Gaussian surface area in \cite{mcgonagle15,barchiesi16,barchiesi18,milman18a,heilman18,milman18b,heilman18b}.  The strategy of analyzing the translations of noise stability first appeared in \cite{heilman20d}.  In the present work, we show that the strategy of \cite{heilman20d} still works when we add a ``penalty term'' to the noise stability.  That is, we adapt the argument of \cite{barchiesi16,heilman18b} from the setting of Gaussian surface area to the more general setting of Gaussian noise stability.

As mentioned earlier, the main strategy of \cite{barchiesi16} involves minimizing the Gaussian surface area, plus a ``penalty'' term.  So, the main problem of interested in \cite{barchiesi16} is to minimize
$$\int_{\partial\Omega}\gamma_{\adimn}(x)\,\d x+\epsilon\vnorm{\int_{\Omega}x\gamma_{\adimn}(x)\,\d x}^{2},$$
over all measurable $\Omega\subset\R^{\adimn}$ with $\gamma_{\adimn}(\Omega)$ fixed, for sufficiently small $\epsilon>0$.  It is then shown that the minimal such set is a half space $H$, so that, for any $\Omega\subset\R^{\adimn}$, if $H\subset\R^{\adimn}$ is a half space with $\gamma_{\adimn}(\Omega)=\gamma_{\adimn}(H)$, then for appropriate $\epsilon>0$, we have
$$\int_{H}\gamma_{\adimn}(x)\,\d x+\epsilon\vnorm{\int_{H}x\gamma_{\adimn}(x)\,\d x}^{2}
\leq \int_{\partial\Omega}\gamma_{\adimn}(x)\,\d x+\epsilon\vnorm{\int_{\Omega}x\gamma_{\adimn}(x)\,\d x}^{2}.$$
Rearranging this inequality then gives a robust Gaussian isoperimetric inequality:
\begin{equation}\label{robiso}
\int_{\partial\Omega}\gamma_{\adimn}(x)\,\d x-\int_{H}\gamma_{\adimn}(x)\,\d x
\geq\epsilon\Big(\vnorm{\int_{H}x\gamma_{\adimn}(x)\,\d x}^{2}-\vnorm{\int_{\Omega}x\gamma_{\adimn}(x)\,\d x}^{2}\Big).
\end{equation}
Since half spaces maximize the quantity $\vnorm{\int_{\Omega}x\gamma_{\adimn}(x)\,\d x}^{2}$, the right side of \eqref{robiso} is nonnegative, and it quantifies how close the arbitrary set $\Omega$ is from being a half space.

It seems natural to use this same method directly for our purposes, i.e. by trying to show that half spaces maximize the quantity
\begin{equation}\label{badone}
\int_{\Omega}T_{\rho}1_{\Omega}(x)\gamma_{\adimn}(x)\,\d x-\epsilon\vnorm{\int_{\Omega}x\gamma_{\adimn}(x)\,\d x}^{2},
\end{equation}
over all measurable $\Omega\subset\R^{\adimn}$ with $\gamma_{\adimn}(\Omega)$ fixed, for sufficiently small $\epsilon>0$.  Unfortunately, this approach fails for the same reason as mentioned in \cite[Section 1.1]{eldan13}.  A half space cannot maximize \eqref{badone}, since sets of the form $\Omega=(-\infty,0]\cup[d,\infty)$ with $d$ large will always have a larger value of \eqref{badone} (e.g. for sets of Gaussian measure close to $1/2$).  The issue here is that the penalty function of \eqref{badone} grows linearly near infinity, with respect to the Gaussian density.  So, in some sense the \textit{only} way for the approach of \cite{barchiesi16} to work for noise stability is to use a penalty function such as $\Phi$ that is bounded near infinity.

The following is therefore our main problem of interest.

\begin{prob}[\embolden{Noise Stability, with Penalty}]\label{prob2}
Let $0<\epsilon<\rho<1$.  Let $0<a<1$.  Find a measurable set $\Omega\subset\R^{\adimn}$ with $\gamma_{\adimn}(\Omega)=a$ that maximizes
$$\int_{\R^{\adimn}}1_{\Omega}(x)T_{\rho}1_{\Omega}(x)\gamma_{\adimn}(x)\,\d x
-\epsilon\pen,$$
where $z\colonequals\int_{\Omega}x\gamma_{\adimn}(x)\,\d x\in\R^{\adimn}$, $\nu(z)\colonequals\frac{z}{\vnorm{z}}$ if $z\neq0$, and $\nu(0)\colonequals(1,0,\ldots,0)$, and $\alpha\colonequals-\Phi^{-1}(a)$, $\Phi(t)\colonequals\gamma_{1}(-\infty,t]$, $\forall$ $t\in\R$.
\end{prob}

The case $\epsilon=0$ corresponds to the problem of maximizing noise stability.  Intuitively, when $\epsilon>0$ is sufficiently small, the ``penalty'' term on the right should not affect the noise stability term very much.

As in \cite{barchiesi16}, for technical reasons, it is more convenient to replace the volume constraint by a volume-penalization term.  That is, we change Problem \ref{prob2} to the following

\begin{prob}[\embolden{Maximizing Noise Stability, with Penalties}]\label{prob2z}
Let $0<\epsilon<\rho<1$.  Let $0<a<1$.  Find a measurable set $\Omega\subset\R^{\adimn}$ that maximizes
$$\int_{\Omega}T_{\rho}1_{\Omega}(x)\gamma_{\adimn}(x)\,\d x
-\epsilon\pen
-2(1+\abs{\alpha})\abs{\gamma_{\adimn}(\Omega)-a},$$
where $z\colonequals\int_{\Omega}x\gamma_{\adimn}(x)\,\d x\in\R^{\adimn}$, $\nu(z)\colonequals\frac{z}{\vnorm{z}}$ if $z\neq0$, and $\nu(0)\colonequals(1,0,\ldots,0)$, and $\alpha\colonequals-\Phi^{-1}(a)$, $\Phi(t)\colonequals\gamma_{1}(-\infty,t]$, $\forall$ $t\in\R$.
\end{prob}

The Main Theorem \ref{mainthm1} is a corollary of the following ``Dimension Reduction'' Theorem.  Theorem \ref{mainthm2} says that the optimizer of Problem \ref{prob2} must be one-dimensional.

\begin{theorem}[\embolden{Dimension Reduction}]\label{mainthm2}
Let $0<\rho<1$, let $z_{0}>0$ and let
$$0<\epsilon<\frac{(1-\rho)^{2}z_{0}^{2}}{\rho10e^{\alpha^{2}\cdot\max\left(0,\frac{\beta}{\rho-\beta}-1\right)}}.$$
Let $\Omega\subset\R^{\adimn}$ maximize Problem \ref{prob2}.  Assume that $\vnorm{z}\geq z_{0}$.  Then, after rotating $\Omega$ and applying Lebesgue measure zero changes to $\Omega$, there exist measurable $\Omega'\subset\R$ such that,
$$\Omega=\Omega'\times\R^{\sdimn}.$$  
\end{theorem}

\subsection{Technical Issues in Theorem \ref{mainthm1}}

There are a few technical issues in the proof of Theorem \ref{mainthm1}.

First, in order to get existence and regularity of $\Omega$ maximizing Problem \ref{prob2} in Lemmas \ref{existlem} and \ref{reglem}, we need to show that the quadratic quantity in Problem \ref{prob2} is a positive semidefinite function of $\Omega\subset\R^{\adimn}$.  To show this, we write the $\Phi$ function as an exponentially decaying sum of Hermite polynomials in Section \ref{secexreg}.  This argument is fairly elementary though technical, and it requires $\epsilon>0$ to be sufficiently small.  (If $\epsilon>0$ is large in Problem \ref{prob2}, then the quadratic quantity is not positive semidefinite.)

Second, we need to show that the second variation of the quadratic quantity in Problem \ref{prob2} is a positive semidefinite function of $f\colon\partial\Omega\to\R$.  To show this, we write the $\Phi$ function as an exponentially decaying sum of Hermite polynomials in Section \ref{secexreg}.  This argument is more involved than that of Section \ref{secexreg}, but the idea is the same.  We write we write $\Phi$ and its first and second derivatives as exponentially decaying sum of Hermite polynomials in Section \ref{secexreg}.  Once again, $\epsilon>0$ must be sufficiently small in order to get a positive semidefinite function.  Most of the effort in proving this positive semidefinite property is contained in Lemma \ref{techlem}.  With results such as Lemma \ref{techlem}, we can then prove Theorem \ref{mainthm2} by adapting a now-standard argument appearing e.g. in \cite{colding12a}, \cite{mcgonagle15} or \cite{heilman20d}.

Finally, using Theorem \ref{mainthm2}, it suffices to solve the one-dimensional case of Problem \ref{prob2}.  In this case, we can explicitly write out various terms that appear in the second variation of Problem \ref{prob2} (when the set is perturbed in the constant normal direction), and we control these terms with the isoperimetric deficit $\int_{\partial\Omega}\gamma_{1}(x)\,\d x-\int_{\partial H}\gamma_{1}(x)\,\d x$, where $H$ is a half space satisfying $\gamma_{\adimn}(H)=\gamma_{\adimn}(\Omega)$.  Such inequalities are proven in Section \ref{lastlems} using technical but somewhat elementary arguments.

All of these technical aspects have lengthened the paper.  Besides its length, we believe the argument of Theorem \ref{mainthm1} is conceptually simple, since it requires simply considering infinitesimal translations of the set $\Omega\subset\R^{\adimn}$ to reduce to the one-dimensional case of Problem \ref{prob2} (Theorem \ref{mainthm2}), and we then only need to perturb the optimal set $\Omega\subset\R$ in the constant normal direction to deduce Theorem \ref{mainthm1}.  Put another way, the details require finding how small $\epsilon>0$ needs to be in order to show that half spaces maximize Problem \ref{prob2} (or equivalently, how large the constant in the lower bound of $\delta$ needs to be in Theorem \ref{mainthm1}).

\subsection{Outline of the Proof of the Main Theorem}

In this section we outline the proof of Theorem \ref{mainthm1} in the case that $\gamma_{\adimn}(\Omega)=a=1/2$.  The proof loosely follows that of a corresponding statement \cite{mcgonagle15,barchiesi16} for the Gaussian surface area (which was then adapted to multiple sets in \cite{milman18a,milman18b,heilman18}), with a few key differences.  For didactic purposes, we will postpone a discussion of technical difficulties (such as existence and regularity of a maximizer) to Section \ref{secpre}.

Fix $a=1/2$, so that $\alpha=0$.  Let $\epsilon>0$.  Suppose there exists measurable $\Omega\subset\R^{\adimn}$ maximizing
$$\int_{\R^{\adimn}}1_{\Omega}(x)T_{\rho}1_{\Omega}(x)\gamma_{\adimn}(x)dx-\epsilon\Simplepen,$$
subject to the constraint $\gamma_{\adimn}(\Omega)=a$.  Define $z\colonequals\int_{\Omega}x\gamma_{\adimn}(x)\,\d x\in\R^{\adimn}$.  A first variation argument (Lemma \ref{latelemma3} below) implies that $\Sigma\colonequals\partial\Omega$ is a level set of the Ornstein-Uhlenbeck operator applied to $1_{\Omega}$, plus another term.  That is, there exists $c\in\R$ such that
\begin{equation}\label{zero1}
\Sigma=\Big\{x\in\R^{\adimn}\colon T_{\rho}1_{\Omega}(x)-\epsilon\SimplePhiz=c\Big\}.
\end{equation}
Since $\Sigma$ is a level set, a vector perpendicular to the level set is also perpendicular to $\Sigma$.  Denoting $N(x)\in\R^{\adimn}$ as the unit length exterior pointing normal vector to $x\in\Sigma$, \eqref{zero1} implies that
\begin{equation}\label{zero1.7}
\overline{\nabla}T_{\rho}1_{\Omega}(x)-\epsilon\SimplePhizderiv = -N(x)\Big\|\overline{\nabla}T_{\rho}1_{\Omega}(x)-\epsilon \SimplePhizderiv\Big\|.
\end{equation}
(It is not obvious that there must be a negative sign here, but it follows from examining the second variation.)  We now observe how the noise stability of $\Omega$ changes as the set is translated infinitesimally.  Fix $v\in\R^{\adimn}$, and consider the variation of $\Omega$ induced by the constant vector field $v$.  Denote $f(x)\colonequals \langle v,N(x)\rangle$ for all $x\in\Sigma$.  Then define
\begin{equation}\label{introsdef}
S(f)(x)\colonequals (1-\rho^{2})^{-(\adimn)/2}(2\pi)^{-(\adimn)/2}\int_{\Sigma}f(y)e^{-\frac{\vnorm{y-\rho x}^{2}}{2(1-\rho^{2})}}\,\d y,\qquad\forall\,x\in\Sigma.
\end{equation}
A second variation argument (Lemma \ref{lemma7p} below) implies that, if $f$ is Gaussian volume-preserving, i.e. if $\int_{\Sigma}f(x)\gamma_{\adimn}(x)\,\d x=0$, and if $z^{(s)}\colonequals\int_{\Omega+sv}x\gamma_{\adimn}(x)\,\d x\in\R^{\adimn}$ $\forall$ $s\in(-1,1)$, then
\begin{equation}\label{zero4.5}
\begin{aligned}
&\frac{1}{2}\frac{\d^{2}}{\d s^{2}}\Big|_{s=0}\Big[\int_{\R^{\adimn}}1_{\Omega+sv}(x)T_{\rho}1_{\Omega+sv}(x)\gamma_{\adimn}(x)\,\d x
-\epsilon\Simplepens\Big]\\
&\,\,=\int_{\Sigma}\Big(S(f)(x)-\epsilon\SimplePhiz\int_{\Sigma}f(y)\SimplePhizy\gamma_{\adimn}(y)\,\d y\\
&\qquad\qquad\qquad\qquad-\Big\|\overline{\nabla}T_{\rho}1_{\Omega}(x)-\epsilon\SimplePhizderiv\Big\| f(x)\Big)f(x)\gamma_{\adimn}(x)\,\d x.
\end{aligned}
\end{equation}
(Technically, we use a slightly different vector field to perturb $\Omega$ that is constant on $\Sigma$ but not constant in a neighborhood of $\Sigma$, but we will not dwell on this point presently.) Somewhat unexpectedly, the function $f(x)=\langle v,N(x)\rangle$ is almost an eigenfunction of the operator $S$ (by Lemma \ref{treig}), in the sense that
\begin{equation}\label{zero5}
\begin{aligned}
S(f)(x)&=
f(x)\frac{1}{\rho}\Big\|\overline{\nabla} T_{\rho}1_{\Omega}(x)-\epsilon a_{0}\Big(\SimplePhizderiv+\zeta\Big)\Big\|\\
&\qquad\qquad\qquad\qquad\qquad\qquad\qquad-\frac{\epsilon a_{0}}{\rho}\Big\langle v, \SimplePhizderiv+\zeta\Big\rangle,\qquad\forall\,x\in\Sigma,
\end{aligned}
\end{equation}
where
$$
 a_{0}\colonequals\Simplepennosq,
\qquad \zeta\colonequals\int_{\Omega}\SimplePhizderiv\frac{y-\langle \frac{z}{\vnorm{z}},y\rangle \frac{z}{\vnorm{z}}}{\vnorm{z}}\gamma_{\adimn}(y)\,\d y.
$$
Equation \eqref{zero5} is the \textit{key fact} used in the proof of the main theorem, Theorem \ref{mainthm1}.  Equation \eqref{zero5} follows from \eqref{zero1.7} and the divergence theorem (see Lemma \ref{treig} for a proof of \eqref{zero5}.)  Plugging \eqref{zero5} into \eqref{zero4.5}, and using also Mehler's formula \eqref{Height} (see Lemma \ref{svartran}),
\begin{equation}\label{zero3}
\begin{aligned}
&\int_{\Sigma}\langle v,N(x)\rangle\gamma_{\adimn}(x)\,\d x=0\quad\Longrightarrow\\
&\qquad\frac{1}{2}\frac{\d^{2}}{\d s^{2}}\Big|_{s=0}\Big[\int_{\R^{\adimn}}1_{\Omega+sv}(x)T_{\rho}1_{\Omega+sv}(x)\gamma_{\adimn}(x)\,\d x
-\epsilon\Simplepens\Big]\\
&\qquad\qquad\geq \Big(\frac{1-\epsilon\rho\frac{10e^{\alpha^{2}\cdot\max\left(0,\frac{\beta}{\rho-\beta}-1\right)}}{(1-\rho)\vnorm{z}^{2}}}{\rho}-1\Big)
\int_{\Sigma}\langle v,N(x)\rangle^{2}\vnorm{\overline{\nabla}T_{\rho}1_{\Omega}(x)}\gamma_{\adimn}(x)\,\d x.
\end{aligned}
\end{equation}
If $\vnorm{z}\geq z_{0}>0$ and if $0<\epsilon<\frac{(1-\rho)^{2}z_{0}^{2}}{\rho10e^{\alpha^{2}\cdot\max\left(0,\frac{\beta}{\rho-\beta}-1\right)}}$, then the last term in \eqref{zero3} is nonnegative. The set
$$V\colonequals\Big\{v\in\R^{\adimn}\colon \int_{\Sigma}\langle v,N(x)\rangle\gamma_{\adimn}(x)\,\d x=0\Big\}$$
has dimension at least $\sdimn$, by the rank-nullity theorem.  Since $\Omega$ maximizes noise stability, the quantity on the right of \eqref{zero3} must be non-positive for all $v\in V$, implying that $f=0$ on $\Sigma$ (except possibly on a set of measure zero on $\Sigma$).  (One can show that $\vnorm{\overline{\nabla}T_{\rho}1_{\Omega}(x)}>0$ for all $x\in\Sigma$.  See Lemma \ref{lemma7p}.)  That is, for all $v\in V$, $\langle v,N(x)\rangle=0$ for all $x\in\Sigma$ (except possibly on a set of measure zero on $\Sigma$).  Since $V$ has dimension at least $\sdimn$, there exists a measurable discrete set $\Omega'\subset\R$ such that $\Omega=\Omega'\times\R^{\sdimn}$ after rotating $\Omega$, concluding the proof of Theorem \ref{mainthm2}.


Theorem \ref{mainthm1} then follows by showing that $\Omega'$ is actually an unbounded interval.  This requires an additional argument that perturbs the optimal set $\Omega'$ in the constant normal direction.  Since this perturbation is not necessarily Gaussian volume preserving, we suppose that $\Omega=\Omega'\times\R$ (which we can by Theorem \ref{mainthm2}), and we perturb $\Omega$ in the normal direction times $x_{2}$.  This argument is sufficient to conclude the proof.

\subsection{Local Stability of Half Spaces}

\begin{definition}
A set $\Omega\subset\R^{\adimn}$ is called \textbf{locally stable} for noise stability for any family of sets $\{\Omega_{s}\}_{s\in(-1,1)}$ with $\Omega_{0}=\Omega$ such that
$$\frac{\d}{\d s}\Big|_{s=0}\gamma_{\adimn}(\Omega_{s})=0,$$
we have
$$\frac{\d^{2}}{\d s^{2}}\Big|_{s=0}\int_{\R^{\adimn}}1_{\Omega^{(s)}}(x)T_{\rho}1_{\Omega^{(s)}}(x)\gamma_{\adimn}(x)\,\d x\leq0.$$
\end{definition}

The following result generalizes the main result of \cite{mcgonagle15} from the setting of Gaussian surface area to the setting of noise stability.

\begin{cor}\label{rk0}
Half spaces are the only locally stable sets for noise stability.
\end{cor}
\begin{proof}
If $\Omega\subset\R^{\adimn}$ is not a half space, then we have by definition \eqref{introsdef} of $S$ the following strict inequality
\begin{equation}\label{dbst}
\int_{\redA}S(\abs{\langle v,N\rangle})(x)\cdot \abs{\langle v,N(x)\rangle} \gamma_{\adimn}(x)\,\d x
>\int_{\redA}S(\langle v,N\rangle)(x)\cdot\langle v,N(x)\rangle \gamma_{\adimn}(x)\,\d x.
\end{equation}
In the case that $\Omega=\Omega'\times\R$ (which we can assume by the Dimension Reduction Theorem \ref{mainthm2}), consider the function $g(x)\colonequals x_{\adimn}\abs{\langle v,N(x)\rangle}$ defined on $\redA$.  Then
\begin{flalign*}
&\int_{\redA}\Big(S(g)(x)-\vnormf{\overline{\nabla} T_{\rho}1_{\Omega}(x)} g(x)\Big) g(x)\gamma_{\adimn}(x)\,\d x\\
&=\int_{\redA}\Big(\rho S(\abs{\langle v,N\rangle})(x)\cdot \abs{\langle v,N(x)\rangle} -\vnormf{\overline{\nabla} T_{\rho}1_{\Omega}(x)} \abs{\langle v,N(x)\rangle}\Big)
\abs{\langle v,N(x)\rangle}\gamma_{\adimn}(x)\,\d x\\
&\stackrel{\eqref{dbst}}{>}\int_{\redA}\Big(\rho S(\langle v,N\rangle)(x) -\vnormf{\overline{\nabla} T_{\rho}1_{\Omega}(x)} \langle v,N(x)\rangle\Big)
\langle v,N(x)\rangle\gamma_{\adimn}(x)\,\d x
\stackrel{\eqref{zero5}\wedge(\epsilon=0)}{=}0.
\end{flalign*}
So, $\int_{\redA} g(x)\gamma_{\adimn}(x)\,\d x=0$ while the corresponding variation of $g$ satisfies (by Lemma \ref{lemma7p} below) $$\frac{\d^{2}}{\d s^{2}}\Big|_{s=0}\int_{\R^{\adimn}}1_{\Omega^{(s)}}(x)T_{\rho}1_{\Omega^{(s)}}(x)\gamma_{\adimn}(x)\,\d x>0.$$
That is, the half space is the only stable maximum of noise stability.
\end{proof}
The above proof does not seem to generalize to the case $\epsilon>0$ needed to prove Theorem \ref{mainthm1}.  There is a slightly longer proof of Corollary \ref{rk0} that does in fact generalize to prove Theorem \ref{mainthm1}, so we present this proof below for didactic purposes.
\begin{proof}[Second proof of Corollary \ref{rk0}]
We may assume that $\Omega=\Omega'\times\R$ for some $\Omega'\subset\R^{\sdimn}$ by the Dimension Reduction Theorem \ref{mainthm2}).  Consider the function $g(x)\colonequals x_{\adimn}$ defined on $\redA$.  Then $\int_{\redA} g(x)\gamma_{\adimn}(x)\,\d x=0$ and by Lemma \ref{lemma7p} below
\begin{flalign*}
&\frac{\d^{2}}{\d s^{2}}\Big|_{s=0}\int_{\R^{\adimn}}1_{\Omega^{(s)}}(x)T_{\rho}1_{\Omega^{(s)}}(x)\gamma_{\adimn}(x)\,\d x\\
&\qquad=\int_{\redA}\Big(S(g)(x)-\vnormf{\overline{\nabla} T_{\rho}1_{\Omega}(x)} g(x)\Big) g(x)\gamma_{\adimn}(x)\,\d x\\
&\qquad=\int_{\redA}\Big(\rho S(1)(x) -\vnormf{\overline{\nabla} T_{\rho}1_{\Omega}(x)}\Big)\gamma_{\adimn}(x)\,\d x.
\end{flalign*}
Let $H\subset\R^{\adimn}$ be a half space with $\gamma_{\adimn}(\Omega)=\gamma_{\adimn}(H)$.  By Lemma \ref{finallem1} below, we have
\begin{flalign*}
&\frac{\d^{2}}{\d s^{2}}\Big|_{s=0}\int_{\R^{\adimn}}1_{\Omega^{(s)}}(x)T_{\rho}1_{\Omega^{(s)}}(x)\gamma_{\adimn}(x)\,\d x\\
&\qquad\qquad\qquad\qquad\qquad\geq\frac{1}{80}\rho(1-\rho) \gamma_{\adimn}(\Omega)(1-\gamma_{\adimn}(\Omega))\Big(\int_{\Sigma}\gamma_{1}(x)\,\d x-\int_{H}\gamma_{1}(x)\,\d x\Big).
\end{flalign*}
So, this quantity is positive, unless $\Sigma$ is a half space, by the usual Gaussian isoperimetric inequality.  That is, the half space is the only stable maximum of noise stability.
\end{proof}
%

\subsection{Remarks on More than Two Sets}

We believe that Theorem \ref{mainthm2} also holds for partitions of Euclidean space optimizing noise stability.  However, this extra generality seemed to make the paper more difficult to write and to comprehend, so we avoided this generality.

\section{Summary of Notation}
Here is a summary of notation used throughout the paper.

\begin{itemize}
\item $T_{\rho}$ denotes the Ornstein-Uhlenbeck operator with correlation $\rho\in(-1,1)$.  (See \eqref{oudef}.)
\item $\Omega\subset\R^{\adimn}$ denotes a measurable set.
\item $\Sigma\colonequals\redb\Omega$ denotes the reduced boundary of $\Omega\subset\R^{\adimn}$.  (See Definition \ref{rbdef}.)
\item $\overline{\nabla}$ denotes the gradient on $\R^{\adimn}$.
\item $\overline{\Delta}\colonequals\sum_{i=1}^{\adimn}\partial^{2}/\partial x_{i}^{2}$ denotes the Laplacian on $\R^{\adimn}$.
\item $N(x)$ is the exterior pointing unit normal vector to $x\in\Sigma$.
\item $z\colonequals\int_{\Omega}x\gamma_{\adimn}(x)\,\d x\in\R^{\adimn}$ denotes the Gaussian barycenter of $\Omega$.
\item $\Phi(t)\colonequals\int_{-\infty}^{t}e^{-x^{2}/2}\,\d x/\sqrt{2\pi}$, $\forall$, $t\in\R$, denotes the Gaussian cumulative distribution function.
\item $\nu(x)\colonequals\begin{cases}\frac{x}{\vnorm{x}},\qquad\,\mbox{if}\,\, x\in\R^{\adimn}\setminus\{0\}\\  (1,0,\ldots,0),\qquad\,\mbox{if}\, x=0\end{cases}$.
\item $\vnorm{x}\colonequals(x_{1}^{2}+\cdots+x_{\adimn}^{2})^{1/2}$ for any $x=(x_{1},\ldots,x_{\adimn})\in\R^{\adimn}$.
\item $a\in(0,1)$ satisfies $\gamma_{\adimn}(\Omega)=a$.
\item $\alpha\in\R$ satisfies $\int_{\alpha}^{\infty}\gamma_{1}(x)\,\d x=\alpha$, i.e. $\alpha\colonequals-\Phi^{-1}(a)$.
\end{itemize}
Throughout the paper, unless otherwise stated, we define $G\colon\R^{\adimn}\times\R^{\adimn}\to\R$ to be the following function.  For all $x,y\in\R^{\adimn}$, $\forall$ $\rho\in(-1,1)$, define
\begin{equation}\label{gdef}
\begin{aligned}
G(x,y)&=(1-\rho^{2})^{-(\adimn)/2}(2\pi)^{-(\adimn)}e^{\frac{-\|x\|^{2}-\|y\|^{2}+2\rho\langle x,y\rangle}{2(1-\rho^{2})}}\\
&=(1-\rho^{2})^{-(\adimn)/2}\gamma_{\adimn}(x)\gamma_{\adimn}(y)e^{\frac{-\rho^{2}(\|x\|^{2}+\|y\|^{2})+2\rho\langle x,y\rangle}{2(1-\rho^{2})}}\\
&=(1-\rho^{2})^{-(\adimn)/2}(2\pi)^{-(\adimn)/2}\gamma_{\adimn}(x)e^{\frac{-\vnorm{y-\rho x}^{2}}{2(1-\rho^{2})}}\\
&\stackrel{\eqref{Height}}{=}\gamma_{\adimn}(x)\gamma_{\adimn}(y)
\sum_{k=0}^{\infty}\rho^{k}\sum_{\substack{\ell\in\N^{\adimn}\colon\\ \|\ell\|_{1}=k}}h_{\ell}(x)h_{\ell}(y)\ell!.
\end{aligned}
\end{equation}

\section{Existence and Regularity}\label{secexreg}

\subsection{Mehler's Formula}

We demonstrate here that a shifted and dilated Gaussian density on the real line can be written as an exponentially decaying sum of Hermit polynomials.  The derivation uses various elementary formulas.

Let $x\in\R$.  Recall that the Hermite polynomials $h_{0}(x),h_{1}(x),\ldots$ are defined by the following generating function.  For all $0<\lambda<1$, and for all $x\in\R$, we have

\begin{flalign}
\sum_{\ell=0}^{\infty}\lambda^{\ell}h_{\ell}(x)
&=e^{\lambda x-\lambda^{2}/2}
=\sum_{p=0}^{\infty}\frac{x^{p}}{p!}\lambda^{p}\sum_{q=0}^{\infty}\frac{(-1)^{q}(\lambda)^{2q}(1/2)^{q}}{q!}\label{Hzero}\\
&=\sum_{\ell=0}^{\infty}\lambda^{\ell}\sum_{k=0}^{\lfloor \ell/2\rfloor}\frac{x^{\ell-2k}(-1)^{k}2^{-k}}{k!(\ell-2k)!}.\label{Hone}
\end{flalign}
The generating function leads to the following explicit formula:
\begin{equation}\label{Htwo}
h_{\ell}(x)=\sum_{k=0}^{\lfloor \ell/2\rfloor}\frac{x^{\ell-2k}(-1)^{k}2^{-k}}{k!(\ell-2k)!},\qquad\forall\,\ell\geq0,\,\,\forall\, x\in\R.
\end{equation}
For example, $h_{0}(x)=1$, $h_{1}(x)=x$, $h_{2}(x)=(1/2)x^{2}-(1/2)$, and $h_{3}(x)=(1/6)x^{3}-(1/2)x$, for all $x\in\R$.  Let $\N\colonequals\{0,1,2,\ldots\}$.  It follows from \eqref{Hzero} that
\begin{equation}\label{Hrec}
\frac{\d}{\d x}h_{\ell+1}(x)=h_{\ell}(x),\qquad\forall\,\ell\geq0,\qquad\forall\,x\in\R.
\end{equation}

Let $\ell=(\ell_{1},\ldots\ell_{\adimn})\in\N^{\adimn}$ and for any $x=(x_{1},\ldots,x_{\adimn})\in\R^{\adimn}$, define
$$h_{\ell}(x)\colonequals\prod_{i=1}^{\adimn}h_{\ell_{i}}(x_{i}).$$
For any $\ell\in\N^{\adimn}$, define $\ell!\colonequals\prod_{i=1}^{\adimn}(\ell_{i}!)$ and define $\vnorm{\ell}_{1}\colonequals\sum_{i=1}^{\adimn}\ell_{i}$.

\begin{equation}\label{Hfour}
h_{\ell}(x)
=\frac{(-1)^{\ell}}{\ell!}e^{x^{2}/2}(d/dx)^{\ell}e^{-x^{2}/2},\qquad\forall\,\ell\geq0,\,\,\forall\, x\in\R.
\end{equation}
\begin{equation}\label{Hsix}
\int_{\R}h_{\ell}^{2}e^{-x^{2}/2}dx/\sqrt{2\pi}=1/\ell!,\qquad\forall\,\ell\geq0.
\end{equation}

Let $t>0$ and let $\lambda\in\R$.  Then
\begin{equation}\label{Hnine}
\begin{aligned}
&\int_{\R}\sum_{\ell=0}^{\infty}\lambda^{\ell}h_{\ell}(x)e^{-t(x-u)^{2}}\frac{dx}{\sqrt{\pi/t}}
\stackrel{\eqref{Hzero}}{=}
\int_{\R}e^{\lambda x-\lambda^{2}/2}e^{-t(x-u)^{2}}\frac{dx}{\sqrt{\pi/t}}\\
&\qquad=\int_{\R}e^{-t(x-u-\lambda/(2t))^{2}}e^{\lambda u}e^{\lambda^{2}(-1/2+1/(4t))}\frac{dx}{\sqrt{\pi/t}}
=e^{\lambda u}e^{\lambda^{2}(-1/2+1/(4t))}\\
&\qquad=\sum_{j=0}^{\infty}\lambda^{j}\frac{u^{j}}{j!}\sum_{k=0}^{\infty}\lambda^{2k}\frac{(-1/2+1/(4t))^{k}}{k!}.
\end{aligned}
\end{equation}
Therefore, $\forall$ nonnegative integers $\ell\geq0$, and $\forall$ $s,t>0$, $\forall$ $u\in\R$,
$$
\int_{\R}h_{\ell}(x)e^{-t(x-u)^{2}}\frac{dx}{\sqrt{\pi/t}}
=\sum_{k=0}^{\lfloor \ell/2\rfloor }\frac{u^{\ell-2k}}{[\ell-2k]!}\frac{(-1/2+1/(4t))^{k}}{k!}.
$$


\begin{flalign*}
&\int_{\R}h_{\ell}(x)e^{-s(x-u)^{2}}e^{-x^{2}/2}\frac{dx}{\sqrt{2\pi}}  
=e^{u^{2}s^{2}/(1/2+s)}e^{-su^{2}}\int_{\R}h_{\ell}(x)e^{-(1/2+s)(x-us/(1/2+s))^{2}}\frac{dx}{\sqrt{2\pi}} \\ %
&\qquad\qquad\qquad=e^{u^{2}s^{2}/(1/2+s)}e^{-su^{2}}\frac{1}{\sqrt{1+2s}}\sum_{k=0}^{\lfloor \ell/2\rfloor }\frac{(us/(1/2+s))^{\ell-2k}}{[\ell-2k]!}\frac{(-s/(1+2s))^{k}}{k!}\\
&\qquad\qquad\qquad=e^{u^{2}s^{2}/(1/2+s)}e^{-su^{2}}\frac{1}{\sqrt{1+2s}}\sum_{k=0}^{\lfloor \ell/2\rfloor }(-1)^{k}2^{-k}\frac{u^{\ell-2k}s^{\ell-k}(1/2+s)^{k-\ell}}{[\ell-2k]!}\frac{1}{k!}\\
&\qquad\qquad\qquad=e^{u^{2}s^{2}/(1/2+s)}e^{-su^{2}}\frac{1}{\sqrt{1+2s}}\sum_{k=0}^{\lfloor \ell/2\rfloor }(-1)^{k}2^{-k}\frac{u^{\ell-2k}(1+1/(2s))^{k-\ell}}{[\ell-2k]!k!}.
\end{flalign*}

%

So, in the $L_{2}(\gamma_{1})$ sense, $\forall$ $x\in\R$,
\begin{equation}\label{expherm}
\begin{aligned}
e^{-s(x-u)^{2}}
&=\sum_{\ell=0}^{\infty}\Big(\int_{\R}e^{-sy^{2}}h_{\ell}(y)\gamma_{1}(y)\,\d x\Big) h_{\ell}(x) \ell!\\
&=e^{u^{2}s^{2}/(1/2+s)}e^{-su^{2}}\frac{1}{\sqrt{1+2s}}\sum_{\ell=0}^{\infty}\sum_{k=0}^{\lfloor \ell/2\rfloor }(-1)^{k}2^{-k}\frac{u^{\ell-2k}(1+1/(2s))^{k-\ell}}{[\ell-2k]!k!}h_{\ell}(x) \ell!\\
&=e^{u^{2}s^{2}/(1/2+s)}e^{-su^{2}}\frac{1}{\sqrt{1+2s}}\sum_{\ell=0}^{\infty}\sum_{k=0}^{\lfloor \ell/2\rfloor }(-1)^{k}2^{-k}\frac{u^{\ell-2k}(1+1/(2s))^{k-\ell}\sqrt{\ell!}}{[\ell-2k]!k!}h_{\ell}(x) \sqrt{\ell!}.
\end{aligned}
\end{equation}
For any integer $\ell\geq0$, define $c_{2\ell}\colonequals (-s/(1+2s))^{\ell}\frac{(2\ell)!}{\ell!}$.  From Stirling's approximation,
$$\frac{\sqrt{(2\ell)!}}{\ell!}
\leq\frac{e^{1/2}(2\ell)^{1/4}(2\ell/e)^{\ell}}{(2\pi\ell)^{1/2}(\ell/e)^{\ell}}
\leq \ell^{-1/4} 2^{\ell}.$$
So,
$$\abs{c_{2\ell}}\leq \ell^{-1/4}(2s/(1+2s))^{\ell}=\ell^{-1/4}\frac{1}{(1+1/(2s))^{\ell}}.$$
Choosing $s\colonequals\beta^{2}/[2(1-\beta^{2})]$, we have
$$\abs{c_{2\ell}}\leq\ell^{-1/4}\beta^{2\ell}. $$

We now estimate \eqref{expherm}.  $\min_{0\leq k\leq \ell/2}(\ell-2k)!k!c^{k}$ occurs when $k=(\ell/2)-\sqrt{2c\ell}/4>0$, or when $k=0$ if $(\ell/2)-\sqrt{2c\ell}/4\leq0$, i.e. if $c\geq2\ell$.

\textbf{Case 1}.  In the case $c\leq 2\ell$, we therefore have
%
%
\begin{flalign*}
&\frac{\sqrt{\ell!}}{(\ell-2k)!k!}
\leq\frac{e^{1/2}\ell^{\ell/2+1/4}e^{-\ell/2}}{2\pi (\ell-2k)^{\ell-2k+1/2}e^{-(\ell-2k)}k^{k+1/2}e^{-k}}\\
&=(2/\sqrt{2c})^{\sqrt{2c\ell}/2+1/2}e^{\sqrt{\ell}(\sqrt{2c}/2-\sqrt{2c}/4)}\frac{e^{1/2}\ell^{\ell/2+1/4}}{2\pi \ell^{\sqrt{2c\ell}/4+1/4}[(\ell/2)-\sqrt{2c\ell}/4]^{(\ell/2)-\sqrt{2c\ell}/4+1/2}}\\
&=\frac{e^{1/2}}{2\pi}(2/\sqrt{2c})^{\sqrt{2c\ell}/2+1/2}e^{\sqrt{c\ell}(\sqrt{2}/4)}\Big(\frac{1}{1/2-\frac{\sqrt{2c}}{4\sqrt{\ell}}}\Big)^{\ell/2}\frac{\ell^{1/4}\cdot \ell^{-\sqrt{2c\ell}/4}[(\ell/2)-\sqrt{2c\ell}/4]^{\sqrt{2c\ell}/4}}{\sqrt{(\ell/2)-\sqrt{2c\ell}/4}}
\\
&\leq\frac{1}{3}2^{\sqrt{2c\ell}/4+1/2}c^{-\sqrt{2c\ell}/4}e^{\sqrt{c\ell}(\sqrt{2}/4)}2^{\ell/2}\Big(\frac{1}{1-\frac{\sqrt{2c}}{2\sqrt{\ell}}}\Big)^{\ell/2}e^{1}\ell^{-1/4}
\ell^{-\sqrt{2c\ell}/4}[(\ell/2)-\sqrt{2c\ell}/4]^{\sqrt{2c\ell}/4}\\
&\leq \ell^{-1/4}2^{\sqrt{2c\ell}/4+1/2}c^{-\sqrt{2c\ell}/4}e^{\sqrt{c\ell}(\sqrt{2}/2)}2^{\ell/2}[(1/2)-\sqrt{2c}/4\sqrt{\ell}]^{\sqrt{2c\ell}/4}\\
&\leq\ell^{-1/4}\sqrt{2}c^{-\sqrt{2c\ell}/4}e^{\sqrt{c\ell}(\sqrt{2}/2)}2^{\ell/2}[1-\sqrt{2c}/2\sqrt{\ell}]^{\sqrt{2c\ell}/4}
\leq\ell^{-1/4}\sqrt{2}c^{-\sqrt{2c\ell}/4}e^{\sqrt{c\ell}(\sqrt{2}/2)}2^{\ell/2}e^{-c/4}\\
&\leq 2\cdot 2^{\ell/2}\ell^{-1/4}c^{-\sqrt{2c\ell}/4}e^{\sqrt{c\ell}(\sqrt{2}/2)}e^{-c/4}.
\end{flalign*}

$$
\frac{\sqrt{\ell!}}{(\ell-2k)!k!c^{k}}
\leq 2\cdot 2^{\ell/2}\ell^{-1/4}c^{-\ell/2}e^{\sqrt{c\ell}(\sqrt{2}/2)}e^{-c/4}.
$$


Choosing $s\colonequals\beta^{2}/[2(1-\beta^{2})]$, we have $1+1/(2s)=\beta^{-2}$ and $c=(2u^{2}\beta^{2})$
\begin{flalign*}
(u\beta^{2})^{\ell}\frac{\sqrt{\ell!}}{(\ell-2k)!k!(2u\beta^{2})^{k}}
&\leq (u\beta^{2})^{\ell}2\cdot 2^{\ell/2}\ell^{-1/4}(2u^{2}\beta^{2})^{-\ell/2}e^{\sqrt{c\ell}(\sqrt{2}/2)}e^{-c/4}\\
&=\beta^{\ell}e^{\sqrt{c\ell}(\sqrt{2}/2)}e^{-c/4}
=\beta^{\ell}e^{\abs{\alpha}\sqrt{\ell}}e^{-\alpha^{2}/2}.
\end{flalign*}

%
%
%

Let $\lambda>0$.  Let us find $t>0$ such that $\beta^{\ell}e^{\abs{\alpha}\sqrt{\ell}}e^{-\alpha^{2}/2}\leq t(\beta+\lambda)^{\ell}$.
Solving for $t$, we get
$$t\geq \frac{e^{\abs{\alpha}\sqrt{\ell}}e^{-\alpha^{2}/2} }{(1+\lambda/\beta)^{\ell}}.$$
That is, we can choose
$$t\colonequals\sup_{\ell\geq0}\frac{e^{\abs{\alpha}\sqrt{\ell}}e^{-\alpha^{2}/2} }{(1+\lambda/\beta)^{\ell}}.$$
The supremum occurs when $\ell \approx \alpha^{2}/(4\log^{2}(1+\lambda/\beta))$, so that  
$$t=\frac{e^{\alpha^{2} /[2\log(1+\lambda/\beta)]}e^{-\alpha^{2}/2} }{(1+\lambda/\beta)^{\alpha^{2}/[4\log^{2}(1+\lambda/\beta)]}}
=(1+\lambda/\beta)^{\alpha^{2}\Big(\frac{1}{2\log(1+\lambda/\beta)}-\frac{1}{4\log^{2}(1+\lambda/\beta)}\Big)}e^{-\alpha^{2}/2}\leq e^{\frac{\alpha^{2}\beta}{2\lambda}}.$$
The last inequality used $0\leq\lambda\leq\beta$ and $\log(1+x)\geq x/2$ for all $0\leq x\leq 1$.
%


%
%

In summary, when $c\leq2\ell$, for any $\lambda>0$ we have
$$
(u\beta^{2})^{\ell}\frac{\sqrt{\ell!}}{(\ell-2k)!k!(2u\beta^{2})^{k}}
\leq  (\beta+\lambda)^{\ell}e^{\frac{\alpha^{2}\beta}{2\lambda}}.
$$

\textbf{Case 2}.  In the remaining case that $c\geq2\ell$, we similarly have


$$
\frac{\sqrt{\ell!}}{(\ell-2k)!k!c^{k}}
\leq\frac{\sqrt{\ell!}}{\ell!}=\frac{1}{\sqrt{\ell!}}.
$$

Choosing $s\colonequals\beta^{2}/[2(1-\beta^{2})]$, we have $1+1/(2s)=\beta^{-2}$ and $c=(2u^{2}\beta^{2})$, $u=\alpha/\beta$
$$
(u\beta^{2})^{\ell}\frac{\sqrt{\ell!}}{(\ell-2k)!k!(2u\beta^{2})^{k}}
\leq (\alpha\beta)^{\ell}\frac{1}{\sqrt{\ell!}}
=\beta^{\ell}\frac{\alpha^{\ell}}{\sqrt{\ell!}}
\leq\beta^{\ell} e^{\alpha^{2}/2}.
$$

%

Equation \ref{expherm} implies the following (noting also that $e^{u^{2}s^{2}/(1/2+s)}e^{-su^{2}}=e^{-\alpha^{2}/2}$ since $s=\beta^{2}/[2(1-\beta^{2})]$ and $u=\alpha/\beta$),

\begin{lemma}\label{gausfourier}
Let $0<\lambda\leq\beta<1$.  Let $\alpha\in\R$.  Then $\exists$ $c_{1}',c_{2}',\ldots\in\R$ with
$$\abs{c_{k}'}\leq (\beta+\lambda)^{k}e^{\alpha^{2}\cdot\max\left(0,\frac{\beta}{2\lambda}-\frac{1}{2}\right)},$$
for all $k\geq1$ and $c_{0}'\colonequals e^{-\alpha^{2}/2}$ and $c_{1}'\colonequals \alpha\beta e^{-\alpha^{2}/2}$ such that
$$(1-\beta^{2})^{-1/2}e^{-\frac{[\beta x-\alpha]^{2}}{2(1-\beta^{2})}}
=\sum_{k=0}^{\infty}c_{k}'h_{k}(x)\sqrt{k!},\qquad\forall\,x\in\R.$$
The sum on the right converges uniformly on compact subsets of $\R$, and it also converges in the $L_{2}(\gamma_{1})$ sense.  Also, when $\alpha=0$ we have $\abs{c_{k}}\leq \beta^{k}k^{-1/4}$ for all $k\geq1$.
\end{lemma}

Denote $\Phi(t)\colonequals\int_{-\infty}^{t}\gamma_{1}(x)\,\d x$.  Since $\frac{\d}{\d x}\Big|_{x=0}\Phi\Big(\frac{\beta x-\alpha }{\sqrt{1-\beta^{2}}}\Big)=\frac{\beta}{\sqrt{1-\beta^{2}}}e^{-\frac{[\beta x-\alpha]^{2}}{2(1-\beta^{2})}}\frac{1}{\sqrt{2\pi}}$, \eqref{Hrec} implies:

%
%
%
%
%

\begin{lemma}\label{erfourier}
Let $0<\beta,\lambda<1$.  Let $\alpha\in\R$.  Then $\exists$ $c_{1},c_{2},\ldots\in\R$ with
$$\abs{c_{k}}\leq (\beta+\lambda)^{k}e^{\alpha^{2}\cdot\max\left(0,\frac{\beta}{2\lambda}-\frac{1}{2}\right)},$$
for all $k\geq2$, with $c_{0}\colonequals\Phi(\alpha)=a$, $c_{1}\colonequals \frac{\beta}{\sqrt{2\pi}}e^{-\alpha^{2}/2}$ such that
$$\Phi\Big(\frac{\beta x-\alpha }{\sqrt{1-\beta^{2}}}\Big)
=\sum_{k=0}^{\infty}c_{k}h_{k}(x)\sqrt{k!},\qquad\forall\,x\in\R.$$
The sum on the right converges uniformly on compact subsets of $\R$, and it also converges in the $L_{2}(\gamma_{1})$ sense.  Also, when $\alpha=0$ we have $\abs{c_{k}}\leq (1-\beta^{2})^{-1}\beta^{k}k^{1/4}$ for all $k\geq1$.
\end{lemma}

\begin{remark}
Let $z\in\R^{\adimn}$.  From the Cauchy-Schwarz inequality (applied to discrete sequences of real numbers),
\begin{equation}\label{cs1}
\begin{aligned}
&\Big(\int_{\R^{\adimn}}f(x)\frac{e^{-\frac{[\beta\langle x,\nu(z)\rangle-\alpha]^{2}}{2(1-\beta^{2})}}}{(\pi/(1/2+\beta^{2}/[2(1-\beta^{2})]))^{1/2}}\gamma_{\adimn}(x)\,\d x\Big)^{2}\\
&\qquad=\Big(\sum_{k=1}^{\infty}c_{k}'\int_{\R^{\adimn}}f(x)h_{k}(\langle x,z\rangle)\sqrt{k!}\gamma_{\adimn}(x)\,\d x\Big)^{2}\\
&\qquad\leq \Big(\sum_{j=1}^{\infty}\abs{c_{j}'}\Big)\Big(\sum_{k=1}^{\infty}\abs{c_{k}'}\Big[\int_{\R^{\adimn}}f(x)h_{k}(\langle x,z\rangle)\sqrt{k!}\gamma_{\adimn}(x)\,\d x\Big]^{2}\Big).
\end{aligned}
\end{equation}
\end{remark}

%
%

\begin{lemma}\label{ersecondfourier}
Let $0<\beta,\lambda<1$.  Let $\alpha\in\R$.  Then $\exists$ $c_{1}'',c_{2}'',\ldots\in\R$ with
$$\abs{c_{k}''}\leq (\beta+\lambda)^{k}k^{1/2}e^{\alpha^{2}\cdot\max\left(0,\frac{\beta}{2\lambda}-\frac{1}{2}\right)},$$
for all $k\geq1$, with $c_{0}''\colonequals\beta\alpha e^{-\alpha^{2}/2}$,
$c_{1}''\colonequals \beta^{2}(\alpha^{2}-1)e^{-\alpha^{2}/2}$ such that
$$(1-\beta^{2})^{-3/2}[\alpha -x\beta ]e^{-\frac{[\beta x-\alpha]^{2}}{2(1-\beta^{2})}}
=\sum_{k=0}^{\infty}c_{k}''h_{k}(x)\sqrt{k!},\qquad\forall\,x\in\R.$$
The sum on the right converges uniformly on compact subsets of $\R$, and it also converges in the $L_{2}(\gamma_{1})$ sense.
\end{lemma}
%

%
Mehler's formula then says: for any $0<\rho<1$, for any $x,y\in\R^{\adimn}$
\begin{equation}\label{Height}
e^{-(\|x\|^{2}+\|y\|^{2})/2}
\sum_{k=0}^{\infty}\rho^{k}\sum_{\substack{\ell\in\N^{\adimn}\colon\\ \|\ell\|_{1}=k}}h_{\ell}(x)h_{\ell}(y)\ell!
=(1-\rho^{2})^{-(\adimn)/2}e^{\frac{-\|x\|^{2}-\|y\|^{2}+2\rho\langle x,y\rangle}{2(1-\rho^{2})}}.
\end{equation}
%

\subsection{Preliminaries and Notation}\label{secpre}

We say that $\Sigma\subset\R^{\adimn}$ is an $\sdimn$-dimensional $C^{\infty}$ manifold with boundary if $\Sigma$ can be locally written as the graph of a $C^{\infty}$ function on a relatively open subset of $\{(x_{1},\ldots,x_{\sdimn})\in\R^{\sdimn}\colon x_{\sdimn}\geq0\}$.  For any $(\adimn)$-dimensional $C^{\infty}$ manifold $\Omega\subset\R^{\adimn}$ such that $\partial\Omega$ itself has a boundary, we denote
\begin{equation}\label{c0def}
\begin{aligned}
C_{0}^{\infty}(\Omega;\R^{\adimn})
&\colonequals\{f\colon \Omega\to\R^{\adimn}\colon f\in C^{\infty}(\Omega;\R^{\adimn}),\, f(\partial\partial \Omega)=0,\\
&\qquad\qquad\qquad\exists\,r>0,\,f(\Omega\cap(B(0,r))^{c})=0\}.
\end{aligned}
\end{equation}
We also denote $C_{0}^{\infty}(\Omega)\colonequals C_{0}^{\infty}(\Omega;\R)$.  We let $\mathrm{div}$ denote the divergence of a vector field in $\R^{\adimn}$.  For any $r>0$ and for any $x\in\R^{\adimn}$, we let $B(x,r)\colonequals\{y\in\R^{\adimn}\colon\vnormt{x-y}\leq r\}$ be the closed Euclidean ball of radius $r$ centered at $x\in\R^{\adimn}$.  Here $\partial\partial\Omega$ refers to the $(\sdimn-1)$-dimensional boundary of $\Omega$.

\begin{definition}[\embolden{Reduced Boundary}]\label{rbdef}
A measurable set $\Omega\subset\R^{\adimn}$ has \embolden{locally finite surface area} if, for any $r>0$,
$$\sup\left\{\int_{\Omega}\mathrm{div}(X(x))\,\d x\colon X\in C_{0}^{\infty}(B(0,r),\R^{\adimn}),\, \sup_{x\in\R^{\adimn}}\vnormt{X(x)}\leq1\right\}<\infty.$$
Equivalently, $\Omega$ has locally finite surface area if $\nabla 1_{\Omega}$ is a vector-valued Radon measure such that, for any $x\in\R^{\adimn}$, the total variation
$$
\vnormt{\nabla 1_{\Omega}}(B(x,1))
\colonequals\sup_{\substack{\mathrm{partitions}\\ C_{1},\ldots,C_{m}\,\mathrm{of}\,B(x,1) \\ m\geq1}}\sum_{i=1}^{m}\vnormt{\nabla 1_{\Omega}(C_{i})}
$$
is finite \cite{cicalese12}.  If $\Omega\subset\R^{\adimn}$ has locally finite surface area, we define the \embolden{reduced boundary} $\redb \Omega$ of $\Omega$ to be the set of points $x\in\R^{\adimn}$ such that
$$N(x)\colonequals-\lim_{r\to0^{+}}\frac{\nabla 1_{\Omega}(B(x,r))}{\vnormt{\nabla 1_{\Omega}}(B(x,r))}$$
exists, and it is exactly one element of $S^{\sdimn}\colonequals\{x\in\R^{\adimn}\colon\vnorm{x}=1\}$.
\end{definition}

The reduced boundary $\redb\Omega$ is a subset of the topological boundary $\partial\Omega$.  Also, $\redb\Omega$ and $\partial\Omega$ coincide with the support of $\nabla 1_{\Omega}$, except for a set of $\sdimn$-dimensional Hausdorff measure zero.

Let $\Omega\subset\R^{\adimn}$ be an $(\adimn)$-dimensional $C^{2}$ submanifold with reduced boundary $\Sigma\colonequals\redb \Omega$.  Let $N\colon\redA\to S^{\sdimn}$ be the unit exterior normal to $\redA$.  Let $X\in C_{0}^{\infty}(\R^{\adimn},\R^{\adimn})$.  We write $X$ in its components as $X=(X_{1},\ldots,X_{\adimn})$, so that $\mathrm{div}X=\sum_{i=1}^{\adimn}\frac{\partial}{\partial x_{i}}X_{i}$.  Let $\Psi\colon\R^{\adimn}\times(-1,1)\to\R^{\adimn}$ such that
\begin{equation}\label{nine2.3}
\Psi(x,0)=x,\qquad\qquad\frac{\d}{\d s}\Psi(x,s)=X(\Psi(x,s)),\quad\forall\,x\in\R^{\adimn},\,s\in(-1,1).
\end{equation}
For any $s\in(-1,1)$, let $\Omega^{(s)}\colonequals\Psi(\Omega,s)$.  Note that $\Omega^{(0)}=\Omega$.  Let $\Sigma^{(s)}\colonequals\redb\Omega^{(s)}$, $\forall$ $s\in(-1,1)$.
\begin{definition}
We call $\{\Omega^{(s)}\}_{s\in(-1,1)}$ as defined above a \embolden{variation} of $\Omega\subset\R^{\adimn}$.  We also call $\{\Sigma^{(s)}\}_{s\in(-1,1)}$ a \embolden{variation} of $\Sigma=\redb\Omega$.
\end{definition}

For any $x\in\R^{\adimn}$ and any $s\in(-1,1)$, define
\begin{equation}\label{two9c}
V(x,s)\colonequals\int_{\Omega^{(s)}}G(x,y)\,\d y.
\end{equation}

Below, when appropriate, we let $\,\d x$ denote Lebesgue measure, restricted to a surface $\redA\subset\R^{\adimn}$.

\begin{lemma}[\embolden{Existence of a Maximizer}]\label{existlem}
Let $0<\beta\leq\rho$, define $\lambda\colonequals \rho-\beta$ and let $0<\epsilon\leq e^{-\alpha^{2}\cdot\max\left(0,\frac{\beta}{2\lambda}-\frac{1}{2}\right)}$.  (If $\alpha=0$ let $0<\epsilon\leq 1$.)  Then there exists a measurable set $\Omega\subset\R^{\adimn}$ maximizing Problem \ref{prob2z}.
\end{lemma}
\begin{proof}
Let $f\colon\R^{\adimn}\to[0,1]$ be measurable.  Denote $z=z(f)\colonequals\int_{\R^{\adimn}}xf(x)\gamma_{\adimn}(x)\,\d x$, $\nu(z)\colonequals z/\vnorm{z}$ if $z\neq0$ and $\nu(0)\colonequals(1,0,\ldots,0)$.   Fix $w\in\R^{\adimn}$.  The set
$$D_{0}\colonequals\{f\colon\R^{\adimn}\to[0,1]\colon \int_{\R^{m}} f(x)\gamma_{\adimn}(x)\,\d x=a\quad\mathrm{and}\quad z(f)=w\}$$
is norm closed, bounded and convex, so it is weakly compact and convex.  Consider the function
$$
C(f)
\colonequals\int_{\R^{\adimn}}f(x)T_{\rho}f(x)\gamma_{\adimn}(x)\,\d x-\epsilon\penf.
$$
This function is weakly continuous on $D_{0}$, and $D_{0}$ is weakly compact, so there exists $\widetilde{f}\in D_{0}$ such that $C(\widetilde{f})=\max_{f\in D_{0}}C(f)$.  If $f\in D_{0}$, then $z(f)=w$, so $\nu(z(f))=\nu(w)$, i.e. $\forall\,f\in D_{0}$,
$$
C(f)
=\int_{\R^{\adimn}}f(x)T_{\rho}f(x)\gamma_{\adimn}(x)\,\d x-\epsilon\penfw.
$$

Moreover, $C$ is convex since for any $0<t<1$ and for any $f,g\in D_{0}$,
\begin{flalign*}
&tC(f)+(1-t)C(g)-C(tf+(1-t)g)\\
&\quad=\int_{\R^{\adimn}}
\Big(tf(x)T_{\rho}f(x)+(1-t)g(x)T_{\rho}g(x)\\
&\qquad\qquad\qquad\qquad\qquad-(tf(x)+(1-t)g(x))T_{\rho}[tf(x)+(1-t)g(x)] \Big)\gamma_{\adimn}(x)\,\d x\\
&\qquad\qquad\qquad-\epsilon t(1-t)\Big[\int_{\R^{\adimn}} (f(x)-g(x))\Phi\Big(\frac{\beta\langle x,\nu(w)\rangle-\alpha}{\sqrt{1-\beta^{2}}}\Big)\gamma_{\adimn}(x)\,\d x\Big]^{2}\\
&\quad=t(1-t)\int_{\R^{\adimn}}
\Big((f(x)-g(x))T_{\rho}[f(x)-g(x)]\Big)\gamma_{\adimn}(x)\,\d x\\
&\qquad\qquad\qquad-\epsilon t(1-t)\Big[\int_{\R^{\adimn}} (f(x)-g(x))\Phi\Big(\frac{\beta\langle x,\nu(w)\rangle-\alpha}{\sqrt{1-\beta^{2}}}\Big)\gamma_{\adimn}(x)\,\d x\Big]^{2}\\
&\stackrel{\eqref{Height}\wedge\eqref{nseq}\wedge\eqref{cs1}}{\geq}t(1-t)\Big(\sum_{k=0}^{\infty}[\rho^{k}-\epsilon\cdot c_{k}]
\sum_{\substack{\ell\in\N^{\adimn}\colon\\ \|\ell\|_{1}=k}}\ell!
\Big[\int_{\R^{\adimn}}(f(x)-g(x))h_{\ell}(x)\gamma_{\adimn}(x)\,\d x\Big]^{2}\Big)\\
&\quad\geq0.
\end{flalign*}
The penultimate inequality used $\epsilon\leq e^{-\alpha^{2}\cdot\max\left(0,\frac{\beta}{2\lambda}-\frac{1}{2}\right)}$ and Lemma \ref{erfourier}, $c_{0}\colonequals \Phi(\alpha)=a$, $c_{1}\colonequals \frac{\beta}{\sqrt{2\pi}}e^{-\alpha^{2}/2}$, $\abs{c_{k}}\leq (\beta+\lambda)^{k}e^{\alpha^{2}\cdot\max\left(0,\frac{\beta}{2\lambda}-\frac{1}{2}\right)}$ for all $k\geq2$,
%

Since $C$ is convex, its maximum must be achieved at an extreme point of $D_{0}$, so that $\gamma_{\adimn}(\{x\in\R^{\adimn}\colon f(x)\in\{0,1\}\})=1$.  Then, define $\Omega\colonequals\{x\in\R^{\adimn}\colon f(x)=1\}$, so that $f=1_{\Omega}$.  Finally, note that $C(1_{\Omega})=C(1_{R\Omega})$ for any rotation $R\colon\R^{\adimn}\to\R^{\adimn}$.
\end{proof}



\begin{lemma}[\embolden{Regularity of a Maximizer}]\label{reglem}
Let $0<\beta\leq\rho$, define $\lambda\colonequals \rho-\beta$ and let $0<\epsilon\leq e^{-\alpha^{2}\cdot\max\left(0,\frac{\beta}{2\lambda}-\frac{1}{2}\right)}$.  (If $\alpha=0$ let $0<\epsilon\leq 1$.)  Let $\Omega\subset\R^{\adimn}$ be the measurable set maximizing Problem \ref{prob2z}, guaranteed to exist by Lemma \ref{existlem}.  Then the set $\Omega$ has locally finite surface area.  Moreover, for all $x\in\partial\Omega$, there exists a neighborhood $U$ of $x$ such that $U\cap \partial\Omega$ is a finite union of $C^{\infty}$ $\sdimn$-dimensional manifolds.
\end{lemma}
\begin{proof}
This follows from a first variation argument and the strong unique continuation property for the heat equation.  Define $\Phi(t)\colonequals\int_{-\infty}^{t}\gamma_{1}(x)\,\d x$ and define
\begin{equation}\label{zdef}
z\colonequals\int_{\Omega}x\gamma_{\adimn}(x)\,\d x\in\R^{\adimn},
\end{equation}
$\nu(z)\colonequals z/\vnorm{z}$ if $z\neq0$ and $\nu(0)\colonequals(1,0,\ldots,0)$.   We first claim that there exist a constant $c\in\R$ such that
\begin{equation}\label{zero8}
\begin{aligned}
\Omega&\supset\Big\{x\in\R^{\adimn}\colon T_{\rho}1_{\Omega}(x)-\epsilon\Phiz>c\Big\}\\
\Omega^{c}&\supset\Big\{x\in\R^{\adimn}\colon T_{\rho}1_{\Omega}(x)-\epsilon\Phiz<c\Big\}.
\end{aligned}
\end{equation}

By the Lebesgue density theorem \cite[1.2.1, Proposition 1]{stein70}, we may assume that, if $y\in \Omega$, then we have $\lim_{r\to0}\gamma_{\adimn}(\Omega\cap B(y,r))/\gamma_{\adimn}(B(y,r))=1$.

We prove \eqref{zero8} by contradiction.  Suppose there exist $c\in\R$, and there exists $y,z\in\Omega$ such that
$$T_{\rho}(1_{\Omega})(y)-\epsilon\langle x,z^{(j)}-z^{(k)}\rangle<c,\qquad T_{\rho}(1_{\Omega})(z)-\epsilon\langle x,z^{(j)}-z^{(k)}\rangle>c.$$
By \eqref{oudef}, $T_{\rho}(1_{\Omega})(x)$ is a continuous function of $x$.  And by the Lebesgue density theorem, there exist disjoint measurable sets $U_{1},U_{2}$ with positive Lebesgue measure such that $U_{1},U_{2}\subset\Omega$ such that $\gamma_{\adimn}(U_{1})=\gamma_{\adimn}(U_{2})$ and such that
\begin{equation}\label{zero9.0}
\begin{aligned}
&T_{\rho}(1_{\Omega})(y')-\epsilon\Phiz<c,\,\,\forall\,y'\in U_{1},\\
&T_{\rho}(1_{\Omega})(y')-\epsilon\Phiz>c,\,\,\forall\,y'\in U_{2}.
\end{aligned}
\end{equation}
We define a new set $\widetilde{\Omega}\subset\R^{\adimn}$ such that $\widetilde{\Omega}\colonequals U_{2}\cup \Omega\setminus U_{1}$.  Denote $\zeta\colonequals\int_{\widetilde{\Omega}}x\gamma_{\adimn}(x)\,\d x$, $z^{(i)}\colonequals\int_{U_{i}}x\gamma_{\adimn}(x)\,\d x$ for $i=1,2$.  Then,
\begin{flalign*}
&\int_{\R^{\adimn}}1_{\widetilde{\Omega}}(x)T_{\rho}1_{\widetilde{\Omega}}(x)\gamma_{\adimn}(x)\,\d x
-\epsilon \Big[\int_{\widetilde{\Omega}}\Phi\Big(\frac{\beta \langle \nu(\zeta),x\rangle-\alpha}{\sqrt{1-\rho^{2}}}\Big)\gamma_{\adimn}(x)dx\Big]^{2}\\
&\qquad-\int_{\R^{\adimn}}1_{\Omega}(x)T_{\rho}1_{\Omega}(x)\gamma_{\adimn}(x)\,\d x
+\epsilon \pen\\
&=\int_{\R^{\adimn}}[1_{\Omega}-1_{U_{1}}+1_{U_{2}}](x)T_{\rho}[1_{\Omega}-1_{U_{1}}+1_{U_{2}}]\gamma_{\adimn}(x)\,\d x\\
&\qquad\qquad-\epsilon \Big[\int_{\widetilde{\Omega}}\Phi\Big(\frac{\beta \langle \nu(z-z^{(1)}+z^{(2)}),x\rangle-\alpha}{\sqrt{1-\rho^{2}}}\Big)\gamma_{\adimn}(x)dx\Big]^{2}\\
&\qquad-\int_{\R^{\adimn}}1_{\Omega}(x)T_{\rho}1_{\Omega}(x)\gamma_{\adimn}(x)\,\d x
+\epsilon \pen\\
&=2\int_{\R^{\adimn}}[-1_{U_{1}}+1_{U_{2}}](x)\Big(T_{\rho}(1_{\Omega})(x)-\epsilon\Phiz\Big)\gamma_{\adimn}(x)\,\d x\\
&\qquad\qquad
+2\int_{\R^{\adimn}}[1_{U_{1}}-1_{U_{2}}](x)T_{\rho}[1_{U_{1}}-1_{U_{2}}](x)\gamma_{\adimn}(x)\,\d x.
\end{flalign*}
Rearranging the terms and using \eqref{zdef}, the previous quantity is
\begin{flalign*}
&2\int_{\R^{\adimn}}[-1_{U_{j}}+1_{U_{k}}](x)\Big(T_{\rho}[1_{\Omega_{j}}-1_{\Omega_{k}}](x)-\epsilon\Phiz\Big)\gamma_{\adimn}(x)\,\d x\\
&\qquad\qquad\qquad+2\int_{\R^{\adimn}}[1_{U_{j}}-1_{U_{k}}]T_{\rho}[1_{U_{j}}-1_{U_{k}}]\gamma_{\adimn}(x)\,\d x\\
&\qquad\qquad\qquad-2\epsilon\vnorm{\int_{\R^{\adimn}}\Phiz[1_{U_{j}}-1_{U_{k}}](x)\gamma_{\adimn}(x)\,\d x}^{2}.
\end{flalign*}
The first quantity is positive by \eqref{zero9.0}.  The remaining two quantities sum to a nonnegative number by Mehler's formula \eqref{Height} and Lemma \ref{erfourier}
\begin{flalign*}
&\int_{\R^{\adimn}}[1_{U_{j}}-1_{U_{k}}]T_{\rho}[1_{U_{j}}-1_{U_{k}}]\gamma_{\adimn}(x)\,\d x\\
&\qquad\qquad\qquad\qquad
-\epsilon\Big\|\int_{\R^{\adimn}}\Phiz[1_{U_{j}}-1_{U_{k}}](x)\gamma_{\adimn}(x)\,\d x\Big\|^{2}\\
&\stackrel{\eqref{Height}\wedge\eqref{nseq}}{=}\sum_{k=0}^{\infty}[\rho^{k}-\epsilon\cdot1_{\{k=1\}}]
\sum_{\substack{\ell\in\N^{\adimn}\colon\\ \|\ell\|_{1}=k}}\ell!
\Big[\int_{\R^{\adimn}}(1_{U_{j}}(x)-1_{U_{k}}(x))h_{\ell}(x)\gamma_{\adimn}(x)\,\d x\Big]^{2}
\geq0.
\end{flalign*}
The last inequality used $\epsilon\leq e^{-\alpha^{2}\cdot\max\left(0,\frac{\beta}{2\lambda}-\frac{1}{2}\right)}$ and Lemma \ref{erfourier}, $c_{0}\colonequals \Phi(\alpha)=a$, $c_{1}\colonequals \frac{\beta}{\sqrt{2\pi}}e^{-\alpha^{2}/2}$, $\abs{c_{k}}\leq (\beta+\lambda)^{k}e^{\alpha^{2}\cdot\max\left(0,\frac{\beta}{2\lambda}-\frac{1}{2}\right)}$ for all $k\geq2$.  In conclusion, we have contradicted the maximality of $\Omega_{1},\ldots,\Omega_{m}$.  We conclude that \eqref{zero8} holds.

We now upgrade \eqref{zero8} by examining the level sets of
$$T_{\rho}(1_{\Omega})(x)-\epsilon\Phiz,\qquad\forall\,x\in\R^{\adimn}.$$
Fix $c\in\R$ and consider the level set
$$\Sigma\colonequals\{x\in\R^{\adimn}\colon T_{\rho}(1_{\Omega})(x)-\epsilon\Phiz=c \}.$$
This level set has Hausdorff dimension at most $\sdimn$ by \cite[Theorem 2.3]{chen98}.

From the Strong Unique Continuation Property for the heat equation \cite{lin90}, $T_{\rho}(1_{\Omega})(x)$ does not vanish to infinite order at any $x\in\R^{\adimn}$, so the argument of \cite[Lemma 1.9]{hardt89} (see \cite[Proposition 1.2]{lin94} and also \cite[Theorem 2.1]{chen98}) shows that in a neighborhood of each $x\in\Sigma$, $\Sigma$ can be written as a finite union of $C^{\infty}$ manifolds.  That is, there exists a neighborhood $U$ of $x$ and there exists an integer $k\geq1$ such that
\begin{flalign*}
U\cap\Sigma&=\cup_{i=1}^{k}\Big\{y\in U\colon D^{i}\Big[T_{\rho}(1_{\Omega})(x)-\epsilon\Phiz\Big]\neq 0,\\
&\qquad\qquad\qquad\quad D^{j}\Big[T_{\rho}(1_{\Omega})(x)-\epsilon\Phiz\Big]=0,\,\,\forall\,1\leq j\leq i-1\Big\}.
\end{flalign*}
Here $D^{i}$ denotes the array of all iterated partial derivatives of order $i\geq1$.  We therefore have
$$\Sigma\colonequals\redb\Omega\supset\Big\{x\in\R^{\adimn}\colon T_{\rho}(1_{\Omega})(x)-\epsilon\Phiz=c \Big\},$$
and the Lemma follows.
\end{proof}

From Lemma \ref{reglem} and Definition \ref{rbdef}, if $x\in\Sigma$, then the exterior pointing unit normal vector $N(x)\in\R^{\adimn}$ is well-defined on $\Sigma$, $\partial\Omega\setminus\Sigma$ has Hausdorff dimension at most $\sdimn-1$, and
\begin{equation}\label{zero11}
N(x)=\pm\frac{\overline{\nabla} T_{\rho}(1_{\Omega})(x)-\epsilon\Phizderiv}{\Big\|\overline{\nabla} T_{\rho}(1_{\Omega})(x)-\epsilon\Phizderiv\Big\|},\qquad\forall\,x\in\Sigma.
\end{equation}
In Lemma \ref{lemma7p} below we will show that the negative sign holds in \eqref{zero11} when $\Omega$ maximizes Problem \ref{prob2z}.


%

\begin{lemma}\label{gradlem}
Let $x\in\R^{\adimn}\setminus\{0\}$.  Fix $v\in\R^{\adimn}$.  Then
$$\overline{\nabla}\frac{\langle x,v\rangle}{\vnorm{x}}
=\frac{v}{\vnorm{x}}-\frac{\langle x,v\rangle \frac{x}{\vnorm{x}}}{\vnorm{x}^{2}}
=\frac{v-\langle \frac{x}{\vnorm{x}},v\rangle \frac{x}{\vnorm{x}}}{\vnorm{x}}
\equalscolon\frac{\mathrm{Proj}_{x^{\perp}}(v)}{\vnorm{x}}$$
\end{lemma}

\begin{lemma}\label{gradlemv2}
Let $x\in\R^{\adimn}$.  Fix $v,w\in\R^{\adimn}$.  Let $z=z^{(s)}$ for all $s\in(-1,1)$.  
\begin{flalign*}
&\frac{\d}{\d s}\Big|_{s=0}\frac{\langle v,w\rangle-\langle \frac{z}{\vnorm{z}},v\rangle \langle \frac{z}{\vnorm{z}}, w\rangle}{\vnorm{z}}\\
&\qquad=\Big[-\langle z',v\rangle\langle w,z\rangle  -\langle z',w\rangle\langle v,z\rangle  -\langle v,w\rangle\langle z,z'\rangle+\frac{3}{\vnorm{z}^{2}}\langle v,z\rangle\langle w,z\rangle\langle z,z'\rangle\Big]/\vnorm{z}^{3}.
\end{flalign*}
\end{lemma}
\begin{proof}
\begin{flalign*}
&\frac{\d}{\d s}\Big|_{s=0}\frac{\langle v,w\rangle-\langle \frac{z}{\vnorm{z}},v\rangle \langle \frac{z}{\vnorm{z}}, w\rangle}{\vnorm{z}}\\
&=\frac{-\langle z',\frac{v-\langle \frac{z}{\vnorm{z}},v\rangle \frac{z}{\vnorm{z}}}{\vnorm{z}}\rangle\langle \frac{z}{\vnorm{z}}, w\rangle
-\langle \frac{z}{\vnorm{z}},v\rangle \langle \frac{w-\langle \frac{z}{\vnorm{z}},w\rangle \frac{z}{\vnorm{z}}}{\vnorm{z}}, z'\rangle}{\vnorm{z}}
-\frac{\Big[\langle v,w\rangle-\langle \frac{z}{\vnorm{z}},v\rangle \langle \frac{z}{\vnorm{z}}, w\rangle\Big]\frac{\langle z,z'\rangle}{\vnorm{z}}}{\vnorm{z}^{2}}\\
&=\Big[-\langle z',v\rangle\langle w,z\rangle +\frac{1}{\vnorm{z}^{2}}\langle z',z\rangle\langle v,z\rangle\langle w,z\rangle   -\langle z',w\rangle\langle v,z\rangle +\frac{1}{\vnorm{z}^{2}}\langle w,z\rangle\langle z,z'\rangle\langle v,z\rangle\\
&\qquad\qquad -\langle v,w\rangle\langle z,z'\rangle+\frac{1}{\vnorm{z}^{2}}\langle v,z\rangle\langle w,z\rangle\langle z,z'\rangle\Big]/\vnorm{z}^{3}.
\end{flalign*}
\end{proof}

\section{First and Second Variation}

In this section, we recall some standard facts for variations of sets with respect to the Gaussian measure.
Our first and second variation formulas will be written in terms of $G$, as defined in \eqref{gdef}.


\begin{lemma}[\embolden{The First Variation}\,{\cite{chokski07}}; also {\cite[Lemma 3.1, Equation (7)]{heilman14}}]\label{latelemma3}
Let $X\in C_{0}^{\infty}(\R^{\adimn},\R^{\adimn})$.  Let $\Omega\subset\R^{\adimn}$ be a measurable set such that $\partial\Omega$ is a locally finite union of $C^{\infty}$ manifolds.  Let $\{\Omega^{(s)}\}_{s\in(-1,1)}$ be the corresponding variation of $\Omega$.  Then
\begin{equation}\label{Bone6}
\frac{\d}{\d s}\Big|_{s=0}\int_{\R^{\adimn}} 1_{\Omega^{(s)}}(y)G(x,y)\,\d y
=\int_{\partial \Omega}G(x,y)\langle X(y),N(y)\rangle \,\d y.
\end{equation}
\end{lemma}

The following Lemma is a consequence of \eqref{Bone6} and Lemma \ref{reglem}.  As usual, we denote $z^{(s)}\colonequals\int_{\Omega_{s}}x\gamma_{\adimn}(x)\,\d x\in\R^{\adimn}$, $z\colonequals z^{(0)}$, $z'\colonequals \int_{\Sigma}x\gamma_{\adimn}(x)\, \d x$, $\nu(z)\colonequals z/\vnorm{z}$ when $z\neq0$ and $\nu(0)\colonequals(1,0,\ldots,0)$.

\begin{lemma}[\embolden{The First Variation for Maximizers}]\label{firstvarmaxns}
Suppose $\Omega\subset\R^{\adimn}$ maximize Problem \ref{prob2z}.  Assume $z\neq0$.  Then there exists $c\in\R$ such that
$$T_{\rho}(1_{\Omega})(x)-\epsilon a_{0}\Big[\Phiz+\langle x,\zeta\,\rangle\Big]=c,\qquad\forall\,x\in\Sigma.$$
Moreover, $\abs{c}\leq2(1+\abs{\alpha})$.  Here $ a_{0}$ and $\zeta$ are defined in \eqref{atildef} and \eqref{ztildef}, respectively.
\end{lemma}
\begin{proof}
Denote $f(x)\colonequals\langle X(x),N(x)\rangle$ for all $x\in\Sigma$.  From Lemma \ref{latelemma3},
\begin{flalign*}
\frac{1}{2}\frac{\d}{\d s}\Big|_{s=0}\int_{\R^{\adimn}}1_{\Omega^{(s)}}(x)T_{\rho}1_{\Omega^{(s)}}(x)\gamma_{\adimn}(x)\,\d x
&=\int_{\Omega}G(x,y)\int_{\Sigma}\langle X(x),N(x)\rangle \,\d x \,\d y\\
&\!\!\stackrel{\eqref{oudef}\wedge\eqref{gdef}}{=}\int_{\Sigma}T_{\rho}(1_{\Omega})(x)f(x)\gamma_{\adimn}(x)\,\d x.
\end{flalign*}
Similarly, using $G(x,y)\colonequals\langle x,y\rangle\gamma_{\adimn}(x)\gamma_{\adimn}(y)$ for all $x,y\in\R^{\adimn}$ in Lemma \ref{latelemma3},
\begin{equation}\label{zderiv}
\begin{aligned}
&\frac{1}{2}\frac{\d}{\d s}\Big|_{s=0}\vnormf{z^{(s)}}^{2}
=\frac{1}{2}\frac{\d}{\d s}\Big|_{s=0}\Big\|\int_{\Omega^{(s)}}x\gamma_{\adimn}(x)\,\d x\Big\|^{2}\\
&\quad=\int_{\Omega}\langle x,y\rangle\gamma_{\adimn}(y)\int_{\Sigma}\langle X(x),N(x)\rangle\gamma_{\adimn}(x)\,\d x \,\d y
\stackrel{\eqref{zdef}}{=}\int_{\Sigma}\langle x,z\rangle f(x)\gamma_{\adimn}(x)\,\d x
=\langle z,z'\rangle.
\end{aligned}
\end{equation}
Define
\begin{equation}\label{ztildef}
\zeta\colonequals\int_{\Omega}\Phizdd\nu'(z)y\gamma_{\adimn}(y)\,\d y.
\end{equation}
Then using $G(x,y)\colonequals\Phi\Big(\frac{\beta\langle x,z\rangle-\alpha}{\sqrt{1-\beta^{2}}}\Big)\Phi\Big(\frac{\beta\langle y,z\rangle-\alpha}{\sqrt{1-\beta^{2}}}\Big)\gamma_{\adimn}(x)\gamma_{\adimn}(y)$ for all $x,y\in\R^{\adimn}$ in Lemma \ref{latelemma3} and the product rule,
\begin{flalign*}
&\frac{\d}{\d s}\Big|_{s=0}\pensnosq\\
&=\int_{\Sigma}\Phiz f(x)\gamma_{\adimn}(x)\,\d x
+\int_{\Omega}\Phizdd\frac{\d}{\d s}\Big|_{s=0}\langle y,\nu(z^{(s)})\rangle\gamma_{\adimn}(y)\,\d y \\
&=\int_{\Sigma}\Phiz f(x)\gamma_{\adimn}(x)\,\d x
+\int_{\Omega}\Phizdd\langle \nu'(z)y,z'\rangle\gamma_{\adimn}(y)\,\d y \\
&\stackrel{\eqref{zderiv}}{=}\int_{\Sigma}\Phiz f(x)\gamma_{\adimn}(x)\,\d x\\
&\qquad\qquad+\int_{\Sigma}  f(x)\Big\langle x,\int_{\Omega}\Phizdd \nu'(z)y\gamma_{\adimn}(y)\,\d y\Big\rangle\gamma_{\adimn}(x)\,\d x \\
&=\int_{\Sigma}\Phiz f(x)\gamma_{\adimn}(x)\,\d x
+\int_{\Sigma}  f(x)\langle x,\zeta\rangle\gamma_{\adimn}(x)\,\d x.
\end{flalign*}

\begin{equation}\label{atildef}
 a_{0}\colonequals\pennosq.
\end{equation}
In summary,
\begin{flalign*}
&\frac{1}{2}\frac{\d}{\d s}\Big|_{s=0}\Big[\int_{\R^{\adimn}}1_{\Omega^{(s)}}(x)T_{\rho}1_{\Omega^{(s)}}(x)\gamma_{\adimn}(x)\,\d x
-\epsilon\pens\Big]\\
&\qquad\qquad\qquad\qquad=\int_{\Sigma}\Big(T_{\rho}(1_{\Omega})(x)-\epsilon a_{0}\Big[\Phiz+\langle x,\zeta\rangle\Big]\Big)f(x)\gamma_{\adimn}(x)\,\d x.
\end{flalign*}
If $T_{\rho}(1_{\Omega})(x)-\epsilon a_{0}\Big[\Phiz+\langle x,\zeta\rangle\Big]$ is nonconstant, then we can construct $f$ with $\int_{\redb\Omega}f(x)\gamma_{\adimn}(x)dx=0$ to give a nonzero derivative:
$$\frac{\d}{\d s}\Big|_{s=0}\Big[\int_{\R^{\adimn}}1_{\Omega^{(s)}}(x)T_{\rho}1_{\Omega^{(s)}}(x)\gamma_{\adimn}(x)\,\d x
-\epsilon\pens\Big]\neq0,
$$
$$\frac{\d}{\d s}\Big|_{s=0}\gamma_{\adimn}(\Omega^{(s)})=0,
$$
contradicting the maximality of $\Omega$ (as in Lemma \ref{reglem} and \eqref{zero9.0}).  (The last equality used Lemma \ref{latelemma3} with $G(x,y)=\gamma_{\adimn}(y)$.)

Finally, in the case of a general vector field $X$ supported in $\Sigma$, we have
\begin{flalign*}
&\frac{1}{2}\frac{\d}{\d s}\Big|_{s=0}\Big[\int_{\R^{\adimn}}1_{\Omega^{(s)}}(x)T_{\rho}1_{\Omega^{(s)}}(x)\gamma_{\adimn}(x)\,\d x
-\epsilon\pens\Big]\\
&\qquad\qquad\qquad\qquad\qquad\qquad\qquad\qquad\qquad\qquad\qquad\qquad\qquad=\int_{\Sigma}cf(x)\,\d x.
\end{flalign*}
$$
\frac{\d}{\d s}\Big|_{s=0}\gamma_{\adimn}(\Omega^{(s)})=\int_{\Sigma}f(x)\gamma_{\adimn}(x)\,\d x.
$$
So, $\absf{\frac{1}{2}\frac{\d}{\d s}|_{s=0}\gamma_{\adimn}(\Omega^{(s)})}$ equals $0$ or $\absf{\int_{\Sigma}f(x)\gamma_{\adimn}(x)\,\d x}$.  Since $\Omega\subset\R^{\adimn}$ maximize Problem \ref{prob2z}, in either case, it follows that $\abs{c}\leq2(1+\abs{\alpha})$.
%
\end{proof}

\begin{theorem}[\embolden{General Second Variation Formula}, {\cite[Theorem 2.6]{chokski07}}; also {\cite[Theorem 1.10]{heilman15}}]\label{thm4}
Let $X\in C_{0}^{\infty}(\R^{\adimn},\R^{\adimn})$.  Let $\Omega\subset\R^{\adimn}$  be a measurable set such that $\partial\Omega$ is a locally finite union of $C^{\infty}$ manifolds.  Let $\{\Omega^{(s)}\}_{s\in(-1,1)}$ be the corresponding variation of $\Omega$.  Define $V$ as in \eqref{two9c}.  Then
\begin{flalign*}
&\frac{1}{2}\frac{\d^{2}}{\d s^{2}}\Big|_{s=0}\int_{\R^{\adimn}} \int_{\R^{\adimn}} 1_{\Omega^{(s)}}(y)G(x,y) 1_{\Omega^{(s)}}(x)\,\d x\d y\\
&\quad=\int_{\redA}\int_{\redA}G(x,y)\langle X(x),N(x)\rangle\langle X(y),N(y)\rangle \,\d x\d y
+\int_{\redA}\mathrm{div}(V(x,0)X(x))\langle X(x),N(x)\rangle \,\d x.
\end{flalign*}

\end{theorem}

\section{Noise Stability and the Calculus of Variations}\label{secnoise}

We now further refine the first and second variation formulas from the previous section.  The following formula follows by using $G(x,y)\colonequals\gamma_{\adimn}(x)\gamma_{\adimn}(y)$ $\forall$ $x,y\in\R^{\adimn}$ in Lemma \ref{latelemma3} and in Theorem \ref{thm4}.

\begin{lemma}[\embolden{Variations of Gaussian Volume}, {\cite{ledoux01}}]\label{lemma41}
Let $\Omega\subset\R^{\adimn}$  be a measurable set such that $\partial\Omega$ is a locally finite union of $C^{\infty}$ manifolds.  Let $X\in C_{0}^{\infty}(\R^{\adimn},\R^{\adimn})$.  Let $\{\Omega^{(s)}\}_{s\in(-1,1)}$ be the corresponding variation of $\Omega$.  Denote $f(x)\colonequals\langle X(x),N(x)\rangle$ for all $x\in\Sigma\colonequals \redb\Omega $.  Then
$$\frac{\d}{\d s}\Big|_{s=0}\gamma_{\adimn}(\Omega^{(s)})=\int_{\Sigma}f(x)\gamma_{\adimn}(x)\,\d x.$$
$$\frac{\d^{2}}{\d s^{2}}\Big|_{s=0}\gamma_{\adimn}(\Omega^{(s)})=\int_{\Sigma}(\mathrm{div}(X)-\langle X,x\rangle)f(x)\gamma_{\adimn}(x)\,\d x.$$
\end{lemma}

\begin{lemma}[\embolden{Extension Lemma for Existence of Volume-Preserving Variations}, {\cite[Lemma 3.9]{heilman18}}]\label{lemma27}
Let $X'\in C_{0}^{\infty}(\R^{\adimn},\R^{\adimn})$ be a vector field.  Define $f\colonequals\langle X',N\rangle\in C_{0}^{\infty}(\Sigma)$.  If
\begin{equation}\label{eight2}
\int_{\Sigma}f(x)\gamma_{\sdimn}(x)\,\d x=0,
\end{equation}
then $X'|_{\Sigma}$ can be extended to a vector field $X\in C_{0}^{\infty}(\R^{\adimn},\R^{\adimn})$ such that the corresponding variation $\{\Omega^{(s)}\}_{s\in(-1,1)}$ satisfy
$$\forall\,s\in(-1,1),\quad \gamma_{\adimn}(\Omega^{(s)})=\gamma_{\adimn}(\Omega).$$
\end{lemma}

\begin{lemma}\label{gpsd}
Define $G\colon\R^{\adimn}\times\R^{\adimn}\to\R$ by \eqref{gdef}.  Let $f\colon\Sigma\to\R$ be continuous and compactly supported.  Then
$$
\int_{\redA}\int_{\redA}G(x,y)f(x)f(y) \,\d x\d y\geq0.
$$
\end{lemma}
\begin{proof}
If $g\colon\R^{\adimn}\to\R$ is continuous and compactly supported, then it follows from \eqref{Height} that
$$
\int_{\redA}\int_{\redA}G(x,y)g(x)g(y) \,\d x\d y\geq0,
$$%
From Mercer's Theorem, this is equivalent to: $\forall$ $p\geq1$, for all $z^{(1)},\ldots,z^{(p)}\in\R^{n}$, for all $\beta_{1},\ldots,\beta_{p}\in\R$,
$$\sum_{i,j=1}^{p}\beta_{i}\beta_{j}G(z^{(i)},z^{(j)})\geq0.$$
In particular, this holds for all $z^{(1)},\ldots,z^{(p)}\in\partial\Omega\subset\R^{\adimn}$.  So, the positive semidefinite property carries over (by restriction) to $\partial\Omega$.
\end{proof}

%
%

\begin{lemma}[\embolden{Second Variation}]\label{lemma6}
Let $\Omega\subset\R^{\adimn}$  be a measurable set such that $\partial\Omega$ is a locally finite union of $C^{\infty}$ manifolds.  Let $X\in C_{0}^{\infty}(\R^{\adimn},\R^{\adimn})$.  Let $\{\Omega^{(s)}\}_{s\in(-1,1)}$ be the corresponding variation of $\Omega$.  Denote $f(x)\colonequals\langle X(x),N(x)\rangle$ for all $x\in\Sigma\colonequals \redb\Omega $.  Let $z\colonequals\int_{\Omega}x\gamma_{\adimn}(x)\,\d x$, $z^{(s)}\colonequals\int_{\Omega^{(s)}}x\gamma_{\adimn}(x)\,\d x$.  Define $ a_{0}$ and $\zeta$ by \eqref{atildef} and \eqref{ztildef}, respectively.  Then
\begin{equation}\label{four30}
\begin{aligned}
&\frac{1}{2}\frac{\d^{2}}{\d s^{2}}\Big|_{s=0}\Big[\int_{\Omega^{(s)}}\int_{\Omega^{(s)}} G(x,y) \,\d x\d y-\epsilon\pens\Big]\\
&=\int_{\redA}\int_{\redA}G(x,y)f(x)f(y)\,\d x\d y-\epsilon\Big(\int_{\redA}\Big[\Phiz +\langle x,\zeta\rangle\Big]f(x)\gamma_{\adimn}(x)\,\d x\Big)^{2}\\
&\qquad-\epsilon a_{0}\int_{\redA}\Big(\Phizdd\langle \nu'(z)x,z'\rangle+\langle x,\zeta'\rangle\Big)f(x)\gamma_{\adimn}(x)\,\d x\\
&\qquad+\int_{\redA}\Big\langle\overline{\nabla} T_{\rho}1_{\Omega}(x)-\epsilon a_{0}\Big(\Phizderiv+\zeta\Big),X(x)\Big\rangle f(x) \gamma_{\adimn}(x)\,\d x\\
&+\int_{\redA} \!\Big[T_{\rho}1_{\Omega}(x)-\epsilon a_{0}\Big[\Phiz+\langle x,\zeta\,\rangle\Big]\Big]\!\Big(\mathrm{div}(X(x))-\langle X(x),x\rangle\Big)f(x)\gamma_{\adimn}(x)\,\d x.
\end{aligned}
\end{equation}
\end{lemma}
\begin{proof}
For all $x\in\R^{\adimn}$, we have $V(x,0)\stackrel{\eqref{two9c}}{=}\int_{\Omega}G(x,y)\,\d y\stackrel{\eqref{oudef}}{=}\gamma_{\adimn}(x)T_{\rho}1_{\Omega}(x)$.  So, from Theorem \ref{thm4},
\begin{flalign*}
&\frac{1}{2}\frac{\d^{2}}{\d s^{2}}\Big|_{s=0}\int_{\R^{\adimn}}\int_{\R^{\adimn}} 1_{\Omega^{(s)}}(y)G(x,y) 1_{\Omega^{(s)}}(x)\,\d x\d y\\
&\qquad\qquad\qquad\qquad=\int_{\redA}\int_{\redA}G(x,y)\langle X(x),N(x)\rangle\langle X(y),N(y)\rangle \,\d x\d y\\
&\qquad\qquad\qquad\qquad\qquad+\int_{\redA}\Big(\sum_{i=1}^{\adimn}[T_{\rho}1_{\Omega}(x]\frac{\partial}{\partial x_{i}}X_{i}(x)-x_{i}T_{\rho}1_{\Omega}(x)X_{i}(x)\\
&\qquad\qquad\qquad\qquad\qquad\qquad\,\,+\frac{\partial}{\partial x_{i}}[T_{\rho}1_{\Omega}(x)]X_{i}(x)\Big)\langle X(x),N(x)\rangle \gamma_{\adimn}(x)\,\d x.
\end{flalign*}
Also, by Lemmas \ref{lemma41} and \ref{gradlem}
\begin{flalign*}
&\frac{1}{2}\frac{\d^{2}}{\d s^{2}}\Big|_{s=0}\pens\\
&\quad=\Big(\int_{\redA}\Big[\Phiz +\langle x,\zeta\rangle\Big]f(x)\gamma_{\adimn}(x)\,\d x\Big)^{2}\\
&\quad+ a_{0}\int_{\redA}\Big(\Phiz[\mathrm{div}(X)-\langle X,x\rangle]+\Phizdd\langle X,\nu(z)\rangle\Big)f(x)\gamma_{\adimn}(x)\,\d x\\
&\quad\qquad+ a_{0}\frac{\d}{\d s}\Big|_{s=0}\int_{\redA}\Big(\Phizs+\langle x,\zeta^{(s)}\rangle\Big)f(x)\gamma_{\adimn}(x)\,\d x\\
&\quad\qquad\qquad\,\,+ a_{0}\int_{\redA}\Big(\langle x,\zeta\rangle[\mathrm{div}(X)-\langle X,x\rangle]+\langle X,\zeta\rangle\Big)f(x)\gamma_{\adimn}(x)\,\d x\\
&\quad=\Big(\int_{\redA}\Phiz f(x)\gamma_{\adimn}(x)\,\d x\Big)^{2}\\
&\quad+ a_{0}\int_{\redA}\Big(\Phiz [\mathrm{div}(X)-\langle X,x\rangle]+\Phizdd\langle X,\nu(z)\rangle\Big)f(x)\gamma_{\adimn}(x)\,\d x\\
&\quad\qquad+ a_{0}\int_{\redA}\Big(\Phizdd\langle \nu'(z)x,z'\rangle+\langle x,\zeta'\rangle\Big)f(x)\gamma_{\adimn}(x)\,\d x\\
&\quad\qquad\qquad\,\,+ a_{0}\int_{\redA}\Big(\langle x,\zeta\rangle[\mathrm{div}(X)-\langle X,x\rangle]+\langle X,\zeta\rangle\Big)f(x)\gamma_{\adimn}(x)\,\d x.
\end{flalign*}
That is, \eqref{four30} holds.
\end{proof}

\begin{lemma}[\embolden{Volume Preserving Second Variation of Maximizers}]\label{lemma7p}
Suppose $\Omega\subset\R^{\adimn}$ maximizes Problem \ref{prob2z}.  Let $\epsilon<\frac{(1-\rho)^{2}z_{0}^{2}}{\rho10e^{\alpha^{2}\cdot\max\left(0,\frac{\beta}{\rho-\beta}-1\right)}}$, where $\vnorm{z}\geq z_{0}>0$.  Let $\{\Omega^{(s)}\}_{s\in(-1,1)}$ be the corresponding variation of $\Omega$.  Denote $f(x)\colonequals\langle X(x),N(x)\rangle$, $\forall$ $x\in\Sigma\colonequals \redb\Omega$, $z\colonequals\int_{\Omega}x\gamma_{\adimn}(x)\,\d x\in\R^{\adimn}$.  If
$$\int_{\Sigma}f(x)\gamma_{\adimn}(x)\,\d x=0,$$
Then there exists an extension of the vector field $X|_{\Sigma}$ such that the corresponding variation $\{\Omega^{(s)}\}_{s\in(-1,1)}$ satisfies
\begin{equation}\label{four32p}
\begin{aligned}
&\frac{1}{2}\frac{\d^{2}}{\d s^{2}}\Big|_{s=0}\Big[\int_{\Omega^{(s)}}\int_{\Omega^{(s)}} G(x,y) \,\d x\d y-\epsilon\pens\Big]\\
&=\int_{\redA}\int_{\redA}G(x,y)f(x)f(y)\,\d x\d y-\epsilon\Big(\int_{\redA}\Big[\Phiz+\langle x,\zeta\rangle\Big] f(x)\gamma_{\adimn}(x)\,\d x\Big)^{2}\\
&\qquad-\epsilon a_{0}\int_{\redA}\Big(\Phizdd\langle \nu'(z)x,z'\rangle+\langle x,\zeta'\rangle\Big)f(x)\gamma_{\adimn}(x)\,\d x\\
&\qquad-\int_{\redA}\Big\|\overline{\nabla} T_{\rho}1_{\Omega}(x)-\epsilon a_{0}\Big(\Phizderiv+\zeta\Big)\Big\| \abs{f(x)}^{2} \gamma_{\adimn}(x)\,\d x.
\end{aligned}
\end{equation}
Moreover,
\begin{equation}\label{nabeq2}
\begin{aligned}
&\overline{\nabla} T_{\rho}1_{\Omega}(x)-\epsilon a_{0}\Big(\Phizderiv+\zeta\Big)\\
&\qquad=-N(x)\Big\|\overline{\nabla} T_{\rho}1_{\Omega}(x)-\epsilon a_{0}\Big(\Phizderiv+\zeta\Big)\Big\|,
\qquad\forall\,x\in\Sigma.
\end{aligned}
\end{equation}
Lastly, $\vnormf{\overline{\nabla} T_{\rho}(1_{\Omega})(x)-\epsilon\Phizderiv}>0$ for all $x\in\Sigma$, except on a set of Hausdorff dimension at most $\sdimn-1$.
\end{lemma}
\begin{proof}
From Lemma \ref{firstvarmaxns}, $T_{\rho}1_{\Omega}(x)-\epsilon a_{0}\Big[\Phiz+\langle x,\zeta\,\rangle\Big]$ is constant for all $x\in\Sigma$.  So, from Lemma \ref{lemma41} and Lemma \ref{lemma27}, the last term in \eqref{four30} vanishes, i.e.
\begin{flalign*}
&\frac{1}{2}\frac{\d^{2}}{\d s^{2}}\Big|_{s=0}\Big[\int_{\Omega^{(s)}}\int_{\Omega^{(s)}} G(x,y) \,\d x\d y-\epsilon\pens\Big]\\
&=\int_{\redA}\int_{\redA}G(x,y)f(x)f(y)\,\d x\d y-\epsilon\Big(\int_{\redA}\Big[\Phiz +\langle x,\zeta\rangle\Big]f(x)\gamma_{\adimn}(x)\,\d x\Big)^{2}\\
&\qquad-\epsilon a_{0}\int_{\redA}\Big(\Phizdd\langle \nu'(z)x,z'\rangle+\langle x,\zeta'\rangle\Big)f(x)\gamma_{\adimn}(x)\,\d x\\
&\qquad+\int_{\redA}\Big\langle\overline{\nabla} T_{\rho}1_{\Omega}(x)-\epsilon a_{0}\Big(\Phizderiv+\zeta\Big),X(x)\Big\rangle f(x) \gamma_{\adimn}(x)\,\d x.
\end{flalign*}
(Here $\overline{\nabla}$ denotes the gradient in $\R^{\adimn}$.)  Since $T_{\rho}1_{\Omega}(x)-\epsilon a_{0}\Big[\Phiz+\langle x,\zeta\,\rangle\Big]$ is constant for all $x\in\partial\Omega$ by Lemma \ref{firstvarmaxns}, $\overline{\nabla} T_{\rho}1_{\Omega}(x)-\epsilon a_{0}\Big(\Phizderiv+\zeta\Big)$ is parallel to $N(x)$ for all $x\in\partial\Omega$.  That is, for all $x\in\Sigma$,
\begin{equation}\label{nabeq}
\begin{aligned}
&\overline{\nabla} T_{\rho}1_{\Omega}(x)-\epsilon a_{0}\Big(\Phizderiv+\zeta\Big)\\
&\qquad\qquad=\pm\Big\|\overline{\nabla} T_{\rho}1_{\Omega}(x)-\epsilon a_{0}\Big(\Phizderiv+\zeta\Big)\Big\|N(x).
\end{aligned}
\end{equation}
In fact, we must have a negative sign in \eqref{nabeq}, otherwise we could find a vector field $X$ supported near $x\in\partial\Omega$ such that \eqref{nabeq} has a positive sign, and then since the second variation is a positive semidefinite function of $f$ by Lemmas \ref{gpsd} and Lemma \ref{techlem} below, we would have
\begin{flalign*}
&\frac{1}{2}\frac{\d^{2}}{\d s^{2}}\Big|_{s=0}\int_{\R^{\adimn}}\int_{\R^{\adimn}} 1_{\Omega^{(s)}}(y)G(x,y) 1_{\Omega^{(s)}}(x)\,\d x\d y\\
&\quad\qquad\geq\int_{\redA}\langle\overline{\nabla} T_{\rho}1_{\Omega}(x)-\epsilon a_{0}\Big(\Phizderiv+\zeta\Big),X(x)\rangle \langle X(x),N(x)\rangle \gamma_{\adimn}(x)\,\d x>0,
\end{flalign*}
a contradiction.  In summary,
\begin{flalign*}
&\frac{1}{2}\frac{\d^{2}}{\d s^{2}}\Big|_{s=0}\Big[\int_{\Omega^{(s)}}\int_{\Omega^{(s)}} G(x,y) \,\d x\d y-\epsilon\pens\Big]\\
&=\int_{\redA}\int_{\redA}G(x,y)f(x)f(y)\,\d x\d y-\epsilon\Big(\int_{\redA}\Big[\Phiz +\langle x,\zeta\rangle\Big] f(x)\gamma_{\adimn}(x)\,\d x\Big)^{2}\\
&\qquad-\epsilon a_{0}\int_{\redA}\Big(\Phizdd\langle \nu'(z)x,z'\rangle+\langle x,\zeta'\rangle\Big)f(x)\gamma_{\adimn}(x)\,\d x\\
&\qquad-\int_{\redA}\Big\|\overline{\nabla} T_{\rho}1_{\Omega}(x)-\epsilon a_{0}\Big(\Phizderiv+\zeta\Big)\Big\| \abs{f(x)}^{2} \gamma_{\adimn}(x)\,\d x.
\end{flalign*}

This same argument implies the final assertion, that $\vnormf{\overline{\nabla} T_{\rho}(1_{\Omega})(x)-\epsilon\Phizderiv}>0$ for all $x\in\Sigma$, except on a set of Hausdorff dimension at most $\sdimn-1$.  More specifically, if $\vnormf{\overline{\nabla} T_{\rho}(1_{\Omega})(x)-\epsilon\Phizderiv}=0$ on a set of positive Hausdorff measure on $\Sigma$, then we let $f$ be supported on this set with $\int_{\Sigma}f(x)\gamma_{\adimn}(x)\,\d x=0$, then Lemma \ref{techlem} and Mehler's formula \eqref{Height} implies that $f$ has positive second variation, a contradiction.
\end{proof}

The following technical lemma shows that the second variation formula \eqref{four32p} is a positive semidefinite function of $f$, when $\epsilon,\beta$ are sufficiently small.

\begin{lemma}\label{techlem}
Define
$$\theta\colonequals\epsilon\rho\frac{10e^{\alpha^{2}\cdot\max\left(0,\frac{\beta}{\rho-\beta}-1\right)}}{(1-\rho)\vnorm{z}^{2}}.$$
\begin{flalign*}
&\frac{1}{2}\frac{\d^{2}}{\d s^{2}}\Big|_{s=0}\Big[\int_{\R^{\adimn}}1_{\Omega^{(s)}}(x)T_{\rho}1_{\Omega^{(s)}}(x)\gamma_{\adimn}(x)\,\d x
-\epsilon\pens\Big]\\
&\qquad\geq (1-\theta)\int_{\Sigma}\int_{\Sigma}f(x)G(x,y)f(y)\,\d x\d y\\
&\qquad-\int_{\Sigma}\Big\|\overline{\nabla} T_{\rho}1_{\Omega}(x)-\epsilon a_{0}\Big(\Phizderiv+\zeta\Big)\Big\| \abs{f(x)}^{2}\gamma_{\adimn}(x)\, \d x.
\end{flalign*}
\end{lemma}
\begin{proof}
From Lemma \ref{lemma7p},
\begin{equation}\label{four32pv8}
\begin{aligned}
&\frac{1}{2}\frac{\d^{2}}{\d s^{2}}\Big|_{s=0}\Big[\int_{\Omega^{(s)}}\int_{\Omega^{(s)}} G(x,y) \,\d x\d y-\epsilon\pens\Big]\\
&=\int_{\redA}\Big(S(f)(x)\\
&\qquad-\epsilon \Big[\int_{\redA}\Big[\Phiz +\langle y,\zeta\rangle\Big]f(y)\gamma_{\adimn}(y)\,\d y\Big]\Big[\Phiz+\langle x,\zeta\rangle\Big]\\
&\qquad-\Big\|\overline{\nabla} T_{\rho}1_{\Omega}(x)-\epsilon a_{0}\Big(\Phizderiv+\zeta\Big)\Big\| f(x) \Big) f(x)\gamma_{\adimn}(x)\, \d x\\
&\quad-\epsilon a_{0}\int_{\redA}\Big(\Phizdd\langle \nu'(z)x,z'\rangle+\langle x,\zeta'\rangle\Big)f(x)\gamma_{\adimn}(x)\,\d x.
\end{aligned}
\end{equation}
Here
\begin{equation}\label{zpri}
z'=\int_{\Sigma}yf(y)\gamma_{\adimn}(y)\,\d y.
\end{equation}
Also
\begin{equation}\label{ztilpri}
\begin{aligned}
\zeta'
&=\frac{\d}{\d s}\Big|_{s=0}\int_{\Omega^{(s)}}\Phizdds \nu'(z^{(s)})y\gamma_{\adimn}(y)\,\d y\\
&=\int_{\Sigma}\Phizddy\nu'(z)y f(y)\gamma_{\adimn}(y)\,\d y\\
&\qquad+\int_{\Omega}\frac{\beta^{2}[\alpha-\beta\langle y,\nu(z)\rangle] e^{-\frac{[\beta\langle y,\nu(z)\rangle-\alpha]^{2}}{2(1-\beta^{2})}}}{(1-\beta^{2})^{3/2}}
\langle \nu'(z)y,z'\rangle \nu'(z)y\gamma_{\adimn}(y)\,\d y\\
&\qquad+\int_{\Omega}\Phizdd\nu''(z) y\gamma_{\adimn}(y)\,\d y.
\end{aligned}
\end{equation}
Using Lemma \ref{gradlemv2}, we have
$$\nu'(z)y=\frac{ \mathrm{Proj}_{z^{\perp}}(y) }{\vnorm{z}}.$$
\begin{equation}\label{nudp}
\nu''(z) y=\frac{1}{\vnorm{z}^{3}}\Big(-\langle z',y\rangle z  -\langle y,z\rangle z'  -\langle z,z'\rangle y+\frac{3}{\vnorm{z}^{2}}\langle y,z\rangle\langle z,z'\rangle z\Big).
\end{equation}

We bound the second term in \eqref{four32pv8} using Lemmas \ref{gausfourier} and \ref{erfourier}.  Since
$$\zeta\stackrel{\eqref{ztildef}}{\colonequals}\int_{\Omega}\Phizdd\nu'(z)y\gamma_{\adimn}(y)\,\d y,$$
we have
\begin{equation}\label{ztibd}
\vnormf{\zeta}\leq\frac{\beta}{\vnorm{z}}.
\end{equation}
So, Lemma  \ref{erfourier} with $\lambda\colonequals \rho-\beta$ implies that $\exists$ $c_{1},c_{2},\ldots\in\R$ with
\begin{equation}\label{ckbd}
\abs{c_{k}}\leq \rho^{k}e^{\alpha^{2}\cdot\max\left(0,\frac{\beta}{2[\rho-\beta]}-\frac{1}{2}\right)},
\end{equation}
for all $k\geq2$, with $c_{0}\colonequals\Phi(\alpha)=a$, $c_{1}\colonequals \frac{\beta}{\sqrt{2\pi}}e^{-\alpha^{2}/2}$ such that
\begin{equation}\label{phicexp}
\Phiz
=\sum_{k=0}^{\infty}c_{k}h_{k}(y)\sqrt{k!},\qquad\forall\,y\in\R
\end{equation}
We have (using also $\int_{\Sigma}f(y)\gamma_{\adimn}\,\d y=0$ and the Cauchy-Schwarz inequality for discrete sequences of real numbers),
\begin{equation}\label{tobd2}
\begin{aligned}
&\Big(\int_{\redA}\Big[\Phiz+\langle y,\zeta\rangle\Big] f(y)\gamma_{\adimn}(y)\,\d y\Big)^{2}\\
&\stackrel{\eqref{ztibd}\wedge\eqref{phicexp}}{\leq}\Big(\int_{\redA}\Big[\frac{\beta}{\vnorm{z}}h_{1}(y)+\sum_{k=0}^{\infty}c_{k}h_{k}(\langle y,\nu(z)\rangle)\sqrt{k!}\Big] f(y)\gamma_{\adimn}(y)\,\d y\Big)^{2}\\
&\leq\Big(\sum_{j=1}^{\infty}\Big|c_{j}+1_{\{j=1\}}\frac{\beta}{\vnorm{z}}\Big|\Big)
\Big(\sum_{k=1}^{\infty}\Big|c_{k}+1_{\{k=1\}}\frac{\beta}{\vnorm{z}}\Big|\Big[\int_{\redA}h_{k}(\langle y,\nu(z)\rangle)\sqrt{k!} f(y)\gamma_{\adimn}(y)\,\d y\Big]^{2}\Big)\\
&\stackrel{\eqref{ckbd}}{\leq}\frac{\rho e^{\alpha^{2}\cdot\max\left(0,\frac{\beta}{\rho-\beta}-1\right)}}{(1-\rho)}
\Big(\sum_{k=1}^{\infty}\rho^{k}\Big[\int_{\redA}h_{k}(\langle y,\nu(z)\rangle)\sqrt{k!} f(y)\gamma_{\adimn}(y)\,\d y\Big]^{2}\Big).
 \end{aligned}
 \end{equation}

 We now bound the last term of \eqref{four32pv8}.  From \eqref{gausfourier}, $\exists$ $c_{1}',c_{2}',\ldots\in\R$ with
$$\abs{c_{k}'}\leq \rho^{k}e^{\alpha^{2}\cdot\max\left(0,\frac{\beta}{2[\rho-\beta]}-\frac{1}{2}\right)},$$
for all $k\geq1$ and $c_{0}'\colonequals e^{-\alpha^{2}/2}$ such that
$$(1-\beta^{2})^{-1/2}e^{-\frac{[\beta x-\alpha]^{2}}{2(1-\beta^{2})}}
=\sum_{k=0}^{\infty}c_{k}'h_{k}(x)\sqrt{k!},\qquad\forall\,x\in\R.$$
So, using the Cauchy-Schwarz and AMGM inequalities,
\begin{equation}\label{bd9}
 \begin{aligned}
& \Big\langle\int_{\Sigma}\Phizdd\nu'(z)x f(x)\gamma_{\adimn}(x)\,\d x,\,\, z'\Big\rangle\\
&\stackrel{\eqref{zpri}}{=}\Big\langle\int_{\Sigma}\Phizdd\nu'(z)x f(x)\gamma_{\adimn}(x)\,\d x,\,\, \int_{\Sigma} y f(y)\gamma_{\adimn}(y)\,\d y\Big\rangle\\
&\leq \frac{\beta}{\vnorm{z}}\Big\|\frac{1}{\sqrt{1-\beta^{2}}}\int_{\Sigma}e^{-\frac{[\beta\langle \nu(z),x\rangle-\alpha]^{2}}{2(1-\beta^{2})}}\mathrm{Proj}_{z^{\perp}}(x)f(x)\gamma_{\adimn}(x)\,\d  x\Big\|^{2}\\
&\qquad\qquad\qquad+\frac{\beta}{\vnorm{z}}\Big\|\int_{\Sigma}xf(x)\gamma_{\adimn}(x)\,\d x\Big\|^{2}\\
&\leq  \frac{\beta e^{\alpha^{2}\cdot\max\left(0,\frac{\beta}{\rho-\beta}-1\right)}}{\vnorm{z}(1-\rho)}\Big(\sum_{k=0}^{\infty}\rho^{k}\\
&\qquad\qquad\cdot\sum_{\substack{j=1,\ldots,\sdimn\colon e_{1},\ldots,e_{\sdimn}\,\,\mathrm{is}\\ \mathrm{an}\,\mathrm{orthonormal}\,\mathrm{basis}\\ \mathrm{of}\,\, z^{\perp}}}\Big[\int_{\redA}\langle y,e_{j}\rangle h_{k}(\langle y,\nu(z)\rangle)\sqrt{k!} f(y)\gamma_{\adimn}(y)\,\d y\Big]^{2}\Big)\\
&\qquad+\frac{\beta}{\vnorm{z}}\Big\|\int_{\Sigma}xf(x)\gamma_{\adimn}(x)\,\d x\Big\|^{2}.
 \end{aligned}
 \end{equation}
 It then remains to bound the last term of \eqref{four32pv8}:

 \begin{equation}\label{lastbd}
  \begin{aligned}
& \Big\langle\int_{\Sigma}x f(x)\gamma_{\adimn}(x)\,\d x,\,\, \zeta'\Big\rangle\\
&\stackrel{\eqref{ztilpri}}{=}
\Big\langle\int_{\Sigma}x f(x)\gamma_{\adimn}(x)\,\d x,\,\, \int_{\Sigma}\Phizdd\nu'(z)y f(y)\gamma_{\adimn}(y)\,\d y\Big\rangle\\
&\qquad+\Big\langle\int_{\Sigma}x f(x)\gamma_{\adimn}(x)\,\d x,\,\, \int_{\Omega}\frac{\beta^{2}[\alpha-\beta\langle y,\nu(z)\rangle] e^{-\frac{[\beta\langle y,\nu(z)\rangle-\alpha]^{2}}{2(1-\beta^{2})}}}{(1-\beta^{2})^{3/2}}\\
&\qquad\qquad\qquad\qquad\qquad\qquad\qquad\qquad\qquad\cdot\vnorm{z}^{-2}\langle\mathrm{Proj}_{z^{\perp}}(y),z'\rangle\mathrm{Proj}_{z^{\perp}}(y)\gamma_{\adimn}(y)\,\d y\Big\rangle\\
&\qquad+\Big\langle\int_{\Sigma}x f(x)\gamma_{\adimn}(x)\,\d x,\,\, \int_{\Omega}\Phizdd\nu''(z) y\gamma_{\adimn}(y)\,\d y\Big\rangle.
 \end{aligned}
 \end{equation}
 The first term of \eqref{lastbd} is bounded by \eqref{bd9}.  The second term of \eqref{lastbd} is bounded similarly as

 \begin{flalign*}
& \Big\langle\int_{\Sigma}x f(x)\gamma_{\adimn}(x)\,\d x,\\
&\qquad\qquad\qquad \int_{\Omega}\frac{\beta^{2}[\alpha-\beta\langle y,\nu(z)\rangle] e^{-\frac{[\beta\langle y,\nu(z)\rangle-\alpha]^{2}}{2(1-\beta^{2})}}}{(1-\beta^{2})^{3/2}}\vnorm{z}^{-2}\langle\mathrm{Proj}_{z^{\perp}}(y),z'\rangle\mathrm{Proj}_{z^{\perp}}(y)\gamma_{\adimn}(y)\,\d y\Big\rangle\\
&\leq
\frac{\beta^{2}}{\vnorm{z}^{2}}\Big\|\int_{\Sigma}x f(x)\gamma_{\adimn}(x)\,\d x\Big\|^{2}\\
&\quad+\frac{\beta^{2}}{\vnorm{z}^{2}}\Big\|\Big\langle\int_{\Omega}\frac{[\alpha-\beta\langle y,\nu(z)\rangle]}{(1-\beta^{2})^{3/2}} e^{-\frac{[\beta\langle y,\nu(z)\rangle-\alpha]^{2}}{2(1-\beta^{2})}}\mathrm{Proj}_{z^{\perp}}(y)\mathrm{Proj}_{z^{\perp}}(y)\gamma_{\adimn}(y)\,\d y,\,\, z'\Big\rangle\Big\|^{2}\\
&\leq
\frac{\beta^{2}}{\vnorm{z}^{2}}\Big\|\int_{\Sigma}x f(x)\gamma_{\adimn}(x)\,\d x\Big\|^{2}\cdot\\
&\qquad\qquad\qquad\Big(1+\Big[\int_{\Omega}\frac{[\alpha-\beta\langle y,\nu(z)\rangle]}{(1-\beta^{2})^{3/2}} e^{-\frac{[\beta\langle y,\nu(z)\rangle-\alpha]^{2}}{2(1-\beta^{2})}}\mathrm{Proj}_{z^{\perp}}(y)\mathrm{Proj}_{z^{\perp}}(y)\gamma_{\adimn}(y)\,\d y\Big]^{2}\Big).
 \end{flalign*}

 The last term of \eqref{lastbd} is bounded in a similar way to the second term, by substituting \eqref{nudp}.
\begin{flalign*}
&\Big\langle\int_{\Sigma}x f(x)\gamma_{\adimn}(x)\,\d x,\,\, \int_{\Omega}\Phizddy\nu''(z) y\gamma_{\adimn}(y)\,\d y\Big\rangle\\
&\qquad=-\vnorm{z}^{-3}\Big\langle\int_{\Sigma}x f(x)\gamma_{\adimn}(x)\,\d x,\,\, \int_{\Omega}\Phizddy\langle z',y\rangle z\gamma_{\adimn}(y)\,\d y\Big\rangle\\
&\qquad\qquad-\vnorm{z}^{-3}\Big\langle\int_{\Sigma}x f(x)\gamma_{\adimn}(x)\,\d x,\,\, \int_{\Omega}\Phizddy\langle y,z\rangle z'\gamma_{\adimn}(y)\,\d y\Big\rangle\\
&\qquad\qquad-\vnorm{z}^{-3}\Big\langle\int_{\Sigma}x f(x)\gamma_{\adimn}(x)\,\d x,\,\, \int_{\Omega}\Phizddy\langle z,z'\rangle y\gamma_{\adimn}(y)\,\d y\Big\rangle\\
&\qquad\qquad+3\vnorm{z}^{-5}\Big\langle\int_{\Sigma}x f(x)\gamma_{\adimn}(x)\,\d x,\,\, \int_{\Omega}\Phizddy\langle y,z\rangle\langle z,z'\rangle z\gamma_{\adimn}(y)\,\d y\Big\rangle\\
&\qquad\stackrel{\eqref{zpri}}{=}-\vnorm{z}^{-3}\Big\langle\int_{\Sigma}x f(x)\gamma_{\adimn}(x)\,\d x,\,\,z\Big\rangle  \int_{\Omega}\Phizddy\langle z',y\rangle \gamma_{\adimn}(y)\,\d y\\
&\qquad\qquad-\vnorm{z}^{-3}\Big\|\int_{\Sigma}x f(x)\gamma_{\adimn}(x)\,\d x\Big\|^{2} \int_{\Omega}\Phizddy\langle y,z\rangle \gamma_{\adimn}(y)\,\d y\\
&\qquad\qquad-\vnorm{z}^{-3}\langle z,z'\rangle\Big\langle\int_{\Sigma}x f(x)\gamma_{\adimn}(x)\,\d x,\,\, \int_{\Omega}\Phizddy y\gamma_{\adimn}(y)\,\d y\Big\rangle\\
&\qquad\qquad+3\vnorm{z}^{-5}\langle z,z'\rangle\Big\langle\int_{\Sigma}x f(x)\gamma_{\adimn}(x)\,\d x,\,\,z\Big\rangle \int_{\Omega}\Phizddy\langle y,z\rangle \gamma_{\adimn}(y)\,\d y
\end{flalign*}
\begin{equation}\label{bigbd}
\begin{aligned}
&\qquad\stackrel{\eqref{zpri}}{=}-\vnorm{z}^{-3}\Big\langle\int_{\Sigma}x f(x)\gamma_{\adimn}(x)\,\d x,\,\,z\Big\rangle  \Big\langle z',\int_{\Omega}\Phizddy y\gamma_{\adimn}(y)\,\d y\Big\rangle\\
&\qquad\qquad-\vnorm{z}^{-3}\Big\|\int_{\Sigma}x f(x)\gamma_{\adimn}(x)\,\d x\Big\|^{2} \int_{\Omega}\Phizddy\langle y,z\rangle \gamma_{\adimn}(y)\,\d y\\
&\qquad\qquad-\vnorm{z}^{-3}\langle z,z'\rangle\Big\langle\int_{\Sigma}x f(x)\gamma_{\adimn}(x)\,\d x,\,\, \int_{\Omega}\Phizddy y\gamma_{\adimn}(y)\,\d y\Big\rangle\\
&\qquad\qquad+3\vnorm{z}^{-5}\langle z,z'\rangle\Big\langle\int_{\Sigma}x f(x)\gamma_{\adimn}(x)\,\d x,\,\,z\Big\rangle \int_{\Omega}\Phizddy\langle y,z\rangle \gamma_{\adimn}(y)\,\d y\\
&\qquad\leq6\vnorm{z}^{-2}\Big\|\int_{\Sigma}x f(x)\gamma_{\adimn}(x)\,\d x\Big\|^{2}\Big\|\int_{\Omega}\Phizddy y\gamma_{\adimn}(y)\,\d y\Big\|\\
&\qquad\leq6\beta\vnorm{z}^{-1}\Big\|\int_{\Sigma}x f(x)\gamma_{\adimn}(x)\,\d x\Big\|^{2}.
\end{aligned}
\end{equation}
Combining all above estimates, we get
\begin{equation}\label{nin3v2}
\begin{aligned}
&\frac{1}{2}\frac{\d^{2}}{\d s^{2}}\Big|_{s=0}\Big[\int_{\R^{\adimn}}1_{\Omega^{(s)}}(x)T_{\rho}1_{\Omega^{(s)}}(x)\gamma_{\adimn}(x)\,\d x
-\epsilon\pens\\
&\qquad\geq\int_{\Sigma}\int_{\Sigma}f(x)G(x,y)f(y)\,\d x\d y\\
&\qquad-\epsilon\frac{\rho e^{\alpha^{2}\cdot\max\left(0,\frac{\beta}{\rho-\beta}-1\right)}}{(1-\rho)}
\Big(\sum_{k=1}^{\infty}\rho^{k}\Big[\int_{\redA}h_{k}(\langle y,\nu(z)\rangle)\sqrt{k!} f(y)\gamma_{\adimn}(y)\,\d y\Big]^{2}\Big)\\
&-\epsilon\frac{\beta e^{\alpha^{2}\cdot\max\left(0,\frac{\beta}{\rho-\beta}-1\right)}}{\vnorm{z}(1-\rho)}\Big(\sum_{k=0}^{\infty}\rho^{k}
\cdot\sum_{\substack{j=1,\ldots,\sdimn\colon e_{1},\ldots,e_{\sdimn}\,\,\mathrm{is}\\ \mathrm{an}\,\mathrm{orthonormal}\,\mathrm{basis}\\ \mathrm{of}\,\, z^{\perp}}}\Big[\int_{\redA}\langle y,e_{j}\rangle h_{k}(\langle y,\nu(z)\rangle)\sqrt{k!} f(y)\gamma_{\adimn}(y)\,\d y\Big]^{2}\Big)\\
&\qquad-\epsilon\Big\|\int_{\Sigma}xf(x)\gamma_{\adimn}(x)\,\d x\Big\|^{2}\Big(2\frac{\beta}{\vnorm{z}}+\frac{\beta}{\vnorm{z}^{2}}\Big[1+\beta\Big]+\frac{6\beta}{\vnorm{z}}\Big)\\
&\qquad-\Big\|\overline{\nabla} T_{\rho}1_{\Omega}(x)-\epsilon a_{0}\Big(\Phizderiv+\zeta\Big)\Big\| \abs{f(x)}^{2}\gamma_{\adimn}(x)\, \d x.
\end{aligned}
\end{equation}
Then Mehler's formula \eqref{Height} and \eqref{gdef} conclude the proof.
\end{proof}

\section{Almost Eigenfunctions of the Second Variation}


Let $\Sigma\colonequals\redb\Omega$.  For any bounded measurable $f\colon\Sigma\to\R$, define the following function (if it exists):
\begin{equation}\label{sdef}
S(f)(x)\colonequals (1-\rho^{2})^{-(\adimn)/2}(2\pi)^{-(\adimn)/2}\int_{\Sigma}f(y)e^{-\frac{\vnorm{y-\rho x}^{2}}{2(1-\rho^{2})}}\,\d y,\qquad\forall\,x\in\Sigma.
\end{equation}

\begin{lemma}[\embolden{Key Lemma, Translations as Almost Eigenfunctions}]\label{treig}
Let $\Omega$ maximize Problem \ref{prob2z}.  Let $v\in\R^{\adimn}$.  Define $z\colonequals\int_{\Omega}x\gamma_{\adimn}(x)dx\in\R^{\adimn}$.  Then
\begin{flalign*}
S(\langle v,N\rangle)(x)&=
\langle v,N(x)\rangle\frac{1}{\rho}\Big\|\overline{\nabla} T_{\rho}1_{\Omega}(x)-\epsilon a_{0}\Big(\Phizderiv+\zeta\Big)\Big\|\\
&\qquad\qquad\qquad\qquad\qquad\qquad-\frac{\epsilon a_{0}}{\rho}\Big\langle v, \Phizderiv+\zeta\Big\rangle,\qquad\forall\,x\in\Sigma.
\end{flalign*}
\end{lemma}
\begin{proof}
Since $T_{\rho}1_{\Omega}(x)-\epsilon a_{0}\Big[\Phiz+\langle x,\zeta\,\rangle\Big]$ is constant for all $x\in\partial\Omega$ by Lemma \ref{firstvarmaxns}, $\overline{\nabla} T_{\rho}1_{\Omega}(x)-\epsilon a_{0}\Big(\Phizderiv+\zeta\Big)$ is parallel to $N(x)$ for all $x\in\partial\Omega$.  That is \eqref{nabeq2} holds:
\begin{equation}\label{firstve}
\begin{aligned}
&\overline{\nabla} T_{\rho}1_{\Omega}(x)-\epsilon a_{0}\Big(\Phizderiv+\zeta\Big)\\
&\qquad=-N(x)\Big\|\overline{\nabla} T_{\rho}1_{\Omega}(x)-\epsilon a_{0}\Big(\Phizderiv+\zeta\Big)\Big\|,
\qquad\forall\,x\in\Sigma.
\end{aligned}
\end{equation}
From Definition \ref{oudef}, and then using the divergence theorem,
\begin{equation}\label{gre}
\begin{aligned}
\langle v,\overline{\nabla} T_{\rho}1_{\Omega}(x)\rangle
&=(1-\rho^{2})^{-(\adimn)/2}(2\pi)^{-(\adimn)/2}\Big\langle v,\int_{\Omega} \overline{\nabla}_{x}e^{-\frac{\vnorm{y-\rho x}^{2}}{2(1-\rho^{2})}}\,\d y\Big\rangle\\
&=(1-\rho^{2})^{-(\adimn)/2}(2\pi)^{-(\adimn)/2}\frac{\rho}{1-\rho^{2}}\int_{\Omega} \langle v,\,y-\rho x\rangle e^{-\frac{\vnorm{y-\rho x}^{2}}{2(1-\rho^{2})}}\,\d y\\
&=-(1-\rho^{2})^{-(\adimn)/2}(2\pi)^{-(\adimn)/2}\rho\int_{\Omega} \mathrm{div}_{y}\Big(ve^{-\frac{\vnorm{y-\rho x}^{2}}{2(1-\rho^{2})}}\Big)\,\d y\\
&=-(1-\rho^{2})^{-(\adimn)/2}(2\pi)^{-(\adimn)/2}\rho\int_{\Sigma}\langle v,N(y)\rangle e^{-\frac{\vnorm{y-\rho x}^{2}}{2(1-\rho^{2})}}\,\d y\\
&\stackrel{\eqref{sdef}}{=}-\rho\, S(\langle v,N\rangle)(x).
\end{aligned}
\end{equation}
Therefore,
\begin{flalign*}
&\langle v,N(x)\rangle\Big\|\overline{\nabla} T_{\rho}1_{\Omega}(x)-\epsilon a_{0}\Big(\Phizderiv+\zeta\Big)\Big\|\\
&\qquad\stackrel{\eqref{firstve}}{=}-\Big\langle v,\overline{\nabla} T_{\rho}1_{\Omega}(x)-\epsilon a_{0}\Big(\Phizderiv+\zeta\Big\rangle\\
&\qquad\stackrel{\eqref{gre}}{=}\rho\, S(\langle v,N\rangle)(x)+\epsilon a_{0}\Big\langle v, \Phizderiv+\zeta\Big\rangle.
\end{flalign*}
\end{proof}
\begin{remark}\label{drk}
To justify the use of the divergence theorem in \eqref{gre}, let $r>0$ and note that we can differentiate under the integral sign of  $T_{\rho}1_{\Omega\cap B(0,r)}(x)$ to get
\begin{equation}\label{grep}
\begin{aligned}
\overline{\nabla} T_{\rho}1_{\Omega\cap B(0,r)}(x)
&=(1-\rho^{2})^{-(\adimn)/2}(2\pi)^{-(\adimn)/2}\Big\langle v,\int_{\Omega\cap B(0,r)} \overline{\nabla}_{x}e^{-\frac{\vnorm{y-\rho x}^{2}}{2(1-\rho^{2})}}\,\d y\Big\rangle\\
&=(1-\rho^{2})^{-(\adimn)/2}(2\pi)^{-(\adimn)/2}\frac{\rho}{1-\rho^{2}}\int_{\Omega\cap B(0,r)} \langle v,\,y-\rho x\rangle e^{-\frac{\vnorm{y-\rho x}^{2}}{2(1-\rho^{2})}}\,\d y\\
&=-(1-\rho^{2})^{-(\adimn)/2}(2\pi)^{-(\adimn)/2}\rho\int_{\Omega\cap B(0,r)} \mathrm{div}_{y}\Big(ve^{-\frac{\vnorm{y-\rho x}^{2}}{2(1-\rho^{2})}}\Big)\,\d y\\
&=-(1-\rho^{2})^{-(\adimn)/2}(2\pi)^{-(\adimn)/2}\rho\int_{(\Sigma\cap B(0,r))\cup(\Omega\cap\partial B(0,r))}\langle v,N(y)\rangle e^{-\frac{\vnorm{y-\rho x}^{2}}{2(1-\rho^{2})}}\,\d y.
\end{aligned}
\end{equation}
Fix $r'>0$.  Fix $x\in\R^{\adimn}$ with $\vnorm{x}<r'$.  The last integral in \eqref{grep} over $\Omega\cap \partial B(0,r)$ goes to zero as $r\to\infty$ uniformly over all such $\vnorm{x}<r'$.  Also
$\overline{\nabla} T_{\rho}1_{\Omega}(x)$
exists a priori for all $x\in\R^{\adimn}$, while
\begin{flalign*}
&\vnorm{\overline{\nabla} T_{\rho}1_{\Omega}(x)-\overline{\nabla} T_{\rho}1_{\Omega\cap B(0,r)}(x)}
\stackrel{\eqref{oudef}}{=}\frac{\rho}{\sqrt{1-\rho^{2}}}\vnorm{\int_{\R^{\adimn}} y 1_{\Omega\cap B(0,r)^{c}}(x\rho+y\sqrt{1-\rho^{2}})\gamma_{\adimn}(y)\,\d y}\\
&\qquad\qquad\qquad
\leq\frac{\rho}{\sqrt{1-\rho^{2}}}\sup_{w\in\R^{\adimn}\colon\vnorm{w}=1}\int_{\R^{\adimn}} \abs{\langle w,y\rangle} 1_{B(0,r)^{c}}(x\rho+y\sqrt{1-\rho^{2}})\gamma_{\adimn}(y)\,\d y.
\end{flalign*}
And the last integral goes to zero as $r\to\infty$, uniformly over all $\vnorm{x}<r'$.
\end{remark}



\begin{lemma}[\embolden{Second Variation of Translations}]\label{svartran}
Let $v\in\R^{\adimn}$.  Let $\Omega$ maximize Problem \ref{prob2z}.  Let $\{\Omega^{(s)}\}_{s\in(-1,1)}$ be the variation of $\Omega$ corresponding to the constant vector field $X\colonequals v$.  Assume that
$$\int_{\Sigma}\langle v,N(x)\rangle \gamma_{\adimn}(x)\,\d x=0.$$
Let $\theta\colonequals\epsilon\rho\frac{10e^{\alpha^{2}\cdot\max\left(0,\frac{\beta}{\rho-\beta}-1\right)}}{(1-\rho)\vnorm{z}^{2}}$.  Then
\begin{flalign*}
&\frac{1}{2}\frac{\d^{2}}{\d s^{2}}\Big|_{s=0}\Big[\int_{\Omega^{(s)}}\int_{\Omega^{(s)}} G(x,y) \,\d x\d y-\epsilon\pens\Big]\\
&\qquad\geq \Big(\frac{1-\theta}{\rho}-1\Big)\Big\|\overline{\nabla} T_{\rho}1_{\Omega}(x)-\epsilon a_{0}\Big(\Phizderiv+\zeta\Big)\Big\| \abs{\langle v,N(x)\rangle}^{2}\gamma_{\adimn}(x)\, \d x.
\end{flalign*}
\end{lemma}
\begin{proof}

Let $f(x)\colonequals\langle v,N(x)\rangle$ for all $x\in\Sigma$.  By Lemma \ref{techlem},
\begin{flalign*}
&\frac{1}{2}\frac{\d^{2}}{\d s^{2}}\Big|_{s=0}\Big[\int_{\R^{\adimn}}1_{\Omega^{(s)}}(x)T_{\rho}1_{\Omega^{(s)}}(x)\gamma_{\adimn}(x)\,\d x
-\epsilon\pens\Big]\\
&\qquad\geq (1-\theta)\int_{\Sigma}\int_{\Sigma}f(x)G(x,y)f(y)\,\d x\d y\\
&\qquad-\int_{\Sigma}\Big\|\overline{\nabla} T_{\rho}1_{\Omega}(x)-\epsilon a_{0}\Big(\Phizderiv+\zeta\Big)\Big\| \abs{f(x)}^{2}\gamma_{\adimn}(x)\, \d x.
\end{flalign*}

Applying Lemma \ref{treig} (with \eqref{sdef}, \eqref{gdef}),
\begin{equation}\label{tobound}
\begin{aligned}
&\frac{1}{2}\frac{\d^{2}}{\d s^{2}}\Big|_{s=0}\Big[\int_{\Omega^{(s)}}\int_{\Omega^{(s)}} G(x,y) \,\d x\d y-\epsilon\pens\Big]\\
&\geq-\int_{\redA}(1-\theta)\frac{\epsilon a_{0}}{\rho}\Big\langle v, \Phizderiv+\zeta\Big\rangle\\
&\qquad\qquad+\Big(\frac{1-\theta}{\rho}-1\Big)\int_{\Sigma}\Big\|\overline{\nabla} T_{\rho}1_{\Omega}(x)-\epsilon a_{0}\Big(\Phizderiv+\zeta\Big)\Big\| \abs{f(x)}^{2}\gamma_{\adimn}(x)\, \d x.
\end{aligned}
\end{equation}
We now show that the first term in \eqref{tobound} is zero.  By assumption, $\int_{\Sigma}\langle v,N(x)\rangle \gamma_{\adimn}(x)\,\d x=0.$  By the Divergence Theorem and the definition of $z$, we have
$$0=\int_{\Sigma}\langle v,N(x)\rangle \gamma_{\adimn}(x)\,\d x
=\int_{\Omega}\mathrm{div}(v\gamma_{\adimn}(x))\,\d x
=\int_{\Omega}\langle -v,x\rangle\gamma_{\adimn}(x)\,\d x
=-\langle v,z\rangle.
$$
So, $\langle v,\nu(z)\rangle=\langle v,z\rangle/\vnorm{z}=0$. Similarly, $\int_{\Sigma}\langle v,\zeta\rangle f(x)\gamma_{\adimn}(x)\,\d x=\langle v,\zeta\rangle(-1)\langle v,z\rangle=0$.  So, the first term $\int_{\Sigma}\big\langle v, \Phizderiv+\zeta\big\rangle f(x)\gamma_{\adimn}(x)\,\d x$ in \eqref{tobound} is zero.  The proof is concluded.  Note also that $\int_{\Sigma}\vnormf{\overline{\nabla}T_{\rho}1_{\Omega}(x)-\epsilon z}\langle v,N(x)\rangle^{2}\gamma_{\adimn}(x)\,\d x$ is finite priori by the divergence theorem and \eqref{nabeq2}:
\begin{flalign*}
\infty&>\abs{\int_{\Omega}\Big\langle v,-x+\nabla\langle v,\overline{\nabla}T_{\rho}1_{\Omega}(x)-\epsilon z\rangle\Big\rangle\gamma_{\adimn}(x)\,\d x}\\
&=\abs{\int_{\Omega}\mathrm{div}\Big(v\langle v,\overline{\nabla}T_{\rho}1_{\Omega}(x)-\epsilon z\rangle\gamma_{\adimn}(x)\Big)\,\d x}\\
&=\abs{\int_{\Sigma}\langle v,N(x)\rangle\langle v,\overline{\nabla}T_{\rho}1_{\Omega}(x)-\epsilon z\rangle\gamma_{\adimn}(x)\,\d x}\\
&\stackrel{\eqref{nabeq2}}{=}\abs{\int_{\Sigma}\vnormf{\overline{\nabla}T_{\rho}1_{\Omega}(x)-\epsilon z}\langle v,N(x)\rangle^{2}\gamma_{\adimn}(x)\,\d x}.
\end{flalign*}

\end{proof}

\begin{remark}\label{drkv2}
To justify that $\int_{\Sigma} \langle v,N(x)\rangle \gamma_{\adimn}(x)\,\d x$ is finite a priori in Lemma \ref{svartran}, let $r>0$ and use the divergence theorem to obtain
\begin{flalign*}
\int_{\Omega\cap B(0,r)} \mathrm{div}(v\gamma_{\adimn}(x))\,\d x
&=\int_{\Sigma\cap B(0,r)} (-1)\langle v,N(x)\rangle\gamma_{\adimn}(x)\,\d x\\
&\qquad+\int_{\Omega\cap \partial B(0,r)} (-1)\langle v,N(x)\rangle\gamma_{\adimn}(x)\,\d x.
\end{flalign*}
And note that
$$\abs{\int_{\Omega\cap \partial B(0,r)} (-1)\langle v,N(x)\rangle\gamma_{\adimn}(x)\,\d x}
\leq \vnorm{v}\int_{\partial B(0,r)}\gamma_{\adimn}(x)\,\d x.$$
The last quantity goes to zero as $r\to\infty$, and $\int_{\Omega\cap B(0,r)} \mathrm{div}(v\gamma_{\adimn}(x))\,\d x$ is uniformly bounded for all $r>0$, so $\int_{\Sigma} \langle v,N(x)\rangle \gamma_{\adimn}(x)\,\d x$ is finite a priori.
\end{remark}

\section{Proof of Dimension Reduction Theorem}\label{secdimred}

\begin{proof}[Proof of Theorem \ref{mainthm2}]
Let $0<\rho<1$ and let
$$\epsilon<\frac{(1-\rho)^{2}\vnorm{z}^{2}}{\rho10e^{\alpha^{2}\cdot\max\left(0,\frac{\beta}{\rho-\beta}-1\right)}}.$$
Fix $0<a<1$.  Let $\Omega\subset\R^{\adimn}$ be a measurable set that maximizes Problem \ref{prob2z}.  The set $\Omega$ exists by Lemma \ref{existlem} and from Lemma \ref{reglem} the boundary of $\Omega$ is a locally finite union of $C^{\infty}$ $\sdimn$-dimensional manifolds.  Define $\Sigma\colonequals\redb\Omega$.  Define $z\in\R^{\adimn}$ by \eqref{zdef}.  Assume that $\vnorm{z}\geq z_{0}>0$.

By Lemma \ref{firstvarmaxns}, there exists $c\in\R$ such that
$$T_{\rho}(1_{\Omega})(x)-\epsilon\Phiz=c,\qquad\forall\,x\in\Sigma.$$
By this condition, the regularity Lemma \ref{reglem}, and the last part of Lemma \ref{lemma7p},  there exists a set $\sigma\subset\Sigma$ of Hausdorff dimension at most $\sdimn-1$ such that
\begin{flalign*}
&\overline{\nabla}T_{\rho}(1_{\Omega})(x)-\epsilon\Phizderiv\\
&\qquad\qquad= -N(x)\Big\|\overline{\nabla}T_{\rho}(1_{\Omega})(x)-\epsilon\Phizderiv\Big\|,\,\,\forall\,x\in\Sigma\setminus\sigma.
\end{flalign*}
Moreover, by the last part of Lemma \ref{lemma7p}, we have
\begin{equation}\label{nine1}
\Big\|\overline{\nabla}T_{\rho}(1_{\Omega})(x)-\epsilon\Phizderiv\Big\|>0,\qquad\forall\,x\in\Sigma\setminus\sigma.
\end{equation}

Fix $v\in\R^{\adimn}$, and consider the variation of $\Omega$ induced by the constant vector field $X\colonequals v$.  Define $S$ as in \eqref{sdef}.  Define
$$
V\colonequals\Big\{v\in\R^{\adimn}\colon \int_{\Sigma}\langle v,N(x)\rangle \gamma_{\adimn}(x)\,\d x=0\Big\}.
$$
From Lemma \ref{svartran}, $\frac{\d}{\d s}\big|_{s=0}\gamma_{\adimn}(\Omega^{(s)})=0$, and

\begin{flalign*}
&v\in V\,\Longrightarrow\\
&\frac{1}{2}\frac{\d^{2}}{\d s^{2}}\Big|_{s=0}\Big[\int_{\R^{\adimn}}1_{\Omega^{(s)}}(x)T_{\rho}1_{\Omega^{(s)}}(x)\gamma_{\adimn}(x)\,\d x
-\epsilon\pens\Big]\\
&\qquad\quad\geq \Big(\frac{1-\theta}{\rho}-1\Big)\Big\|\overline{\nabla} T_{\rho}1_{\Omega}(x)-\epsilon a_{0}\Big(\Phizderiv+\zeta\Big)\Big\| \abs{\langle v,N(x)\rangle}^{2}\gamma_{\adimn}(x)\, \d x,
\end{flalign*}
where $\theta\colonequals\epsilon\rho\frac{10e^{\alpha^{2}\cdot\max\left(0,\frac{\beta}{\rho-\beta}-1\right)}}{(1-\rho)\vnorm{z}^{2}}$.  Since $$\epsilon<\frac{(1-\rho)^{2}\vnorm{z}^{2}}{\rho10e^{\alpha^{2}\cdot\max\left(0,\frac{\beta}{\rho-\beta}-1\right)}},$$
we have $(1-\theta)/\rho -1  >0$, so \eqref{nine1} implies
%
\begin{equation}\label{nine2}
v\in V\,\Longrightarrow\,\langle v,N(x)\rangle=0,\qquad\forall\,x\in\Sigma.
\end{equation}
The set $V$ has dimension at least $\sdimn$, since it is a set of vectors in $\R^{\adimn}$ that is perpendicular to another fixed vector.  So, by \eqref{nine2}, after rotating $\Omega$, we conclude that there exist measurable $\Omega'\subset\R$ such that
$$\Omega=\Omega'\times\R^{\sdimn}.$$
\end{proof}

\section{One-Dimensional Case}\label{secpro}

By Theorem \ref{mainthm2}, it remains to solve the one-dimensional case of Problem \ref{prob2}.  That is, it suffices to assume that $\Omega\subset\R$.

\begin{proof}[Proof of Theorem \ref{mainthm1}]%
We will show that the maximizer of Problem \ref{prob2z} is a half space $H$ with Gaussian measure $a$.  Consequently, any $\Omega'\subset\R^{\adimn}$ with Gaussian measure $a$ satisfies
\begin{equation}\label{weq}
\begin{aligned}
&\int_{\Omega'}T_{\rho}1_{\Omega'}(x)\gamma_{\adimn}(x)\,\d x
-\epsilon\Big[\int_{\R^{\adimn}}\Phiz1_{\Omega'}(x)\gamma_{\adimn}(x)\,\d x\Big]^{2}\\
&\qquad\qquad
\leq \int_{H}T_{\rho}1_{H}(x)\gamma_{\adimn}(x)\,\d x
-\epsilon\Big[\int_{\R^{\adimn}}\Phiz1_{H}(x)\gamma_{\adimn}(x)\,\d x\Big]^{2},
\end{aligned}
\end{equation}
where $H\subset\R^{\adimn}$ is a half space such that $\int_{\Omega}x\gamma_{\adimn}(x)\,\d x\in\R^{\adimn}$ is a positive multiple of $\int_{H}x\gamma_{\adimn}(x)\,\d x\in\R^{\adimn}$. Rearranging \eqref{weq} gives the first inequality (i.e. the only part we need to prove) of Theorem \ref{mainthm1}:
\begin{flalign*}
&\epsilon\Big( \int_{\R^{\adimn}}\Phiz(1_{H}(x)-1_{\Omega'}(x))\gamma_{\adimn}(x)\d x\Big)\\
&\qquad\cdot \Big( \int_{\R^{\adimn}}\Phiz(1_{H}(x)+1_{\Omega'}(x))\gamma_{\adimn}(x)\d x\Big)\\
&\qquad\qquad\qquad\qquad\leq\int_{\R^{\adimn}}1_{H}(x)T_{\rho}1_{H}(x)\gamma_{\adimn}(x)\,\d x-\int_{\R^{\adimn}}1_{\Omega'}(x)T_{\rho}1_{\Omega'}(x)\gamma_{\adimn}(x)\,\d x,
\end{flalign*}
using also the inequality: 
\begin{flalign*}
&\int_{\R^{\adimn}}\Phiz 1_{H}(x)\gamma_{\adimn}(x)\d x
=\int_{\alpha}^{\infty}\Phi\Big(\frac{\beta x-\alpha}{\sqrt{1-\beta^{2}}}\Big)\gamma_{1}(x)\,\d x\\
&\qquad\geq \int_{\alpha}^{\infty}\Phi\Big(\frac{-\alpha(1-\beta)}{\sqrt{1-\beta^{2}}}\Big)\gamma_{1}(x)\,\d x
=\Phi\Big(-\alpha\sqrt{\frac{1-\beta}{1+\beta}}\Big)\int_{\alpha}^{\infty}\gamma_{1}(x)\,\d x\\
&\qquad\geq\min\Big(\frac{1}{2},\Phi(-\alpha)\Big)\cdot a
=\min\Big(\frac{1}{2},a\Big)\cdot a
\geq a^{2}/2.
\end{flalign*}

Let us therefore find the maximizer of Problem \ref{prob2z}.  By Theorem \ref{mainthm2}, we may assume that $\Omega\subset\R$.  Denote $\Sigma\colonequals\partial\Omega$, $\widetilde{\Omega}\colonequals\Omega\times\R$, and $\widetilde{\Sigma}\colonequals\partial\widetilde{\Omega}$.  Let $f\colon\widetilde{\Sigma}\to\R$ be defined so that
$$f(x_{1},x_{2})\colonequals x_{2},\qquad\forall\, (x_{1},x_{2})\in\widetilde{\Sigma}.$$
Then the corresponding variation of $\widetilde{\Omega}$ satisfies

\begin{flalign*}
&\frac{1}{2}\frac{\d^{2}}{\d s^{2}}\Big|_{s=0}\Big[\int_{\widetilde{\Omega}^{(s)}}\int_{\widetilde{\Omega}^{(s)}} G(x,y) \,\d x\d y-\epsilon\penstilde\Big]\\
&=\int_{\redA\times\R}\int_{\Sigma\times\R} G(x,y)f(x)f(y)\,\d x\d y\\
&\qquad -\int_{\redA\times\R}\Big\|\overline{\nabla} T_{\rho}1_{\Omega\times\R}(x)-\epsilon  a_{0}\frac{\beta e^{-\frac{ [\beta x_{1}-\alpha]^{2}}{2(1-\beta^{2})}}}{\sqrt{1-\beta^{2}}}(1,0)\Big\| \abs{f(x)}^{2}\gamma_{2}(x)\,\d x\\
&\qquad-\epsilon a_{0}\int_{\redA\times\R}\Big(\Phizdd\langle \nu'(\widetilde{z})x,\widetilde{z}'\rangle+\langle x,\zeta'\rangle\Big)f(x)\gamma_{2}(x)\,\d x.
\end{flalign*}

We begin by simplifying the last $\nu'$ term.  We get

 \begin{flalign*}
& \Big\langle\int_{\Sigma\times\R}\Phizdd\nu'(\widetilde{z})x f(x)\gamma_{2}(x)\,\d x,\,\, \widetilde{z}'\Big\rangle\\
&\stackrel{\eqref{zpri}}{=}\Big\langle\int_{\Sigma\times\R}\Phizdd\nu'(\widetilde{z})x f(x)\gamma_{2}(x)\,\d x,\,\, \int_{\Sigma\times\R} y f(y)\gamma_{2}(y)\,\d y\Big\rangle\\
&= \frac{\beta}{\vnorm{z}\sqrt{1-\beta^{2}}}\int_{\Sigma\times\R}e^{-\frac{[\beta x_{1}-\alpha]^{2}}{2(1-\beta^{2})}}x_{2}^{2}\gamma_{2}(x)\,\d  x\cdot \int_{\Sigma\times\R}y_{2}^{2}\gamma_{2}(y)\,\d y\\
&= \frac{\beta}{\vnorm{z}\sqrt{1-\beta^{2}}}\int_{\Sigma}e^{-\frac{[\beta x_{1}-\alpha]^{2}}{2(1-\beta^{2})}}\gamma_{1}(x_{1})\,\d  x_{1}\cdot
\int_{\Sigma}\gamma_{1}(y_{1})\,\d y_{1}.
\end{flalign*}

We examine each of the three terms from \eqref{lastbd}.  The first term was just dealt with.  For the second term, we have

 \begin{flalign*}
& \Big\langle\int_{\Sigma\times\R}x f(x)\gamma_{2}(x)\,\d x,\\
&\qquad\qquad\qquad\int_{\Omega\times\R}\frac{\beta^{2}[\alpha-\beta\langle y,\nu(z)\rangle] e^{-\frac{[\beta\langle y,\nu(z)\rangle-\alpha]^{2}}{2(1-\beta^{2})}}}{(1-\beta^{2})^{3/2}}\vnorm{z}^{-2}\langle\mathrm{Proj}_{z^{\perp}}(y),z'\rangle\mathrm{Proj}_{z^{\perp}}(y)\gamma_{2}(y)\,\d y\Big\rangle\\
&=\int_{\Sigma\times\R}x_{2} f(x)\gamma_{2}(x)\,\d x\cdot \int_{\Omega\times\R}\frac{\beta^{2}[\alpha-\beta y_{1}] e^{-\frac{[\beta y_{1}-\alpha]^{2}}{2(1-\beta^{2})}}}{(1-\beta^{2})^{3/2}}\vnorm{z}^{-2}y_{2}\langle e_{2},z'\rangle y_{2}\gamma_{2}(y)\,\d y\\
&=\frac{1}{\vnorm{z}^{2}}\int_{\Sigma}\gamma_{1}(x)\,\d x\cdot \int_{\Omega}\frac{\beta^{2}[\alpha-\beta y_{1}] e^{-\frac{[\beta y_{1}-\alpha]^{2}}{2(1-\beta^{2})}}}{(1-\beta^{2})^{3/2}}\langle e_{2},z'\rangle\gamma_{1}(y)\,\d y\\
&\stackrel{\eqref{zpri}}{=}\frac{1}{\vnorm{z}^{2}}\Big(\int_{\Sigma}\gamma_{1}(x)\,\d x\Big)^{2}\cdot \int_{\Omega}\frac{\beta^{2}[\alpha-\beta y_{1}] e^{-\frac{[\beta y_{1}-\alpha]^{2}}{2(1-\beta^{2})}}}{(1-\beta^{2})^{3/2}}\gamma_{1}(y)\,\d y.
\end{flalign*}

Finally, from \eqref{bigbd}, (noting that the first, third fourth terms there are zero),
\begin{flalign*}
&\Big\langle\int_{\Sigma\times\R}x f(x)\gamma_{2}(x)\,\d x,\,\, \int_{\Omega\times\R}\Phizddy\nu''(z) y\gamma_{2}(y)\,\d y\Big\rangle\\
&\qquad\qquad=-\vnorm{z}^{-3}\Big\|\int_{\Sigma\times\R}x \gamma_{2}(x)\,\d x\Big\|^{2} \int_{\Omega\times\R}\Phizddy\langle y, z\rangle \gamma_{2}(y)\,\d y\\
&\qquad\qquad=-\vnorm{z}^{-2}\Big(\int_{\Sigma}\gamma_{1}(x)\,\d x\Big)^{2} \int_{\Omega}\Phizddy y_{1} \gamma_{1}(y)\,\d y.
&\end{flalign*}
Therefore,
\begin{equation}\label{tobd6}
\begin{aligned}
&\frac{1}{2}\frac{\d^{2}}{\d s^{2}}\Big|_{s=0}\Big[\int_{\R^{2}}\int_{\R^{2}} 1_{\Omega^{(s)}}(y)G(x,y) 1_{\Omega^{(s)}}(x)\,\d x\d y-\epsilon\pens\Big]\\
&=\int_{\redA\times\R}\int_{\Sigma\times\R} G(x,y)\,\d x\d y
-\int_{\redA\times\R}\Big\langle N(x),\, \overline{\nabla} T_{\rho}1_{\Omega\times\R}(x)-\epsilon  a_{0}\frac{\beta e^{-\frac{ [\beta x_{1}-\alpha]^{2}}{2(1-\beta^{2})}}}{\sqrt{1-\beta^{2}}}(1,0)\Big\rangle \gamma_{2}(x)\,\d x\\
&\qquad-2\epsilon a_{0}\frac{\beta}{\vnorm{z}\sqrt{1-\beta^{2}}}\int_{\Sigma}e^{-\frac{[\beta x_{1}-\alpha]^{2}}{2(1-\beta^{2})}}\gamma_{1}(x_{1})\,\d  x_{1}\cdot
\gamma_{1}(\Sigma)\\
&\qquad+\epsilon a_{0}\frac{1}{\vnorm{z}^{2}}\gamma_{1}(\Sigma)^{2}\cdot \Big[-\int_{\Omega}\frac{\beta^{2}[\alpha-\beta y_{1}] e^{-\frac{[\beta y_{1}-\alpha]^{2}}{2(1-\beta^{2})}}}{(1-\beta^{2})^{3/2}}\gamma_{1}(y)\,\d y
+\int_{\Omega}\frac{\beta e^{-\frac{[\beta y_{1}-\alpha]^{2}}{2(1-\beta^{2})}}}{\sqrt{1-\beta^{2}}}y_{1}\gamma_{1}(y)\,\d y\Big]\\
&=\int_{\redA\times\R}\int_{\Sigma\times\R} G(x,y)\,\d x\d y
-\int_{\redA\times\R}\langle N(x),\, \overline{\nabla} T_{\rho}1_{\Omega\times\R}(x)\rangle \gamma_{2}(x)\,\d x\\
&\qquad-\epsilon a_{0}\frac{\beta}{\sqrt{1-\beta^{2}}}\int_{\Sigma}e^{-\frac{[\beta x_{1}-\alpha]^{2}}{2(1-\beta^{2})}}\gamma_{1}(x_{1})\,\d  x_{1}\cdot
(-1+2\gamma_{1}(\Sigma)/\vnorm{z})\\
&\qquad+\epsilon a_{0}\frac{1}{\vnorm{z}^{2}}\gamma_{1}(\Sigma)^{2}\frac{\beta}{(1-\beta^{2})^{3/2}}\int_{\Omega}(y_{1}-\alpha\beta)e^{-\frac{[\beta y_{1}-\alpha]^{2}}{2(1-\beta^{2})}}\gamma_{1}(y)\,\d y.
\end{aligned}
\end{equation}

We will now split into two cases, thereby giving two separate bounds for \eqref{tobd6}

\textbf{Case 1.}  Assume in this case that $\frac{1}{10}\int_{H}\gamma_{1}(y)\,\d y>24\pi(1+\alpha^{2})(\gamma_{1}(\Sigma)-\gamma_{1}(\partial H))$.  Lemma \ref{finallem1} with $k=1$ lower bounds the first term.  Lemma \ref{finallem7} controls the middle term.  And Lemma \ref{finallem9} controls the last term.  Combining these estimates,

\begin{flalign*}
&\frac{1}{2}\frac{\d^{2}}{\d s^{2}}\Big|_{s=0}\Big[\int_{\R^{2}}\int_{\R^{2}} 1_{\Omega^{(s)}}(y)G(x,y) 1_{\Omega^{(s)}}(x)\,\d x\d y-\epsilon\pens\Big]\\
&\qquad\geq\Big(\int_{\partial\Omega}\gamma_{1}(x)\,\d x
-\int_{\partial H}\gamma_{1}(x)\,\d x\Big)\\
&\qquad\quad
\cdot\Big[\rho(1-\rho)\frac{\min(a,1-a)}{80}
-\epsilon a_{0}e^{\frac{\alpha^{2}}{1+\beta}}8\frac{(6+\abs{\alpha})^{2}}{\beta(1-\beta^{2})}\Big((-1+2\gamma_{1}(\Sigma)/\vnorm{z})+\frac{1}{\vnorm{z}^{2}}\gamma_{1}(\Sigma)^{2}\Big)\Big]\\
&\qquad-\epsilon a_{0}(-1+2\gamma_{1}(\Sigma)/\vnorm{z})\frac{\beta}{\sqrt{1-\beta^{2}}}\int_{H}e^{-\frac{[\beta x-\alpha]^{2}}{2(1-\beta^{2})}}\gamma_{1}(x)\,\d x\\
&\qquad+\epsilon a_{0}\frac{1}{\vnorm{z}^{2}}\gamma_{1}(\Sigma)^{2}\frac{\beta}{(1-\beta^{2})^{3/2}}\int_{H}(x-\alpha\beta)e^{-\frac{[\beta x-\alpha]^{2}}{2(1-\beta^{2})}}\gamma_{1}(x)\,\d x\\
&\quad=\Big(\int_{\partial\Omega}\gamma_{1}(x)\,\d x
-\int_{\partial H}\gamma_{1}(x)\,\d x\Big)\\
&\qquad\quad\cdot\Big[\rho(1-\rho)\frac{\min(a,1-a)}{80}
+\epsilon a_{0}8e^{\frac{\alpha^{2}}{1+\beta}}\frac{(6+\abs{\alpha})^{2}}{\beta(1-\beta^{2})}\Big((1-\gamma_{1}(\Sigma)/\vnorm{z})^{2}-2\frac{1}{\vnorm{z}^{2}}\gamma_{1}(\Sigma)^{2}\Big)\Big]\\
&\qquad+\epsilon a_{0}\frac{\beta}{\sqrt{1-\beta^{2}}}\int_{H}e^{-\frac{[\beta x-\alpha]^{2}}{2(1-\beta^{2})}}\gamma_{1}(x)\,\d x(1-\gamma_{1}(\Sigma)/\vnorm{z})^{2}.
\end{flalign*}
Ignoring some nonnegative terms,

\begin{flalign*}
&\frac{1}{2}\frac{\d^{2}}{\d s^{2}}\Big|_{s=0}\Big[\int_{\R^{2}}\int_{\R^{2}} 1_{\Omega^{(s)}}(y)G(x,y) 1_{\Omega^{(s)}}(x)\,\d x\d y-\epsilon\pens\Big]\\
&\geq\Big(\int_{\partial\Omega}\gamma_{1}(x)\,\d x
-\int_{\partial H}\gamma_{1}(x)\,\d x\Big)\\
&\qquad\qquad\qquad\qquad\qquad\qquad\cdot\Big[\rho(1-\rho)\frac{\min(a,1-a)}{80}
-\epsilon a_{0}16e^{\frac{\alpha^{2}}{1+\beta}}\frac{(6+\abs{\alpha})^{2}}{\beta(1-\beta^{2})}\frac{1}{\vnorm{z}^{2}}\gamma_{1}(\Sigma)^{2}\Big].
\end{flalign*}%

From the main result of \cite{barchiesi16},
$$
\abs{\int_{\Omega}x\gamma_{1}(x)\,\d x
-\int_{H}x\gamma_{1}(x)\,\d x}
\leq24\pi(1+\alpha^{2})\Big(\int_{\partial\Omega}\gamma_{1}(x)\,\d x
-\int_{\partial H}\gamma_{1}(x)\,\d x\Big).
$$
So, using this inequality
\begin{flalign*}
&\frac{1}{2}\frac{\d^{2}}{\d s^{2}}\Big|_{s=0}\Big[\int_{\R^{2}}\int_{\R^{2}} 1_{\Omega^{(s)}}(y)G(x,y) 1_{\Omega^{(s)}}(x)\,\d x\d y-\epsilon\pens\Big]\\
&\qquad\geq\Big(\int_{\partial\Omega}\gamma_{1}(x)\,\d x
-\int_{\partial H}\gamma_{1}(x)\,\d x\Big)\cdot\Big[\rho(1-\rho)\frac{\min(a,1-a)}{80}\\
&\qquad\qquad-\epsilon \frac{a_{0}16e^{\frac{\alpha^{2}}{1+\beta}}(6+\abs{\alpha})^{2}}{\beta(1-\beta^{2})}\Big(\frac{\gamma_{1}(\Sigma)}{\int_{H}y\gamma_{1}(y)\,\d y-24\pi(1+\alpha^{2})(\gamma_{1}(\Sigma)-\gamma_{1}(\partial H))}\Big)^{2}\Big].
\end{flalign*}
Then, using the assumption of Case 1,
\begin{flalign*}
&\frac{1}{2}\frac{\d^{2}}{\d s^{2}}\Big|_{s=0}\Big[\int_{\R^{2}}\int_{\R^{2}} 1_{\Omega^{(s)}}(y)G(x,y) 1_{\Omega^{(s)}}(x)\,\d x\d y-\epsilon\pens\Big]\\
&\qquad\geq\Big(\int_{\partial\Omega}\gamma_{1}(x)\,\d x
-\int_{\partial H}\gamma_{1}(x)\,\d x\Big)\cdot\Big[\rho(1-\rho)\frac{\min(a,1-a)}{80}\\
&\qquad\qquad-\epsilon \frac{a_{0}16e^{\frac{\alpha^{2}}{1+\beta}}(6+\abs{\alpha})^{2}}{\beta(1-\beta^{2})}\Big(\frac{\gamma_{1}(\partial H)+\frac{1}{240\pi(1+\alpha^{2})}\int_{H}y\gamma_{1}(y)\,\d y}{\frac{9}{10}\int_{H}y\gamma_{1}(y)\,\d y}\Big)^{2}\Big]\\
&\geq\Big(\int_{\partial\Omega}\gamma_{1}(x)\,\d x
-\int_{\partial H}\gamma_{1}(x)\,\d x\Big)\cdot\Big[\rho(1-\rho)\frac{\min(a,1-a)}{80}\\
&\qquad\qquad\qquad\qquad\qquad\qquad\qquad\qquad\qquad
-\epsilon \frac{a_{0}16e^{\frac{\alpha^{2}}{1+\beta}}(6+\abs{\alpha})^{2}\gamma_{1}(\partial H)^{2}}{\beta(1-\beta^{2})}\Big(\frac{1+\frac{1}{240\pi(1+\alpha^{2})}}{\frac{9}{10}}\Big)^{2}\Big]\\
&\geq\Big(\int_{\partial\Omega}\gamma_{1}(x)\,\d x
-\int_{\partial H}\gamma_{1}(x)\,\d x\Big)\cdot\Big[\rho(1-\rho)\frac{\min(a,1-a)}{80}
-\epsilon \frac{a_{0}16e^{\frac{\alpha^{2}}{1+\beta}}(6+\abs{\alpha})^{2}}{\beta(1-\beta^{2})}\Big].
\end{flalign*}%

\textbf{Case 2.}  Assume in this case that $\frac{1}{10}\int_{H}\gamma_{1}(y)\,\d y\leq 24\pi(1+\alpha^{2})(\gamma_{1}(\Sigma)-\gamma_{1}(\partial H))$.  We now give our alternate bound of \eqref{tobd6}.  Lemma \ref{finallem1} lower bounds the first term, though we use the $k=2$ case instead of the $k=1$ case.  Lemma \ref{finallem7} controls the middle term again.  For the final term, we just lower bound it by zero.  Combining these estimates,

\begin{flalign*}
&\frac{1}{2}\frac{\d^{2}}{\d s^{2}}\Big|_{s=0}\Big[\int_{\R^{2}}\int_{\R^{2}} 1_{\Omega^{(s)}}(y)G(x,y) 1_{\Omega^{(s)}}(x)\,\d x\d y-\epsilon\pens\Big]\\
&\geq\Big(\int_{\partial\Omega}\gamma_{1}(x)\,\d x
-\int_{\partial H}\gamma_{1}(x)\,\d x\Big)\cdot\Big[\rho(1-\rho)\frac{\min(a,1-a)}{80}\Big(\int_{\partial\Omega}\gamma_{1}(x)\,\d x
-\int_{\partial H}\gamma_{1}(x)\,\d x\Big)\\
&\qquad\qquad\qquad\qquad\qquad\qquad
-\epsilon a_{0}8e^{\frac{\alpha^{2}}{1+\beta}}\frac{(6+\abs{\alpha})^{2}}{\beta(1-\beta^{2})}(-1+2\gamma_{1}(\Sigma)/\vnorm{z})\Big]\\
&\qquad-\epsilon a_{0}(-1+2\gamma_{1}(\Sigma)/\vnorm{z})\frac{\beta}{\sqrt{1-\beta^{2}}}\int_{H}e^{-\frac{[\beta x-\alpha]^{2}}{2(1-\beta^{2})}}\gamma_{1}(x)\,\d x\\
&\geq\Big(\int_{\partial\Omega}\gamma_{1}(x)\,\d x
-\int_{\partial H}\gamma_{1}(x)\,\d x\Big)\cdot\Big[\rho(1-\rho)\frac{\min(a,1-a)}{80}\Big(\int_{\partial\Omega}\gamma_{1}(x)\,\d x
-\int_{\partial H}\gamma_{1}(x)\,\d x\Big)\\
&\qquad\qquad\qquad\qquad\qquad\qquad
-\epsilon a_{0}10e^{\frac{\alpha^{2}}{1+\beta}}\frac{(6+\abs{\alpha})^{2}}{\beta(1-\beta^{2})}(-1+2\gamma_{1}(\Sigma)/\vnorm{z})\Big].
\end{flalign*}
Ignoring some nonnegative terms and rearranging

\begin{flalign*}
&\frac{1}{2}\frac{\d^{2}}{\d s^{2}}\Big|_{s=0}\Big[\int_{\R^{2}}\int_{\R^{2}} 1_{\Omega^{(s)}}(y)G(x,y) 1_{\Omega^{(s)}}(x)\,\d x\d y-\epsilon\pens\Big]\\
&\geq\Big(\int_{\partial\Omega}\gamma_{1}(x)\,\d x
-\int_{\partial H}\gamma_{1}(x)\,\d x\Big)\cdot\Big[\rho(1-\rho)\frac{\min(a,1-a)}{80}\Big(\int_{\partial\Omega}\gamma_{1}(x)\,\d x
-\int_{\partial H}\gamma_{1}(x)\,\d x\Big)\\
&\qquad\qquad\qquad\qquad\qquad\qquad
-\epsilon \frac{1}{\vnorm{z}}a_{0}200e^{\frac{\alpha^{2}}{1+\beta}}\frac{(6+\abs{\alpha})^{2}}{\beta(1-\beta^{2})}[\gamma_{1}(\Sigma)-\gamma_{1}(\partial H)+\gamma_{1}(\partial H)]\Big]\\
&=\Big(\int_{\partial\Omega}\gamma_{1}(x)\,\d x
-\int_{\partial H}\gamma_{1}(x)\,\d x\Big)\cdot\Big[\rho(1-\rho)\frac{\min(a,1-a)}{80}\Big(\int_{\partial\Omega}\gamma_{1}(x)\,\d x
-\int_{\partial H}\gamma_{1}(x)\,\d x\Big)\\
&\qquad\qquad\qquad\qquad\qquad
-\epsilon \frac{1}{\vnorm{z}}a_{0}200e^{\frac{\alpha^{2}}{1+\beta}}\frac{(6+\abs{\alpha})^{2}}{\beta(1-\beta^{2})}\Big]
-\epsilon \frac{1}{\vnorm{z}}a_{0}200e^{\frac{\alpha^{2}}{1+\beta}}\frac{(6+\abs{\alpha})^{2}}{\beta(1-\beta^{2})}\gamma_{1}(\partial H)].
\end{flalign*}

Then, using the assumption of Case 2,

\begin{flalign*}
&\frac{1}{2}\frac{\d^{2}}{\d s^{2}}\Big|_{s=0}\Big[\int_{\R^{2}}\int_{\R^{2}} 1_{\Omega^{(s)}}(y)G(x,y) 1_{\Omega^{(s)}}(x)\,\d x\d y-\epsilon\pens\Big]\\
&\qquad\geq\frac{1}{10}\gamma_{1}(\partial H)\cdot\Big[\rho(1-\rho)\frac{\min(a,1-a)}{80}\frac{1}{10}\gamma_{1}(\partial H)
-\epsilon \frac{1}{\vnorm{z}}a_{0}200e^{\frac{\alpha^{2}}{1+\beta}}\frac{(6+\abs{\alpha})^{2}}{\beta(1-\beta^{2})}\Big]\\
&\qquad\qquad\qquad\qquad\qquad\qquad\qquad\qquad\qquad
-\epsilon \frac{1}{\vnorm{z}}a_{0}200e^{\frac{\alpha^{2}}{1+\beta}}\frac{(6+\abs{\alpha})^{2}}{\beta(1-\beta^{2})}\gamma_{1}(\partial H)]\\
&\qquad=\frac{1}{10}\gamma_{1}(\partial H)\cdot\Big[\rho(1-\rho)\frac{\min(a,1-a)}{80}\frac{1}{10}\gamma_{1}(\partial H)
-\epsilon \frac{1}{\vnorm{z}}a_{0}2200e^{\frac{\alpha^{2}}{1+\beta}}\frac{(6+\abs{\alpha})^{2}}{\beta(1-\beta^{2})}\Big].
\end{flalign*}

\textbf{Completing the Proof by combining Cases 1 and 2}.

Suppose
$$\epsilon<\frac{\rho(1-\rho)\frac{\min(a,1-a)}{80}\frac{1}{10}\gamma_{1}(\partial H)}{\frac{1}{\vnorm{z}}a_{0}2200e^{\frac{\alpha^{2}}{1+\beta}}\frac{(6+\abs{\alpha})^{2}}{\beta(1-\beta^{2})}}.$$
It then follows that
$$\epsilon<\frac{\rho(1-\rho)\frac{\min(a,1-a)}{80}}{\Big[\frac{a_{0}16e^{\frac{\alpha^{2}}{1+\beta}}(6+\abs{\alpha})^{2}}{\beta(1-\beta^{2})}\Big]}.$$
So, in either Case 1 or Case 2, the following quantity is positive:

$$\frac{1}{2}\frac{\d^{2}}{\d s^{2}}\Big|_{s=0}\Big[\int_{\R^{2}}\int_{\R^{2}} 1_{\Omega^{(s)}}(y)G(x,y) 1_{\Omega^{(s)}}(x)\,\d x\d y-\epsilon\pens\Big].
$$
That is, we have contradicted the assumption that $\Omega$ maximizes Problem \ref{prob2z}, unless $\Omega$ is a half space.

Finally, Lemma \ref{finallem2} shows that $\Omega$ has the desired Gaussian measure, i.e. $\gamma_{1}(\Omega)=a$, so that $\Omega$ maximizes Problem \ref{prob2}, as desired.

This inequality contradicts Lemma \ref{lemma7p} for $\Omega\times\R\subset\R^{2}$.  We conclude that $\Omega$ is a half space.  It remains to show that $\gamma_{1}(\Omega)=a$.  This follows from Lemma \ref{finallem2}.
\end{proof}

\section{One-Dimensional Lemmas}\label{lastlems}

The lemmas below are needed in order to lower bound the second variation in Section \ref{secpro}.

\begin{lemma}\label{finallem1}
Let $\Omega\subset\R$, $\Sigma\colonequals\partial\Omega$.  Then
\begin{flalign*}
&\int_{\redA\times\R}\int_{\redA\times\R}G(x,y) \,\d x\d y
-\int_{\redA\times\R}\vnorm{\overline{\nabla} T_{\rho}1_{\Omega\times\R}(x)}\gamma_{2}(x)\,\d x\\
&\qquad\qquad\geq\rho(1-\rho)\frac{\min(a,1-a)}{80}\max_{k\in\{1,2\}}\Big(\int_{\Sigma}\gamma_{1}(x)\,\d x-\int_{H}\gamma_{1}(x)\,\d x\Big)^{k}.
\end{flalign*}
\end{lemma}
\begin{proof}
\begin{flalign*}
&\int_{\redA\times\R}\int_{\redA\times\R}G(x,y) \,\d x\d y
-\int_{\redA\times\R}\vnorm{\overline{\nabla} T_{\rho}1_{\Omega\times\R}(x)}\gamma_{2}(x)\,\d x\\
&\qquad\qquad=\int_{\R^{2}}\rho x_{2}^{2} T_{\rho}1_{\Sigma}(x_{1})\gamma_{2}(x) \,\d x\d y
-\int_{\redA\times\R}\vnorm{\overline{\nabla} T_{\rho}1_{\Omega}(x)}x_{2}^{2} \gamma_{2}(x)\,\d x.\\
&\qquad\qquad=\rho\int_{\Sigma'}T_{\rho}1_{\Sigma}(x_{1})\gamma_{1}(x) \,\d x
-\int_{\redA}\vnorm{\overline{\nabla} T_{\rho}1_{\Omega}(x)}\gamma_{1}(x)\,\d x.\\
\end{flalign*}
Using \eqref{firstve} then the divergence theorem,
\begin{flalign*}
&\int_{\redA\times\R}\int_{\redA\times\R}G(x,y) \,\d x\d y
-\int_{\redA\times\R}\vnorm{\overline{\nabla} T_{\rho}1_{\Omega\times\R}(x)}\gamma_{2}(x)\,\d x\\
&\qquad\qquad=\rho\int_{\R}(T_{\sqrt{\rho}}1_{\Sigma}(x_{1}))^{2}\gamma_{1}(x) \,\d x
+\int_{\redA}\langle  N(x),\overline{\nabla} T_{\rho}1_{\Omega}(x)\rangle\gamma_{1}(x)\,\d x.\\
&\qquad\qquad=\rho\int_{\R}(T_{\sqrt{\rho}}1_{\Sigma}(x))^{2}\gamma_{1}(x) \,\d x
+\int_{\Omega}\mathrm{div}\Big(\overline{\nabla} T_{\rho}1_{\Omega}(x)\gamma_{1}(x)\Big)\,\d x.\\
&\qquad\qquad=\rho\int_{\R}(T_{\sqrt{\rho}}1_{\Sigma}(x))^{2}\gamma_{1}(x) \,\d x
+\int_{\R}1_{\Omega}(x)(\overline{\Delta}-\langle x,\overline{\nabla}\rangle)T_{\rho}1_{\Omega}(x)\gamma_{1}(x)\,\d x.
\end{flalign*}
Integrating by parts,
$$
\int_{\R}1_{\Omega}(x)(\overline{\Delta}-\langle x,\overline{\nabla}\rangle)T_{\rho}1_{\Omega}(x)\gamma_{1}(x)\,\d x.
=-\int_{\R}\vnormf{\overline{\nabla} T_{\sqrt{\rho}}1_{\Omega}(x)}^{2}\gamma_{1}(x)\,\d x.
$$
Since
$$T_{\sqrt{\rho}}1_{\Omega}(x)=\int_{\frac{\Omega-\sqrt{\rho}x}{1-\rho}}\gamma_{1}(y)\,\d y,$$
$$\overline{\nabla} T_{\sqrt{\rho}}1_{\Omega}(x)=\sqrt{\rho}\int_{\frac{\Omega-\sqrt{\rho}x}{1-\rho}}y\gamma_{1}(y)\,\d y,$$
In summary,
\begin{flalign*}
&\int_{\redA\times\R}\int_{\redA\times\R}G(x,y) \,\d x\d y
-\int_{\redA\times\R}\vnorm{\overline{\nabla} T_{\rho}1_{\Omega\times\R}(x)}\gamma_{2}(x)\,\d x\\
&\qquad\qquad\qquad\qquad=\rho\int_{\R}\Big((T_{\sqrt{\rho}}1_{\Sigma}(x))^{2}-\vnormf{\int_{\frac{\Omega-\sqrt{\rho}x}{1-\rho}}y\gamma_{1}(y)\,\d y}^{2}\Big)\gamma_{1}(x)\,\d x\\
&\qquad\qquad\qquad\qquad=\rho\int_{\R}\Big(\int_{\frac{\Sigma-\sqrt{\rho}x}{1-\rho}}\gamma_{1}(y)\,\d y)^{2}-\vnormf{\int_{\frac{\Omega-\sqrt{\rho}x}{1-\rho}}y\gamma_{1}(y)\,\d y}^{2}\Big)\gamma_{1}(x)\,\d x.
\end{flalign*}
The latter quantity is nonnegative by rearrangement, as noted e.g. in Eldan's proof of Borell's inequality \cite{eldan13}.  However, we need a better lower bound.  To this end, write $\Sigma=\cup_{i=1}^{r}\sigma_{i}\subset\R$ with $r$ a positive integer or $\infty$, such that $\sigma_{2i-1}<\sigma_{2i}$ for all $i\geq1$.  In the case $r=2$ we have
\begin{equation}\label{expl1}
\begin{aligned}
&\int_{\R}\Big((\int_{\frac{\Sigma-\sqrt{\rho}x}{1-\rho}}\gamma_{1}(y)\,\d y)^{2}-\vnormf{\int_{\frac{\Omega-\sqrt{\rho}x}{1-\rho}}y\gamma_{1}(y)}^{2}\Big)\gamma_{1}(x)\,\d x\\
&\qquad=\frac{4}{(2\pi)^{3/2}}\int_{\R}\Big[\Big(e^{-\frac{1}{2}(\sigma_{1}-\sqrt{\rho}x/(1-\rho))^{2}}+e^{-\frac{1}{2}(\sigma_{1}-\sqrt{\rho}x/(1-\rho))^{2}}\Big)^{2}\\
&\qquad\qquad\qquad\qquad\qquad\qquad-\Big(e^{-\frac{1}{2}(\sigma_{1}-\sqrt{\rho}x/(1-\rho))^{2}}-e^{-\frac{1}{2}(\sigma_{1}-\sqrt{\rho}x/(1-\rho))^{2}}\Big)^{2}e^{-x^{2}/2}\Big]\,\d x\\
&\qquad=\frac{4}{(2\pi)^{3/2}}\int_{\R}e^{-\frac{1}{2}(\sigma_{1}-\sqrt{\rho}x/(1-\rho))^{2}-\frac{1}{2}(\sigma_{2}-\sqrt{\rho}x/(1-\rho))^{2}-\frac{1}{2}x^{2}}\,\d x\\
&\qquad=\frac{4}{(2\pi)^{3/2}}\int_{\R}e^{-x^{2}\Big(\frac{1}{2}+\frac{\rho}{(1-\rho)^{2}}\Big)+\frac{(\sigma_{1}+\sigma_{2})\sqrt{\rho}x}{(1-\rho)}-\frac{1}{2}(\sigma_{1}^{2}+\sigma_{2}^{2})}\,\d x\\
&\qquad=\frac{4}{(2\pi)^{3/2}}\int_{\R}e^{-\Big(\frac{1}{2}+\frac{\rho}{(1-\rho)^{2}}\Big)\Big(x-\frac{(\sigma_{1}+\sigma_{2})\sqrt{\rho}}{(1-\rho)}\frac{1}{2}\frac{1}{1/2+\rho/(1-\rho)^{2}}\Big)^{2}
+\frac{\Big(\frac{(\sigma_{1}+\sigma_{2})^{2}\rho}{(1-\rho)^{2}}\Big)}{4(1/2+\rho/(1-\rho)^{2})^{2}}-\frac{1}{2}(\sigma_{1}^{2}+\sigma_{2}^{2})}\,\d x\\
&\qquad=\frac{16}{(2\pi)^{3/2}}\frac{\sqrt{\pi}}{\sqrt{\frac{1}{2}+\frac{\rho}{(1-\rho)^{2}}}}e^{\frac{(\sigma_{1}+\sigma_{2})^{2}\rho}{4((1/2)(1-\rho)+\rho/(1-\rho))^{2}}}
e^{-\frac{1}{2}(\sigma_{1}^{2}+\sigma_{2}^{2})}\\
&\qquad=\frac{4\sqrt{2}(1-\rho)}{\pi\sqrt{\frac{1}{2}(1-\rho)^{2}+\rho}}e^{-\frac{1}{2}(\sigma_{1}^{2}+\sigma_{2}^{2})\Big(1-\frac{\rho(1-\rho)^{2}}{2((1/2)(1-\rho)^{2}+\rho)^{2}}\Big)}\\
&\qquad=\frac{4\sqrt{2}(1-\rho)}{\pi\sqrt{\rho^{2}+1}}\Big(e^{-\frac{1}{2}\sigma_{1}^{2}}e^{-\frac{1}{2}\sigma_{2}^{2}}\Big)^{\Big(1-\frac{\rho(1-\rho)^{2}}{2((1/2)(1-\rho)^{2}+\rho)^{2}}\Big)}
=\frac{4\sqrt{2}(1-\rho)}{\pi\sqrt{\rho^{2}+1}}\Big(e^{-\frac{1}{2}\sigma_{1}^{2}}e^{-\frac{1}{2}\sigma_{2}^{2}}\Big)^{\Big(1-\frac{2\rho(1-\rho)^{2}}{\rho^{2}+1}\Big)}.
\end{aligned}
\end{equation}
More generally,
\begin{equation}\label{stackeq}
\begin{aligned}
&\int_{\R}\Big(\Big[\int_{\frac{\Sigma-\sqrt{\rho}x}{1-\rho}}\gamma_{1}(y)\,\d y\Big]^{2}-\vnormf{\int_{\frac{\Omega-\sqrt{\rho}x}{1-\rho}}y\gamma_{1}(y)}^{2}\Big)\gamma_{1}(x)\,\d x\\
&\quad=\frac{4}{(2\pi)^{3/2}}\int_{\R}\Big(\sum_{i=1}^{r}e^{-\frac{1}{2}(\sigma_{i}-\sqrt{\rho}x/(1-\rho))^{2}}\Big)^{2}
-\Big(\sum_{i=1}^{r}(-1)^{i}e^{-\frac{1}{2}(\sigma_{i}-\sqrt{\rho}x/(1-\rho))^{2}}\Big)^{2}e^{-x^{2}/2}\,\d x\\
&\quad=\frac{4}{(2\pi)^{3/2}}\int_{\R}\Big(\sum_{i=2,\ldots,r,\,\,\mathrm{even}}e^{-\frac{1}{2}(\sigma_{i}-\sqrt{\rho}x/(1-\rho))^{2}}\Big)
\cdot\Big(\sum_{i=1,\ldots,r,\,\,\mathrm{odd}}e^{-\frac{1}{2}(\sigma_{i}-\sqrt{\rho}x/(1-\rho))^{2}}\Big)e^{-x^{2}/2}\,\d x\\
&\quad\stackrel{\eqref{expl1}}{=}\frac{4\sqrt{2}(1-\rho)}{\pi\sqrt{\rho^{2}+1}}\Big(\sum_{i=2,\ldots,r,\,\,\mathrm{even}}e^{-\frac{1}{2}\sigma_{i}^{2}\Big(1-\frac{2\rho(1-\rho)^{2}}{\rho^{2}+1}\Big)}\Big)
\Big(\sum_{i=1,\ldots,r,\,\,\mathrm{odd}}e^{-\frac{1}{2}\sigma_{i}^{2}\Big(1-\frac{2\rho(1-\rho)^{2}}{\rho^{2}+1}\Big)}\Big).
\end{aligned}
\end{equation}

We now analyze a variant of the last quantity.  For any $\Omega\subset\R$ denote $\Sigma\colonequals\partial\Omega$, $\Sigma_{L}\colonequals\{x\in\Sigma\colon N(x)=-1\}$, $\Sigma_{R}\colonequals\{x\in\Sigma\colon N(x)=1\}$, so that $\Sigma$ is the disjoint union $\Sigma=\Sigma_{L}\cup\Sigma_{R}$.  Consider now the following minimization problem in $\R$: fix $0<a<1$ and minimize
\begin{equation}\label{minprob}
\int_{\Sigma_{L}}\gamma_{1}(x)\,\d x\cdot \int_{\Sigma_{R}}\gamma_{1}(y)\,\d y-\kappa\int_{\Sigma}\gamma_{1}(x)\,\d x
\end{equation}
over all $\Omega\subset\R$ such that $\gamma_{1}(\Omega)=a$.  The first variation condition implies that the minimum of this problem is attained at an interval of the form $\Omega=[c,d]$ or its complement, where $-\infty\leq c\leq d\leq\infty$.  For any $s\in\R$, let $c(s),d(s)\in\R$ such that $\gamma_{1}([c(s),d(s)])=a$.  (Assume that $c'(s),d'(s)>0$ for all $s\in\R$.)  Then $\Phi(d(s))-\Phi(c(s))=a$, so differentiating gives
$$\gamma_{1}(d(s))(-d'(s))-\gamma_{1}(c(s))(-c'(s))=0.$$
That is,
\begin{equation}\label{gdifeq}
\gamma_{1}(d(s))d'(s)=\gamma_{1}(c(s))c'(s).
\end{equation}
Also,
\begin{flalign*}
&\frac{\d}{\d s}\int_{\Sigma_{L}}\gamma_{1}(x)\,\d x\cdot \int_{\Sigma_{R}}\gamma_{1}(y)\,\d y-\kappa \int_{\Sigma}\gamma_{1}(x)\,\d x\\
&=\gamma_{1}(c(s))\gamma_{1}(d(s))(-c(s)c'(s)-d(s)d'(s))-\kappa(\gamma_{1}(c(s))(-c(s)c'(s))+\gamma_{1}(d(s))(-d(s)d'(s)))\\
&\stackrel{\eqref{gdifeq}}{=}\gamma_{1}(c(s))c'(s)\Big(-\gamma_{1}(d(s))c(s)-\gamma_{1}(c(s))d(s)+\kappa(c(s)+d(s))\Big).
\end{flalign*}
So, a critical point occurs when $c'(s)=d'(s)=0$, or when
$$-c(s)(\kappa-\gamma_{1}(d(s)))=d(s)(\kappa-\gamma_{1}(c(s))).$$
That is,
\begin{equation}\label{critpteq}
\frac{\kappa-\gamma_{1}(c(s))}{-c(s)}=\frac{\kappa-\gamma_{1}(d(s))}{d(s)}.
\end{equation}

Assume for now that $a\geq1/2$, let $\beta>0$ such that $\gamma_{1}[-\beta,\beta]=a$.  Assume also that $c(0)=-\beta$ and $d(0)=\beta$.  Let $\kappa\in\R$ such that $0<\kappa<\gamma_{1}(\beta)$.  In the case $a\geq1/2$, the unique solution of this equation occurs when $-c(s)=d(s)=\beta$, since $c(s)\leq0$ always, and $c\mapsto\frac{\kappa-\gamma_{1}(c)}{-c}$ is decreasing and negative (when $s\geq0$)), whereas $d\mapsto\frac{\kappa-\gamma_{1}(d)}{d}$ is increasing and then positive (when $s\geq0$).

We now consider the case $a<1/2$.  In this case, again there is a solution $-c(s)=d(s)=\beta$ to the critical point condition \eqref{critpteq}, and then $c\mapsto\frac{\kappa-\gamma_{1}(c)}{-c}$ is decreasing and negative (when $c(s)<0$), whereas $d\mapsto\frac{\kappa-\gamma_{1}(d)}{d}$ is increasing, as long as $c(s)\leq0$.  When $c(s)>0$, $c\mapsto\frac{\kappa-\gamma_{1}(c)}{-c}$ is increasing and negative, whereas $d\mapsto\frac{\kappa-\gamma_{1}(d)}{d}$ is increasing and positive.  By choice of $\kappa$, for all $s\geq0$ in the domain of $c$, we have $c(s)\leq \inf_{t\geq} d(t)$.  So, as before, there is only one solution to the equation \eqref{critpteq}, corresponding to $-c(s)=d(s)=\beta$.
%

The function $f(t)\colonequals (\kappa- e^{-t^{2}/2})/t$ satisfies $f'(t)=0$ when $e^{-t/2}(t^{2}+1)=\kappa$, so $f(t)=\kappa t/(t^{2}+1)\leq\kappa/2$ at this critical point.  Moreover, if $\kappa\colonequals\min(1/\abs{\alpha},1) \frac{1}{2}e^{-\alpha^{2}/2}$, then $f(\alpha)<-(1/2)e^{-\alpha^{2}/2}/\abs{\alpha} <-\kappa$.

In conclusion, the only two possible critical points of the minimization problem \eqref{minprob} are the symmetric interval $[-\beta,\beta]$ (or its complement) and the half space $[-\infty,-\alpha)$ (or $(\alpha,\infty]$).  By choice of $\kappa$, the minimum occurs at the half space.  We conclude that, for any set $\Omega$ with $\gamma_{1}(\Omega)=a$, we have
$$\int_{\Sigma_{L}}\gamma_{1}(x)\,\d x\cdot \int_{\Sigma_{R}}\gamma_{1}(y)\,\d y-\kappa\int_{\Sigma}\gamma_{1}(x)\,\d x
\geq -\kappa\int_{\partial H}\gamma_{1}(x)\,\d x.
$$
That is,
$$\int_{\Sigma_{L}}\gamma_{1}(x)\,\d x\cdot \int_{\Sigma_{R}}\gamma_{1}(y)\,\d y
\geq \kappa \Big(\int_{\Sigma}\gamma_{1}(x)\,\d x-\int_{\partial H}\gamma_{1}(x)\,\d x\Big),
$$
Plugging this into \eqref{stackeq} completes the proof for the $k=1$ case of the lemma, where $\kappa\colonequals\min(a,1-a)/2$ suffices.

The case $k=2$ of the lemma is simpler.  By identifying the left endpoints of an interval in $\Omega$ with its right endpoint (if it exists) we have
$$\abs{\int_{\Sigma_{L}}\gamma_{1}(x)\,\d x-\int_{\Sigma_{R}}\gamma_{1}(y)\,\d y}\leq\frac{2}{\sqrt{2\pi}}<.8.$$
And $\int_{\Sigma_{L}}\gamma_{1}(x)\,\d x+\int_{\Sigma_{R}}\gamma_{1}(y)\,\d y=\int_{\Sigma}\gamma_{1}(x)\,\d x$, so
\begin{equation}\label{lreq}
\begin{aligned}
\min\Big(\int_{\Sigma_{L}}\gamma_{1}(x)\,\d x,\int_{\Sigma_{R}}\gamma_{1}(x)\,\d x\Big)
&\geq\frac{1}{2}\Big(\int_{\Sigma_{L}}\gamma_{1}(x)\,\d x+\int_{\Sigma_{R}}\gamma_{1}(x)\,\d x -.8\Big)\\
&=\frac{1}{2}\Big(\int_{\Sigma}\gamma_{1}(x)\,\d x-.8\Big).
\end{aligned}
\end{equation}
If $\int_{\Sigma}\gamma_{1}(x)\,\d x-\int_{\partial H}\gamma_{1}(x)\,\d x>1$, then \eqref{stackeq} and \eqref{lreq} imply that
\begin{flalign*}
&\int_{\R}\Big(\Big[\int_{\frac{\Sigma-\sqrt{\rho}x}{1-\rho}}\gamma_{1}(y)\,\d y\Big]^{2}-\vnormf{\int_{\frac{\Omega-\sqrt{\rho}x}{1-\rho}}y\gamma_{1}(y)}^{2}\Big)\gamma_{1}(x)\,\d x\\
&\qquad\geq\frac{4\sqrt{2}(1-\rho)}{\pi\sqrt{\rho^{2}+1}}\frac{1}{4}\Big(\int_{\Sigma}\gamma_{1}(x)\,\d x-.8\Big)^{2}\\
&\qquad\geq\frac{4\sqrt{2}(1-\rho)}{\pi\sqrt{\rho^{2}+1}}\frac{1}{100}\Big(\int_{\Sigma}\gamma_{1}(x)\,\d x-\int_{\partial H}\gamma_{1}(x)\,\d x\Big)^{2}.
\end{flalign*}
The case $k=2$ of the lemma follows.

\end{proof}

\begin{lemma}\label{finallem9}
Let $\Omega\subset\R$ satisfy $\int_{\Omega}x\gamma_{1}(x)\,\d x\geq0$.  Let $H\subset\R$ be a half space such that $\gamma_{1}(H)=\gamma_{1}(\Omega)$ and $\int_{H}x\gamma_{1}(x)\,\d x>0$.  Let $0<a\leq1/2$, let $\alpha\in\R$ satisfy $\int_{\alpha}^{\infty}\gamma_{1}(t)\,\d t=a$.  Then
\begin{flalign*}
&\int_{H}(x-\alpha\beta) e^{-\frac{[\beta x-\alpha]^{2}}{2(1-\beta^{2})}}\gamma_{1}(x)\,\d x
-\int_{\Omega}(x-\alpha\beta) e^{-\frac{[\beta x-\alpha]^{2}}{2(1-\beta^{2})}}\gamma_{1}(x)\,\d x\\
&\qquad\qquad\qquad\leq8e^{\frac{\alpha^{2}}{1+\beta}}\frac{(6+\abs{\alpha})^{2}}{\beta(1-\beta^{2})}\Big(\int_{\partial\Omega}\gamma_{1}(x)\,\d x
-\int_{\partial H}\gamma_{1}(x)\,\d x\Big).
\end{flalign*}
Moreover, the left and right sides of the inequality are both nonnegative.
\end{lemma}
\begin{proof}
Choose $\epsilon>0$ such that
\begin{equation}\label{epsbound}
\epsilon\leq\frac{1}{8\Big(\frac{\beta^{2}}{(1-\beta^{2})^{3}}+3(2\abs{\alpha}+2+\abs{\alpha}\beta)^{2}(1/2+\frac{\beta}{1-\beta^{2}})\Big)}\leq\frac{(1-\beta^{2})^{3}}{8(6+\abs{\alpha})^{2}}.
\end{equation}
Suppose we minimize over all $\Omega\subset\R$ the quantity
\begin{equation}\label{qtomin}
\int_{\partial\Omega}\gamma_{1}(x)\,\d x+\epsilon\frac{\beta}{2(1-\beta^{2})^{3}}\Big(\int_{\Omega}(x-\alpha\beta) e^{-\frac{[\beta x-\alpha]^{2}}{2(1-\beta^{2})}}\gamma_{1}(x)\,\d x\Big)^{2}
+2(1+\abs{\alpha})\abs{\gamma_{1}(\Omega)-a}.
\end{equation}

 The first variation condition \cite[Lemma 3.2]{heilman18} says $\exists$ $c\in\R$ such that
\begin{equation}\label{six1}
- xN(x)+\epsilon\frac{\beta^{2}}{(1-\beta^{2})^{3}}\Big(\int_{\Omega}(y-\alpha\beta) e^{-\frac{[\beta y-\alpha]^{2}}{2(1-\beta^{2})}}\gamma_{1}(y)\,\d y\Big)(x-\alpha\beta)e^{-\frac{[\beta x-\alpha]^{2}}{2(1-\beta^{2})}}=c,\quad\forall\,x\in\Sigma\colonequals\partial\Omega,
\end{equation}
and $\abs{c}\leq2(1+\abs{\alpha})$, which follows by repeating the argument of Lemma \ref{firstvarmaxns}:  For a general vector field $X$ supported in $\Sigma$, its first variation (without the Gaussian volume term) is $\int_{\Sigma}cf(x)\,\d x$, and
$$
\frac{\d}{\d s}\Big|_{s=0}\gamma_{\adimn}(\Omega^{(s)})=\int_{\Sigma}f(x)\gamma_{\adimn}(x)\,\d x.
$$
So, $\absf{\frac{1}{2}\frac{\d}{\d s}|_{s=0}\gamma_{\adimn}(\Omega^{(s)})}$ equals $0$ or $\absf{\int_{\Sigma}f(x)\gamma_{\adimn}(x)\,\d x}$.  In either case, it follows that $\abs{c}\leq2(1+\abs{\alpha})$.

The second variation \cite[Lemma 3.7]{heilman18} \cite[Theorem 4.1]{colding12a} says: $\forall$ $f\colon\partial\Omega\to\R$ with $\int_{\partial\Omega}f(x)\gamma_{1}(x)\,\d x=0$,
\begin{flalign*}
&\int_{\Sigma}\Big[-(f(x))^{2}+\epsilon\frac{\beta^{2}}{(1-\beta^{2})^{3}}\Big(\int_{\Omega}(y-\alpha\beta) e^{-\frac{[\beta y-\alpha]^{2}}{2(1-\beta^{2})}}\gamma_{1}(y)\,\d y\Big) N(x)(f(x))^{2}e^{-\frac{[\beta x-\alpha]^{2}}{2(1-\beta^{2})}}\Big]\gamma_{1}(x)\,\d x\\
&\qquad\qquad\qquad+\epsilon\frac{\beta^{2}}{(1-\beta^{2})^{3}}\abs{\int_{\Sigma}(x-\alpha\beta)f(x) e^{-\frac{[\beta x-\alpha]^{2}}{2(1-\beta^{2})}}\gamma_{1}(x)\,\d x}^{2}.
\end{flalign*}
Multiplying \eqref{six1} by $x N$ and integrating, then using the AMGM inequality and \eqref{epsbound},
\begin{flalign*}
\int_{\Sigma}x^{2}\gamma_{1}(x)\,\d x -\frac{1}{8}\int_{\Sigma}x^{2}e^{-\frac{[\beta x-\alpha]^{2}}{2(1-\beta^{2})}} \gamma_{1}(x)\,\d x
&\leq(\abs{c}+\alpha\beta)\int_{\Sigma}\abs{x}\gamma_{1}(x)\,\d x\\
&\leq (\abs{c}+\alpha\beta)^{2}\int_{\Sigma}\gamma_{1}(x)\,\d x+\frac{1}{2}\int_{\Sigma}x^{2}\gamma_{1}(x)\,\d x.
\end{flalign*}
That is,
\begin{equation}\label{six2}
\int_{\Sigma}x^{2}\gamma_{1}(x)\,\d x\leq 3(\abs{c}+\alpha\beta)^{2}\int_{\Sigma}\gamma_{1}(x)\,\d x.
\end{equation}

So, choose $f$ such that $f$ is supported on two distinct points and $\int_{\Sigma}f(x)\gamma_{1}(x)\,\d x=0$.  Then
\begin{flalign*}
&\int_{\Sigma}\Big[-(f(x))^{2}+\epsilon\frac{\beta^{2}}{(1-\beta^{2})^{3}}\Big(\int_{\Omega}(y-\alpha\beta) e^{-\frac{[\beta y-\alpha]^{2}}{2(1-\beta^{2})}}\gamma_{1}(y)\,\d y\Big) N(x)(f(x))^{2}e^{-\frac{[\beta x-\alpha]^{2}}{2(1-\beta^{2})}}\Big]\gamma_{1}(x)\,\d x\\
&\qquad\qquad\qquad+\epsilon\frac{\beta^{2}}{(1-\beta^{2})^{3}}\abs{\int_{\Sigma}(x-\alpha\beta)f(x) e^{-\frac{[\beta x-\alpha]^{2}}{2(1-\beta^{2})}}\gamma_{1}(x)\,\d x}^{2}\\
&\qquad\stackrel{\eqref{six2}}{\leq}
\int_{\Sigma}\Big[-(f(x))^{2}\\
&\qquad\qquad\qquad+\epsilon\frac{\beta^{2}}{(1-\beta^{2})^{3}}\Big(\int_{\Omega}(y-\alpha\beta) e^{-\frac{[\beta y-\alpha]^{2}}{2(1-\beta^{2})}}\gamma_{1}(y)\,\d y\Big) N(x)(f(x))^{2}e^{-\frac{[\beta x-\alpha]^{2}}{2(1-\beta^{2})}}\Big]\gamma_{1}(x)\,\d x\\
&\qquad\qquad\qquad+\epsilon 3(\abs{c}+\alpha\beta)^{2}(\int_{\Sigma}(f(x))^{2}\gamma_{1}(x)\,\d x\Big)\Big(\int_{\Sigma}\gamma_{1}(y)\,\d y\Big)\\
&\qquad\stackrel{\eqref{epsbound}}{\leq}
\int_{\Sigma}(f(x))^{2}\Big(-1+\frac{1}{8}+\epsilon 3(\abs{c}+\alpha\beta)^{2}\int_{\Sigma}\gamma_{1}(y)\,\d y\Big)\gamma_{1}(x)\,\d x
\stackrel{\eqref{epsbound}}{<}0.
\end{flalign*}
The last inequality used $\int_{\Sigma}\gamma_{1}(y)\,\d y\leq1/2+\frac{\beta}{1-\beta^{2}}$.  This follows by the minimizing property of $\Omega$ and since a half space $\Omega'=[\alpha,\infty)$ with $\gamma_{1}(\Omega')=a$ satisfies
\begin{flalign*}
&\int_{\partial\Omega}\gamma_{1}(x)\,\d x+\epsilon\frac{\beta}{(1-\beta^{2})^{3}}\Big(\int_{\Omega}(x-\alpha\beta) e^{-\frac{[\beta x-\alpha]^{2}}{2(1-\beta^{2})}}\gamma_{1}(x)\,\d x\Big)^{2}
+2(1+\abs{\alpha})\abs{\gamma_{1}(\Omega)-a}\\
&\leq\int_{\partial\Omega'}\gamma_{1}(x)\,\d x+\epsilon\frac{\beta}{(1-\beta^{2})^{3}}\Big(\int_{\Omega'}(x-\alpha\beta) e^{-\frac{[\beta x-\alpha]^{2}}{2(1-\beta^{2})}}\gamma_{1}(x)\,\d x\Big)^{2}
+2(1+\abs{\alpha})\abs{\gamma_{1}(\Omega')-a}\\
&= e^{-\alpha^{2}/2}/\sqrt{2\pi}+\frac{\epsilon\beta}{1-\beta^{2}}e^{-\alpha^{2}/2}e^{-\frac{[\alpha-\alpha\beta]^{2}}{2(1-\beta^{2})}} <1/2+\frac{\beta}{1-\beta^{2}}.
\end{flalign*}

We have then arrived at a contradiction, since the second variation must be nonnegative.  We conclude that no such $f$ exists, i.e. $\Omega$ itself must be a half space.  Then Lemma \ref{finallem3} implies that $\Omega$ has Gaussian measure $a$.  Since a half-space minimizes the quantity \eqref{qtomin} over sets of Gaussian measure $a$, we have: for any $\Omega\subset\R$ with $\gamma_{1}(\Omega)=a$,
\begin{flalign*}
&\int_{\partial H}\gamma_{1}(x)\,\d x+\epsilon\frac{\beta}{(1-\beta^{2})^{3}}\Big(\int_{ H}(x-\alpha\beta) e^{-\frac{[\beta x-\alpha]^{2}}{2(1-\beta^{2})}}\gamma_{1}(x)\,\d x\Big)^{2}\\
&\leq\int_{\partial\Omega}\gamma_{1}(x)\,\d x+\epsilon\frac{\beta}{(1-\beta^{2})^{3}}\Big(\int_{\Omega}(x-\alpha\beta) e^{-\frac{[\beta x-\alpha]^{2}}{2(1-\beta^{2})}}\gamma_{1}(x)\,\d x\Big)^{2}.
\end{flalign*}
Rearranging this inequality completes the proof
\begin{flalign*}
&\epsilon\Big[\frac{\beta}{(1-\beta^{2})^{3}}\Big(\int_{ H}(x-\alpha\beta) e^{-\frac{[\beta x-\alpha]^{2}}{2(1-\beta^{2})}}\gamma_{1}(x)\,\d x\Big)^{2}
-\frac{\beta}{(1-\beta^{2})^{3}}\Big(\int_{\Omega}(x-\alpha\beta) e^{-\frac{[\beta x-\alpha]^{2}}{2(1-\beta^{2})}}\gamma_{1}(x)\,\d x\Big)^{2}\Big]\\
&\qquad\qquad\leq\int_{\partial\Omega}\gamma_{1}(x)\,\d x-\int_{\partial H}\gamma_{1}(x)\,\d x.
\end{flalign*}
To conclude, we use again the equality
$$
\frac{1}{1-\beta^{2}}\int_{H}(x-\alpha\beta) e^{-\frac{[\beta x-\alpha]^{2}}{2(1-\beta^{2})}}\gamma_{1}(x)\,\d x
=e^{-\alpha^{2}/2}e^{-\frac{[\alpha-\alpha\beta]^{2}}{2(1-\beta^{2})}}
=e^{-\frac{\alpha^{2}}{1+\beta}}.
$$
\end{proof}

\begin{lemma}\label{finallem3}
Let $a\in\R$.  Let $0<\rho<1$ and let $0\leq \epsilon\leq \frac{(1-\beta^{2})^{3}}{8(6+\abs{\alpha})^{2}}$.  For any $t\in\R$, define
$$
h(t)
\colonequals \frac{e^{-t^{2}/2}}{\sqrt{2\pi}}
+\epsilon\frac{\beta}{(1-\beta^{2})^{3/2}}\int_{t}^{\infty}(x-\alpha\beta)e^{-\frac{\abs{\beta x-\alpha}^{2}}{2(1-\beta^{2})}}\gamma_{1}(x)\,\d x
+2(1+\abs{\alpha})\abs{\gamma_{1}[t,\infty)-a}.
$$
Then the unique minimum of $h(t)$ occurs when $t\in\R$ satisfies $\gamma_{1}[t,\infty)=a$ (i.e. $t=\alpha$).
\end{lemma}
\begin{proof}
For any $t\in\R$,
\begin{flalign*}
h'(t)
&=-t\gamma_{1}(t)
- \epsilon\frac{\beta}{(1-\beta^{2})^{3/2}}(t-\alpha\beta)e^{-\frac{\abs{\beta x-\alpha}^{2}}{2(1-\beta^{2})}}\gamma_{1}(t)
-2(1+\abs{\alpha})\gamma_{1}(t)\cdot\mathrm{sign}(\gamma_{1}[t,\infty)-a)\\
&=\gamma_{1}(t)\Big[-t
- \epsilon\frac{\beta}{(1-\beta^{2})^{3/2}}(t-\alpha\beta)e^{-\frac{\abs{\beta x-\alpha}^{2}}{2(1-\beta^{2})}}
-2(1+\abs{\alpha})\cdot\mathrm{sign}(\gamma_{1}[t,\infty)-a)\Big].
\end{flalign*}
The derivative of the term inside the square brackets is negative (since $\epsilon<\frac{(1-\beta^{2})^{3}}{8(6+\abs{\alpha})^{2}}$), and $h'(t)>0$ for sufficiently negative $t$.  If $t<\alpha$, then $\mathrm{sign}(\gamma_{1}[t,\infty)-a)=1$, so $\lim_{t\to\alpha^{-}}h'(t)=h'(\alpha^{-})< 0$.  So, $h$ is increasing and then decreasing on $(-\infty,\alpha]$.  If $t>\alpha$, then $\mathrm{sign}(\gamma_{1}[t,\infty)-a)=-1$, so $\lim_{t\to\alpha^{+}}h'(t)=h'(\alpha^{+})>0$.  And $h'(t)<0$ for sufficiently positive $t$.  In summary, $h$ increases, then decreases then increases and then decreases.  So, the minimum value of $h$ is either $h(-\infty),h(\infty)$ or $h(\alpha)$.  In the case $a\geq1/2$, we have
$$\lim_{t\to\infty}h(t)=2(1+\abs{\alpha})a\geq (1+\abs{\alpha})
>e^{-\alpha^{2}/2}/\sqrt{2\pi}+\epsilon\frac{\beta}{\sqrt{1-\beta^{2}}}e^{-\frac{\alpha^{2}}{1+\beta}}=h(\alpha).$$
$$\lim_{t\to-\infty}h(t)=2(1+\abs{\alpha})(1-a)=2(1+\abs{\alpha})\int_{-\infty}^{\alpha}\gamma_{1}(w)\,\d w> \frac{e^{-\alpha^{2}/2}}{\sqrt{2\pi}}+\epsilon\frac{\beta}{\sqrt{1-\beta^{2}}}e^{-\frac{\alpha^{2}}{1+\beta}}=h(\alpha).$$
(The last inequality used $0<\beta<1$.)  Therefore, $h$ is uniquely minimized at $t=\alpha$.  In the remaining case $a\leq1/2$, we have
$$\lim_{t\to-\infty}h(t)=2(1+\abs{\alpha})(1-a)\geq(1+\abs{\alpha})
>\frac{e^{-\alpha^{2}/2}}{\sqrt{2\pi}}+\epsilon\frac{\beta}{\sqrt{1-\beta^{2}}}e^{-\frac{\alpha^{2}}{1+\beta}}=h(\alpha).$$
$$\lim_{t\to\infty}h(t)=2(1+\abs{\alpha})a=2(1+\abs{\alpha})\int_{\alpha}^{\infty}\gamma_{1}(w)\,\d w
>e^{-\alpha^{2}/2}/\sqrt{2\pi}+\epsilon\frac{\beta}{\sqrt{1-\beta^{2}}}e^{-\frac{\alpha^{2}}{1+\beta}}=h(\alpha).$$
Therefore, $h$ is uniquely minimized at $t=\alpha$.

%

\end{proof}

\begin{lemma}\label{finallem7}
Let $\Omega\subset\R$ satisfy $\int_{\Omega}x\gamma_{1}(x)\,\d x\geq0$.  Let $H\subset\R$ be a half space such that $\gamma_{1}(H)=\gamma_{1}(\Omega)$ and $\int_{H}x\gamma_{1}(x)\,\d x>0$.  Then
\begin{flalign*}
&\abs{\int_{\partial H} e^{-\frac{[\beta x-\alpha]^{2}}{2(1-\beta^{2})}}\gamma_{1}(x)\,\d x
-\int_{\partial\Omega} e^{-\frac{[\beta x-\alpha]^{2}}{2(1-\beta^{2})}}\gamma_{1}(x)\,\d x}\\
&\qquad\qquad\qquad\leq8\frac{(6+\abs{\alpha})^{2}}{\beta(1-\beta^{2})}\Big(\int_{\partial\Omega}\gamma_{1}(x)\,\d x
-\int_{\partial H}\gamma_{1}(x)\,\d x\Big).
\end{flalign*}
\end{lemma}
\begin{proof}
Let $0<\epsilon'<1$.  Suppose we minimize over all $\Omega\subset\R$ with $\gamma_{1}(\Omega)=a$ the quantity
$$\int_{\partial\Omega}\gamma_{1}(x)\,\d x\pm\epsilon'\int_{\partial \Omega}e^{-\frac{[\beta x-\alpha]^{2}}{2(1-\beta^{2})}}\gamma_{1}(x)\,\d x
=\int_{\partial\Omega}(1\pm\epsilon'e^{-\frac{[\beta x-\alpha]^{2}}{2(1-\beta^{2})}})\gamma_{1}(x)\,\d x.$$
(The minimum exists since $\epsilon'<1$.)  The first variation condition \cite[Lemma 3.2]{heilman18} says $\exists$ $c\in\R$ such that
\begin{equation}\label{six1c}
- xN(x)\mp\epsilon'\frac{x-\alpha\beta}{1-\beta^{2}}e^{-\frac{[\beta x-\alpha]^{2}}{2(1-\beta^{2})}} N(x)=c,\qquad\forall\,x\in\Sigma\colonequals\partial\Omega.
\end{equation}
%
The second variation \cite[Lemma 3.7]{heilman18} \cite[Theorem 4.1]{colding12a} says: $\forall$ $f\colon\partial\Omega\to\R$ with $\int_{\partial\Omega}f(x)\gamma_{1}(x)\,\d x=0$,
$$
\int_{\partial\Omega}\Big(-1\mp\epsilon'\frac{1+x(x-\alpha\beta)}{1-\beta^{2}}e^{-\frac{[\beta x-\alpha]^{2}}{2(1-\beta^{2})}}\Big)(f(x))^{2}\gamma_{1}(x)\,\d x\geq0.
$$
Choosing $0<\epsilon'<(1-\beta^{2})(6+\abs{\alpha})^{-2}$ completes the proof, since if $\Omega$ is not a half space, we can select an $f$ supported on two points of $\partial\Omega$ with $\int_{\partial\Omega}f(x)\gamma_{1}(x)\,\d x=0$ that has negative second variation, a contradiction.  We conclude that
$$
\int_{\partial H}\gamma_{1}(x)\,\d x\pm\epsilon \int_{\partial H}e^{-\frac{[\beta x-\alpha]^{2}}{2(1-\beta^{2})}} \gamma_{1}(x)\,\d x
\leq \int_{\partial \Omega}\gamma_{1}(x)\,\d x\pm\epsilon \int_{\partial \Omega}e^{-\frac{[\beta x-\alpha]^{2}}{2(1-\beta^{2})}} \gamma_{1}(x)\,\d x.
$$
Rearranged, this gives the desired result:
$$
\pm\epsilon \Big(\int_{\partial H}e^{-\frac{[\beta x-\alpha]^{2}}{2(1-\beta^{2})}} \gamma_{1}(x)\,\d x-\int_{\partial \Omega}e^{-\frac{[\beta x-\alpha]^{2}}{2(1-\beta^{2})}} \gamma_{1}(x)\,\d x\Big)
\leq \int_{\partial \Omega}\gamma_{1}(x)\,\d x-\int_{\partial H}\gamma_{1}(x)\,\d x.
$$
\end{proof}

\begin{lemma}\label{finallem2}
Let $0<a<1$.  Let $0<\rho<1$ and let $0\leq \epsilon\leq e^{-\frac{\alpha^{2}}{1+\beta}}\frac{\sqrt{1-\beta^{2}}}{8\beta (6+\abs{\alpha})^{2}}$.  $\forall$ $t\in\R$, define
$$
h(t)
\colonequals e^{-t^{2}/2}/\sqrt{2\pi}
+\epsilon\Big[\int_{t}^{\infty}\Phi\Big(\frac{\beta x-\alpha}{\sqrt{1-\beta^{2}}}\Big)\gamma_{1}(x)\,\d x\Big]^{2}
+2(1+\abs{\alpha})\abs{\gamma_{1}[t,\infty)-a}.
$$
Then the unique minimum of $h(t)$ occurs when $\gamma_{1}[t,\infty)=a$, i.e. when $t=\alpha\colonequals-\Phi^{-1}(a)$.
\end{lemma}
\begin{proof}
\begin{flalign*}
h'(t)
&=-t\gamma_{1}(t)
- 2\epsilon\Phi\Big(\frac{\beta t-\alpha}{\sqrt{1-\beta^{2}}}\Big)\gamma_{1}(t)\int_{t}^{\infty}\Phi\Big(\frac{\beta x-\alpha}{\sqrt{1-\beta^{2}}}\Big)\gamma_{1}(x)\,\d x\\
&\qquad\qquad\qquad
-2(1+\abs{\alpha})\gamma_{1}(t)\cdot\mathrm{sign}(\gamma_{1}[t,\infty)-a)\\
&=\gamma_{1}(t)\Big[-t
- 2\epsilon\Phi\Big(\frac{\beta t-\alpha}{\sqrt{1-\beta^{2}}}\Big)\int_{t}^{\infty}\Phi\Big(\frac{\beta x-\alpha}{\sqrt{1-\beta^{2}}}\Big)\gamma_{1}(x)\,\d x\\
&\qquad\qquad\qquad+2(1+\abs{\alpha})\cdot\mathrm{sign}(\gamma_{1}[t,\infty)-a)\Big].
\end{flalign*}%
The derivative of the term inside the square brackets is negative (since $\epsilon<\frac{\sqrt{1-\beta^{2}}}{\beta8(6+\abs{\alpha})^{2}}$), and $h'(t)>0$ for sufficiently negative $t$.  If $t<\alpha$, then $\mathrm{sign}(\gamma_{1}[t,\infty)-a)=1$, so $\lim_{t\to\alpha^{-}}h'(t)=h'(\alpha^{-})< 0$.  So, $h$ is increasing and then decreasing on $(-\infty,\alpha]$.  If $t>\alpha$, then $\mathrm{sign}(\gamma_{1}[t,\infty)-a)=-1$, so $\lim_{t\to\alpha^{+}}h'(t)=h'(\alpha^{+})>0$.  And $h'(t)<0$ for sufficiently positive $t$.  In summary, $h$ increases, then decreases then increases and then decreases.  So, the minimum value of $h$ is either $h(-\infty),h(\infty)$ or $h(\alpha)$.  Suppose for now that $a\leq 1/2$ that $\alpha\leq0$.  Then
$$
\lim_{t\to-\infty}h(t)
=2(1+\abs{\alpha})(1-a)\geq (1+\abs{\alpha})
>\frac{e^{-\alpha^{2}/2}}{\sqrt{2\pi}}
+\epsilon\Big[\int_{\alpha}^{\infty}\Phi\Big(\frac{\beta x-\alpha}{\sqrt{1-\beta^{2}}}\Big)\gamma_{1}(x)\,\d x\Big]^{2}=h(\alpha).
$$
\begin{flalign*}
\lim_{t\to\infty}h(t)
&=2(1+\abs{\alpha})a
=2(1+\abs{\alpha})\int_{\alpha}^{\infty}\gamma_{1}(w)\,\d w\\
&>e^{-\alpha^{2}/2}/\sqrt{2\pi}
+\epsilon\Big[\int_{\alpha}^{\infty}\Phi\Big(\frac{\beta x-\alpha}{\sqrt{1-\beta^{2}}}\Big)\gamma_{1}(x)\,\d x\Big]^{2}=h(\alpha).
\end{flalign*}
The last inequality used
$$
\int_{\alpha}^{\infty}\Phi\Big(\frac{\beta x-\alpha}{\sqrt{1-\beta^{2}}}\Big)\gamma_{1}(x)\,\d x\leq \int_{\alpha}^{\infty}\gamma_{1}(x)\,\d x.$$
Therefore, $h$ is uniquely minimized at $t=\alpha$.  In the remaining case that $a> 1/2$, we have $\alpha\leq0$ and
\begin{flalign*}
\lim_{t\to-\infty}h(t)
&=2(1+\abs{\alpha})(1-a)=2(1+\abs{\alpha})\int_{-\infty}^{\alpha}\gamma_{1}(w)\,\d w\\
&>e^{-\alpha^{2}/2}/\sqrt{2\pi}
+\epsilon\Big[\int_{\alpha}^{\infty}\Phi\Big(\frac{\beta x-\alpha}{\sqrt{1-\beta^{2}}}\Big)\gamma_{1}(x)\,\d x\Big]^{2}=h(\alpha).
\end{flalign*}
The last inequality used the definition of $\epsilon$.  Also,
$$\lim_{t\to\infty}h(t)=2(1+\abs{\alpha})a\geq (1+\abs{\alpha})
>e^{-\alpha^{2}/2}/\sqrt{2\pi}
+\epsilon\Big[\int_{\alpha}^{\infty}\Phi\Big(\frac{\beta x-\alpha}{\sqrt{1-\beta^{2}}}\Big)\gamma_{1}(x)\,\d x\Big]^{2}=h(\alpha).$$
Therefore, $h$ is uniquely minimized at $t=\alpha$.
%
%

\end{proof}

\medskip
\noindent\textbf{Acknowledgement}.  Thanks to Alex Tarter for helpful discussions.

\bibliographystyle{amsalpha}
\bibliography{12162011}

\end{document}